\documentclass{article}
\usepackage{graphicx} 

\usepackage[english]{babel}
\usepackage[T1]{fontenc}

\usepackage{tikz}
\tikzstyle{vertex}=[circle, draw, inner sep=2pt, minimum size=6pt]

\usepackage{mathrsfs}

\usepackage[a4paper,top=3cm,bottom=2cm,left=3cm,right=3cm,marginparwidth=1.75cm]{geometry}

\usepackage{comment}
\usepackage{amsfonts}
\usepackage{fullpage}
\usepackage{latexsym}
\usepackage{amsfonts}
\usepackage{blkarray}
\usepackage{graphicx}
\usepackage{float}
\usepackage{enumerate}
\usepackage{amsthm,amsmath,amssymb}
\usepackage{verbatim}
\usepackage{enumitem}
\usepackage{blkarray}
\usepackage{mathtools}
\usepackage{algorithm}
\usepackage[noend]{algpseudocode}

\usepackage{mathtools}
\usepackage{mathrsfs}
\usepackage{latexsym}
\usepackage{xcolor}

\usepackage{caption}
\usepackage{subcaption}

\usepackage{rotating}
\usepackage{url}
\usepackage{hyperref}

\usepackage{multirow}
\usepackage{hhline}

\providecommand{\keywords}[1]{
  \small	
  \textbf{\textit{Keywords---}} #1
}

\def\T{^\top}
\newtheorem{theorem}{Theorem}[section]
\newtheorem{lemma}{Lemma}[section]

\newtheorem{proposition}{Proposition}[section]
\newtheorem{remark}{Remark}

\newtheorem{example}{Example}[section]

\newcommand{\nc}{\newcommand}

\nc{\cA}{{\cal A}}
\nc{\cB}{{\cal B}}
\nc{\cC}{{\cal C}}
\nc{\cD}{{\cal D}}
\nc{\cE}{{\cal E}}
\nc{\cG}{{\cal G}}
\nc{\cF}{{\cal F}}
\nc{\cH}{{\cal H}}
\nc{\cI}{{\cal I}}
\nc{\cK}{{\cal K}}
\nc{\cL}{{\cal L}}
\nc{\cM}{{\cal M}}
\nc{\cN}{{\cal N}}
\nc{\cO}{{\cal O}}
\nc{\cP}{{\cal P}}
\nc{\cQ}{{\cal Q}}
\nc{\cR}{{\cal R}}
\nc{\cS}{{\cal S}}
\nc{\cT}{{\cal T}}
\nc{\cV}{{\cal V}}
\nc{\tx}{{\tilde x}}
\nc{\la}{{\langle}}
\nc{\ra}{{\rangle}}
\nc{\ts}{\textsuperscript}
\nc{\cp}[1]{\mathcal{CP}_{#1}} 
\def\A{\mathbb{A}}
\def\B{\mathbb{B}}
\def\R{\mathbb{R}}

\def\diag{\mbox{diag }}

\def\T{^\top}
\def\du#1#2{\langle {#1} , {#2} \rangle}

\def\Ab{{\mathsf A}}
\def\Bb{{\mathsf B}}
\def\Cb{{\mathsf C}}

\def\Eb{{\mathsf E}}
\def\Fb{{\mathsf F}}

\def\Hb{{\mathsf H}}
\def\Ib{{\mathsf I}}

\def\Mb{{\mathsf M}}
\def\Nb{{\mathsf N}}
\def\Ob{{\mathsf O}}
\def\Pb{{\mathsf P}}
\def\Qb{{\mathsf Q}}
\def\Rb{{\mathsf R}}
\def\Sb{{\mathsf S}}
\def\Tb{{\mathsf T}}
\def\Ub{{\mathsf U}}

\def\Wb{{\mathsf W}}
\def\Xb{{\mathsf X}}
\def\Yb{{\mathsf Y}}
\def\Zb{{\mathsf Z}}

\def\x{\mathbf x}

\def\y{\mathbf y}

\def\v{\mathbf v}

\def\e{\mathbf e}
\def\u{\mathbf u}
\def\b{\mathbf b}
\def\c{\mathbf c}
\def\d{\mathbf d}

\def\a{{\mathbf a}}

\def\x{{\mathbf x}}
\def\y{{\mathbf y}}

\def\u{{\mathbf u}}
\def\v{{\mathbf v}}

\def\o{{\mathbf o}}

\date{\today}

\title{Tighter yet more tractable relaxations\\ 
 and nontrivial instance generation\\ for sparse  standard quadratic optimization}

\author{Immanuel Bomze\thanks{VCOR and Research Network Data Science, University of Vienna, Oskar-Morgenstern-Platz 1,
1090 Wien, Austria. ORCID ID: 0000-0002-6288-9226 E-mail: {\tt immanuel.bomze@univie.ac.at}} \and Bo Peng\thanks{VGSCO, University of Vienna, Oskar-Morgenstern-Platz 1,
1090 Wien, Austria. ORCID ID: 0000-0002-2650-0295 E-mail: {\tt bo.peng@univie.ac.at}} \and Yuzhou Qiu\thanks{School of Mathematics, Peter Guthrie Tait Road, The University of Edinburgh, Edinburgh, EH9 3FD, United Kingdom. E-mail: \tt{y.qiu-16@sms.ed.ac.uk}} \and E. Alper Y{\i}ld{\i}r{\i}m\thanks{School of Mathematics, Peter Guthrie Tait Road, The University of Edinburgh, Edinburgh, EH9 3FD, United Kingdom. ORCID ID: 0000-0003-4141-3189 E-mail: \tt{E.A.Yildirim@ed.ac.uk}}}
\date{\today}

\begin{document}

\maketitle

\begin{abstract}

The Standard Quadratic optimization Problem (StQP), arguably the simplest among all classes of NP-hard optimization problems,  consists of extremizing a quadratic form (the simplest nonlinear polynomial) over the standard simplex (the simplest polytope/compact feasible set). As a problem class, StQPs may be nonconvex with an exponential number of inefficient local solutions. StQPs arise in a multitude of applications, among them mathematical finance, machine learning (clustering), and modeling in biosciences (e.g., selection and ecology). 
This paper deals with such StQPs under an additional sparsity or cardinality constraint, which, even for convex objectives, renders NP-hard problems. One motivation to study StQPs under such sparsity restrictions is the high-dimensional portfolio selection problem with too many assets to handle, in particular, in the presence of transaction costs. Here, relying on modern conic optimization techniques, we present  
tractable convex relaxations for 
this relevant but difficult problem. We propose novel equivalent reformulations of these relaxations with significant dimensional reduction, which is essential for the tractability of these relaxations when the problem size grows. 
Moreover, we propose an instance generation procedure which systematically avoids too easy instances. Our extensive computational results illustrate the high quality of the relaxation bounds in 
a significant number of instances. Furthermore, in contrast with exact mixed-integer quadratic programming models, the solution time of the relaxations is very robust to the choices of the problem parameters. In particular, the reduced formulations achieve significant improvements in terms of the solution time over their counterparts.

\end{abstract}

\keywords{Quadratic optimization, cardinality constraint, mixed-integer quadratic optimization, copositive optimization, doubly nonnegative relaxation}

{\bf AMS Subject Classification:} 90C11, 90C20, 90C22

\section{The problem --- introduction and some motivation} \label{intro}

In many applications, one is interested in sparse solutions to optimization problems. In this paper, we consider Standard Quadratic optimization Problems (StQPs) under a hard sparsity constraint. 

An StQP consists of minimizing a (not necessarily convex) quadratic form over the standard simplex:
\begin{equation}
\tag{StQP}\label{StQP} \ell_n(\Qb) := \min\limits_{\x \in \R^n} \left\{\x\T  \Qb \x: \x \in F_n\right\}\, , 
\end{equation}
where 
\begin{equation} \label{def_F_n}
F_n := \left\{\x \in \R^n_+ : \e\T \x = 1  \right\}
\end{equation}
is the standard simplex, the simplest polytope in the $n$-dimensional Euclidean space $\R^n$. Here, $\e \in \R^{n}$ denotes the vector of all ones.

Despite its simplicity, this problem class serves to model manifold real-world applications, ranging from the analysis of social networks (community detection via dominant-set-clustering)~\cite{Bomz20a,motzkin_straus_1965} through biology, game theory to economy and finance~\cite{Bomz02a,markowitz1952portfolio}, to cite just a few. For a survey on mixed-integer convex quadratic optimization approaches to portfolio selection, see, e.g.,~\cite{MDA2019}.

One motivation to introduce sparsity constraints comes from high-dimensional portfolio selection. Recently, the significant benefits of controlling the cardinality of a portfolio have been widely recognized in the literature, e.g., by~\cite{du2022high} who observed that maintaining a well-diversified portfolio on S$\&$P 500 datasets only required approximately $10 \sim 30$ assets. By fixing the number of holding stocks in advance, they bypass estimation of the covariance matrix $\Sigma$, focusing solely on the cardinality constraints. Even earlier,~\cite{hautsch2019large} already found empirically that commonly employed shrinkage methods significantly underperform in constructing portfolios. This effect, namely that using a mathematically well-behaved surrogate instead of the true sparsity term (which is discontinuous and thus non-convex) can be grossly misleading, is by now widely known in other application domains, e.g., in signal reconstruction.

Therefore, we propose to study the StQP with an exact, hard sparsity constraint, with the aim to develop tight yet tractable relaxations. Such a problem can be expressed as follows:
\begin{equation*}
\tag{StQP($\rho$)}\label{sStQP} \ell_\rho(\Qb) := \min\limits_{\x \in \R^n} \left\{\x\T  \Qb \x: \x \in F_\rho\right\}, 
\end{equation*}
where 
\begin{equation} \label{def_F_rho}
F_\rho := \left\{\x \in F_n : \|\x\|_0 \leq \rho \right\} = \left\{\x \in \R^n_+ : \e\T \x = 1,  \quad \|\x\|_0 \leq \rho\right\}\, .
\end{equation}
Here, $\|\x\|_0$ denotes the number of nonzero components of a vector $\x$ and $\rho \in [1\! : \!n] :=\{1,\ldots,n\}$ is the sparsity parameter. 

Before we start our discussion, let us stress that the primary aim of this study is to provide rigorous and tight bounds on an extremely difficult problem. In practice, this will most likely be complemented by fast, high-performance heuristics. Combining both methods will, as we hope, provide practical solutions of reasonable, guaranteed quality.

As an aside, let us mention that we address here any StQP of any finite size $n$, not just the average case as $n\to\infty$. As shown in~\cite{chen2015new,chen2013sparse}, asymptotically (for large $n$) with high probability, an StQP instance already has a global solution in $F_2$, so very sparse. However, it is unclear, in general, for which value of $n$ this effect kicks in (if at all) and what happens in highly structured (e.g., convex) StQPs as in the portfolio selection case, which may be somehow hidden in possibly lower-dimensional regions in the space of all instances; see, e.g.,~\cite{Bomz16a}.

\subsection{Literature review and contributions}

\textcolor{black}{
The sparse StQP problem arises in a wide range of applications.} \textcolor{black}{A prominent class is comprised of} \textcolor{black}{portfolio selection problems with an upper bound on the number of assets~\cite{bienstock1996computational,perold1984large,Xuetalijoc.2022.0344}.} \textcolor{black}{The relevance of this problem is well illustrated by the fact that the most recent release of {\tt MATLAB}~\cite{matlabOptimizationToolboxR2024a} now offers an iterative mixed-integer linear programming (MILP) algorithmic framework to solve the mixed-integer quadratic programming (MIQP) formulation of (convex) portfolio selection problems with lower and upper bounds on the cardinality, as well as the possibility of using semi-continuous variables (taking either zero or positive values above a certain fixed threshold). However, a quick test of this toolbox reveals that even for $n=30$ and extreme sparsity $\rho=1$, the solution time was around 40 seconds (note that the solution is trivial for $\rho=1$ by extracting the smallest diagonal element of the data matrix $\Qb$). For $\rho=2$, no solution was found within two hours. Clearly, solution times are dominated by all models discussed here, cf.~Figure~\ref{fig3a} below.}

\textcolor{black}{Further applications include} \textcolor{black}{selecting subsets of regression variables in computational algorithms \cite{miller1984selection}, and sparsity in the information rate of compressive sampling \cite{candes2008introduction}.
}

\textcolor{black}{
There are various studies on the sparse StQP problem. Those studies can be separated into two main categories. The first group of studies indirectly handles the sparsity constraint by regularization via the $\ell_1$-norm or the $\ell_p$-pseudonorms with $p \in (0,1)$, see for instance~\cite{brodie2009sparse,chen2013sparse}. The second category of studies considers the cardinality constraints in the sparse StQP problem directly with different approaches. A semidefinite programming based reformulation of the sparse StQP problem that helps generating perspective cuts has been studied in~\cite{frangioni2007sdp,zheng2014improving}; see also~\cite{Han2022}. In~\cite{burdakov2016mathematical}, the authors study a nonconvex reformulation of the sparse StQP problem with continuous variables and apply a regularization method for solving the reformulation. A semidefinite programming based heuristic method that can give an upper bound on the sparse StQP problem is introduced in~\cite{braun2005semidefinite}. A tight semidefinite relaxation of the sparse StQP problem with strong computational behavior under the assumption of convexity of the objective function is studied in~\cite{wiegele2021tight}. Based on the mixed-integer formulation of the sparse StQP problem, there are also many algorithms using discrete optimization techniques to either solve the problem to optimality, or to find an approximate solution (see, for instance,~\cite{bertsimas2009algorithm, bienstock1996computational, di2012concave, murray2012local, ruiz2010optimization, sun2013recent, zheng2014successive}).
}

This work builds upon the recent work of the same set of authors on the sparse StQP problem~\cite{BPQY23}, where the focus was on classical convex relaxations given by the reformulation-linearization technique (RLT), Shor relaxation (SDP), and their combination (SDP-RLT) arising from \eqref{P1}, one of the mixed-integer quadratic formulations considered this paper. They established several structural properties of these relaxations in relation to the corresponding relaxations of StQPs without any sparsity constraints, and utilized these relations to obtain several results about the quality of the lower bounds arising from different relaxations. In contrast, in this paper, we consider convex relaxations arising from exact copositive reformulations of two different MIQP formulations, which are provably tighter than each of the aforementioned classical relaxations. Furthermore, our focus is on finding tight but tractable relaxations that scale well with the problem dimension. 

Our contributions are as follows:

\begin{itemize}
    \item[(i)] We present exact copositive reformulations of two mixed integer quadratic optimization models \eqref{P1} and \eqref{P2}.

    \item[(ii)] We propose novel equivalent reformulations of doubly nonnegative relaxations of each of the two copositive reformulations in significantly smaller dimensions.

    \item[(iii)] We establish that one of the two relaxations is provably tighter than the other one.

    \item[(iv)] We propose an instance generation procedure that avoids instances of \eqref{sStQP} for which the cardinality constraint is redundant.
\end{itemize}

\subsection{Notation and organisation of the paper}

We use $\R^n$, $\R^n_+$, $\R^{m \times n}$, and $\cS^n$ to denote the $n$-dimensional Euclidean space, nonnegative orthant, the space of $m \times n$ real matrices, and the space of $n \times n$ real symmetric matrices, respectively. Throughout the paper, vectors and matrices are denoted by bold lowercase and bold uppercase letters, respectively (e.g.~$\x$ and $\Xb$). In particular, we reserve $\e$~(resp., $\Eb$) and $\o$~(resp., $\Ob$) for the vector (resp. matrix) of all ones and all zeroes of appropriate dimensions, respectively. The dimension should be clear from the context. Scalars are denoted by regular lowercase Roman or Greek letters. We employ subscripts to indicate a specific element of vectors or matrices. For instance, we denote the $i$th component of the vector $\x$ by $x_i$, and the $(i,j)$-entry of the matrix $\Xb$ by $X_{ij}$. For $\Xb \in \R^{n \times n}$ and two index sets $\mathbb{A}\subset\{1,...,n\}$, $\mathbb{B}\subset\{1,...,n\}$, $\Xb_{\mathbb{A}\mathbb{B}}$ is the  submatrix of $\Xb$ given by the rows and columns indexed by $\mathbb{A}$ and $\mathbb{B}$, respectively. Similarly, $\x_\mathbb{A}$ denotes the subvector of $\x \in \R^n$ given by the components indexed by $\mathbb{A}$. For $\Xb \in \R^{n \times n}$, we denote the column vector formed by the diagonal entries of $\Xb$ by $\diag(\Xb)$. Inequalities on vectors and matrices are understood to be componentwise scalar inequalities on each of the corresponding components. For $\Xb \in \cS^n$, we use $\Xb \succeq \Ob$ (resp., $\Xb \succ \Ob$) to denote that $\Xb$ is positive semidefinite (resp., positive definite). The trace inner product of $\Xb \in \R^{m \times n}$ and $\Yb \in \R^{m \times n}$ is denoted by $\langle \Xb, \Yb \rangle = \textrm{trace}\left(\Xb\T \Yb \right) = \sum\limits_{i=1}^m \sum\limits_{j=1}^n X_{ij} Y_{ij}$.

We define the following closed, convex cones in $\cS^n$:
\begin{eqnarray} \label{def_cones}
\label{def_CP}
{\cal CP}_n & = & \left\{\Mb \in \cS^n: \Mb = \Ab \Ab\T \quad \textrm{for some}~ \Ab \in \R^{n \times k}~\textrm{such that}~ \Ab \geq \Ob \right\}\, , \\
\label{def_DNN}
{\cal D}_n & = &  \left\{\Mb \in \cS^n: \Mb \succeq \Ob,~\Mb \geq \Ob \right\}\, , \\
\label{def_SPN}
{\cal SPN}_n & = &  \left\{\Mb \in \cS^n: \Mb = \Pb + \Nb, \textrm{ for some}~\Pb \succeq \Ob,~\Nb \geq \Ob \right\}\, , \\
\label{def_COP}
{\cal COP}_n & = & \left\{\Mb \in \cS^n: \u\T \Mb \u \geq 0, \quad \textrm{for all}~ \u \in \R^n_+ \right\}\, ,
\end{eqnarray}
i.e., ${\cal CP}_n$ is the cone of completely positive matrices, ${\cal D}_n$ is the cone of doubly nonnegative matrices, ${\cal SPN}^n$ is the cone of SPN matrices (those which can be decomposed into the sum of a positive semidefinite (PSD) and a componentwise nonnegative matrix), and ${\cal COP}^n$ is the cone of copositive matrices. It is well-known and easy to check that
\begin{equation} \label{cone_inclusion}
{\cal CP}_n \subseteq {\cal D}_n \subseteq {\cal SPN}_n \subseteq {\cal COP}_n\, .    
\end{equation}
Furthermore, ${\cal CP}_n = {\cal D}_n$ and ${\cal SPN}_n = {\cal COP}_n$ if and only if $n \leq 4$~(see \cite{ref:Diananda}).

The boundary of ${\cal COP}_n$, denoted by $\mathbf{bd} ~{\cal COP}_n$, is given by
\begin{equation} \label{cop_bd}
\mathbf{bd} ~{\cal COP}_n = \left\{\Mb \in {\cal COP}_n: \exists~\u \in F_n \textrm{ s.t. } \u\T \Mb \u = 0 \right\},
\end{equation} 
where $F_n$ is given by \eqref{def_F_n}.

For $\Mb \in \mathbf{bd} ~{\cal COP}_n$, the set of zeroes of $\Mb$ is given by
\begin{equation} \label{def_M_zeros}
\mathbf{V}^\Mb = \left\{\u \in F_n: \u\T \Mb \u = 0\right\}.
\end{equation}

The paper is organized as follows: Section~\ref{ExactForm} presents two exact conic reformulations, one by using binary variables directly (Section~\ref{ExactForm-P1}) and the other involving complementarity constraints (Section~\ref{ExactForm-P2}). In Section~\ref{SecConvRelax}, we discuss their doubly nonnegative relaxations and propose smaller equivalents. Section~\ref{comparDNN} compares the tractable lower bounds generated by them.
Section~\ref{generation} 
is devoted to careful generation procedures which systematically avoid too easy instances and provide testbeds of different complexity scale: PSD instances, then SPN instances, and finally COP instances which provably cannot be reduced to an instance in the previously mentioned classes. In Section~\ref{CompRes}, we describe our computational experiment in detail (Section~\ref{CRSetup}), discuss solution times 
in Section~\ref{CRSolTime}, and the quality of the lower bounds achieved in Section~\ref{CRQLB}. We conclude in Section~\ref{concl}.

\section{Two exact  completely positive formulations} \label{ExactForm}

In this section, we present two exact formulations of the sparse StQP problem~\eqref{sStQP} as mixed-binary quadratic optimization problems, a direct one covered by Section~\ref{ExactForm-P1}, whereas Section~\ref{ExactForm-P2} is devoted to a formulation involving a complementarity constraint.

\subsection{Conic formulation of direct MIQP} \label{ExactForm-P1}

By introducing binary variables, the sparse StQP can be reformulated as the following mixed-binary quadratic optimization problem:
\[
\begin{array}{llrrcl}
\tag{P1($\rho$)}\label{P1} & \ell_\rho (\Qb) = &\min\limits_{(\x,\u) \in \R^n \times \R^n} & \x\T  \Qb\x & & \\
 &&\textrm{s.t.}  & \e\T  \x & = &1 \\ 
 && & \e\T  \u & = & \rho \\ 
 && & \x & \leq & \u \\
 && & \u & \in & \{0,1\}^n  \\
 && & \x & \geq & \o\, .
    \end{array}
\]

By introducing the redundant constraints $\u \leq \e$, slack variables $\y \geq \o$ and $\v \geq \o$ so that $\x + \y = \u$ and $\u + \v = \e$, we arrive at the following reformulation of \eqref{P1}:
\[
\begin{array}{llrrcl}
\tag{P1A($\rho$)}\label{P1A} & \ell_\rho (\Qb) = &\min\limits_{(\x,\u,\v,\y) \in \R^n \times \R^n \times \R^n \times \R^n} & \x\T  \Qb\x & & \\
 &&\textrm{s.t.}  & \e\T  \x & = &1 \\ 
 && & \e\T  \u & = & \rho \\ 
 && & \x + \y & = & \u \\
 && & \u + \v & = & \e \\
 && & \u & \in & \{0,1\}^n  \\
 && & \x & \geq & \o \\
 && & \y &  \geq & \o \\
 && & \v & \geq & \o\,.
    \end{array}
\]

By~\cite{Burer09}, \eqref{P1A} admits an equivalent reformulation as the following conic optimization problem:
\[
\begin{array}{lllrcl}
 \tag{CP1A($\rho$)}\label{CP1A}
 &\ell_\rho (\Qb) = & \min\limits_{\Zb \in \cS^{4n + 1}} & \du{\Qb}{{\mathsf Z}^{\x\x}}  & & \\
 && \textrm{s.t.}  & \e\T  \x & = & 1 \\ 
  && & \e\T  \u & = & \rho \\ 
   && & \x + \y & = & \u \\
 && & \u + \v & = & \e \\
 && & \du{\Eb}{{\mathsf Z}^{\x\x}} & = & 1 \\ 
 && & \du{\Eb}{{\mathsf Z}^{\u\u}} & = & \rho^2 \\ 
  && & \diag({\mathsf Z}^{\u\u})  & = & \u\\
 && & \diag({\mathsf Z}^{\x\x}) + \diag({\mathsf Z}^{\y\y}) + \diag({\mathsf Z}^{\u\u}) + 2 \, \left[\diag({\mathsf Z}^{\x\y}) - \, \diag({\mathsf Z}^{\x\u}) - \, \diag({\mathsf Z}^{\u\y}) \right] & = & \o \\
 && & \diag({\mathsf Z}^{\u\u}) + 2 \, \diag({\mathsf Z}^{\u\v}) + \diag({\mathsf Z}^{\v\v}) & = & \e \\
 && & \Zb:=\begin{pmatrix}
    1 & \x\T & \u\T & \v\T & \y\T\\
    \x & {\mathsf Z}^{\x\x} & {\mathsf Z}^{\x\u} & {\mathsf Z}^{\x\v} & {\mathsf Z}^{\x\y} \\
    \u & ({\mathsf Z}^{\x\u})\T & {\mathsf Z}^{\u\u} & {\mathsf Z}^{\u\v} & {\mathsf Z}^{\u\y} \\
    \v & ({\mathsf Z}^{\x\v})\T & ({\mathsf Z}^{\u\v})\T & {\mathsf Z}^{\v\v} & {\mathsf Z}^{\v\y} \\
    \y & ({\mathsf Z}^{\x\y})\T & ({\mathsf Z}^{\u\y})\T & ({\mathsf Z}^{\v\y})\T & {\mathsf Z}^{\y\y}
    \end{pmatrix}& \in & \cp{4n+1}\, .
    \end{array}
\]

By~\cite{Burer09}, \eqref{P1A} and \eqref{CP1A} are equivalent in the following sense: Both of their optimal values are equal to $\ell_\rho (\Qb)$ and for any optimal solution $\Zb \in \cS^{4n + 1}$ of \eqref{CP1A}, $(\x,\u,\y,\v) \in \R^n \times \R^n \times \R^n \times \R^n$ is in the convex hull of the set of optimal solutions of \eqref{P1A}. Therefore, \eqref{CP1A} is an exact convex reformulation of the nonconvex optimization problem \eqref{P1A}.

\subsection{Conic formulation of MIQP with complementarity constraint} \label{ExactForm-P2}

In \eqref{P1}, sparsity constraints are handled by big-M constraints. Alternatively, sparsity can be enforced by dropping the big-M constraints $\x\le \u$ and replacing them with the complementarity constraints $\x_j (1 - \u_j) = 0,~j = 1,\ldots,n$ (see, e.g.,~\cite{burdakov2016mathematical}). Defining $\v = \e - \u \in \{0,1\}^n$, we obtain the following quadratic optimization problem with complementarity constraints:
\[
\begin{array}{llrrcl}
\tag{P2($\rho$)}\label{P2} & \ell_\rho (Q) = &\min\limits_{(\x,\v) \in \R^n \times \R^n} & \x\T  \Qb \x & & \\
 && \textrm{s.t.} 
  & \e\T  \x & = & 1 \\ 
& & & \e\T\v & = & n - \rho \\ 
& & & \x\T\v & = & 0 \\
& & & \v & \in & \{0,1\}^n  \\
& & & \x & \geq & \o.
    \end{array}
\]

Similarly, adding the redundant constraints $\v \leq \e$ and reintroducing the slack variables $\u \geq \o$ so that $\u + \v = \e$, \eqref{P2} can be reformulated as follows:
\[
\begin{array}{llrrcl}
\tag{P2A($\rho$)}\label{P2A} & \ell_\rho (\Qb) = &\min\limits_{(\x,\u,\v) \in \R^n \times \R^n \times \R^n} & \x\T  \Qb\x & & \\
 &&\textrm{s.t.}  & \e\T  \x & = &1 \\ 
  && & \e\T  \u & = & \rho \\
   && & \u + \v & = & \e \\
 && & \x\T \v & = & 0 \\ 
 & & & \u & \in & \{0,1\}^n  \\
 && & \x & \geq & \o \\
 && & \v &  \geq & \o \\
 && & \u & \geq & \o\,.
    \end{array}
\]

By~\cite{Bomz23a,Burer09}, \eqref{P2A} admits an equivalent reformulation as the following copositive optimization problem:
\[
\begin{array}{lllrcl}
 \tag{CP2A($\rho$)}\label{CP2A}
 &\ell_\rho (\Qb) = & \min\limits_{\Yb \in \cS^{3n+1}} & \du{\Qb}{{\mathsf Y}^{\x\x}}  & & \\
 && \textrm{s.t.}  & \e\T  \x & = & 1 \\ 
  && & \e\T  \u & = & \rho \\
 && & \u + \v & = & \e \\
 && & \du{\Eb}{{\mathsf Y}^{\x\x}} & = & 1 \\ 
 && & \du{\Eb}{{\mathsf Y}^{\u\u}} & = & \rho^2 \\ 
 && & \diag({\mathsf Y}^{\u\u}) & = & \u \\
 && & \diag({\mathsf Y}^{\u\u}) + 2 \, \diag({\mathsf Y}^{\u\v}) + \diag({\mathsf Y}^{\v\v}) & = & \e \\
  && & \e\T \diag({\mathsf Y}^{\x\v}) & = & 0 \\
 && & \Yb:=\begin{pmatrix}
    1 & \x\T & \u\T & \v\T\\
    \x & {\mathsf Y}^{\x\x} & {\mathsf Y}^{\x\u} & {\mathsf Y}^{\x\v}\\
    \u & ({\mathsf Y}^{\x\u})\T & {\mathsf Y}^{\u\u} & {\mathsf Y}^{\u\v}\\
    \v & ({\mathsf Y}^{\x\v})\T & ({\mathsf Y}^{\u\v})\T & {\mathsf Y}^{\v\v}\\
    \end{pmatrix}& \in & \cp{3n+1}.
    \end{array}
\]

Once again, the equivalence between \eqref{P2A} and \eqref{CP2A} is understood similarly to that between \eqref{P1A} and \eqref{CP1A}.

\section{Two convex relaxations and their smaller equivalents}  
\label{SecConvRelax}

In this section, we consider two convex relaxations arising from the conic formulations \eqref{CP1A} and \eqref{CP2A}. 

Despite the fact that \eqref{CP1A} and \eqref{CP2A} are convex optimization problems, they are, in general, NP-hard~\cite{MurtyK87}. A well-known tractable convex relaxation of completely positive conic problems is obtained by replacing the intractable cone $\cp{d}$ with the larger cone $\cD_{d}$ of doubly nonnegative (DNN) matrices, i.e., the convex cone of symmetric $d\times d$ matrices that are both PSD and componentwise nonnegative. It follows that the optimal value of the DNN relaxation is a lower bound on the optimal value of the completely positive conic optimization problem. Here, we present the DNN relaxations arising from conic optimization problems \eqref{CP1A} and \eqref{CP2A}. Furthermore, we establish that each of the two DNN relaxations can, in fact, be reformulated in smaller dimensions without sacrificing the strength of the corresponding lower bound. Finally, we present a theoretical comparison of the strength of the two DNN relaxations.

\subsection{DNN relaxations of the direct MIQP formulation} \label{ConvRelaxP1}

In this section, we consider the DNN relaxation of the problem \eqref{CP1A}. By replacing the intractable conic constraint $\Zb \in \cp{4n+1}$ in \eqref{CP1A} by $\Zb \in \cD_{4n+1}$, where $\cD_{4n+1}$ denotes the cone of doubly nonnegative (DNN) matrices in $\cS^{4n+1}$, we obtain the following tractable convex relaxation of \eqref{P1A}:

\[
\begin{array}{lllrcl}
 \tag{D1A($\rho$)}\label{D1A}
 &\nu(\textrm{D1A($\rho$)}) := & \min\limits_{\Zb \in \cS^{4n + 1}} & \du{\Qb}{{\mathsf Z}^{\x\x}}  & & \\
 && \textrm{s.t.}  & \e\T  \x & = & 1 \\ 
  && & \e\T  \u & = & \rho \\ 
   && & \x + \y & = & \u \\
 && & \u + \v & = & \e \\
 && & \du{\Eb}{{\mathsf Z}^{\x\x}} & = & 1 \\ 
 && & \du{\Eb}{{\mathsf Z}^{\u\u}} & = & \rho^2 \\ 
  && & \diag({\mathsf Z}^{\u\u})  & = & \u\\
 && & \diag({\mathsf Z}^{\x\x}) + \diag({\mathsf Z}^{\y\y}) + \diag({\mathsf Z}^{\u\u}) + 2 \, \left[\diag({\mathsf Z}^{\x\y}) - \, \diag({\mathsf Z}^{\x\u}) - \, \diag({\mathsf Z}^{\u\y}) \right] & = & \o \\
 && & \diag({\mathsf Z}^{\u\u}) + 2 \, \diag({\mathsf Z}^{\u\v}) + \diag({\mathsf Z}^{\v\v}) & = & \e \\
 && & \Zb:=\begin{pmatrix}
    1 & \x\T & \u\T & \v\T & \y\T\\
    \x & {\mathsf Z}^{\x\x} & {\mathsf Z}^{\x\u} & {\mathsf Z}^{\x\v} & {\mathsf Z}^{\x\y} \\
    \u & ({\mathsf Z}^{\x\u})\T & {\mathsf Z}^{\u\u} & {\mathsf Z}^{\u\v} & {\mathsf Z}^{\u\y} \\
    \v & ({\mathsf Z}^{\x\v})\T & ({\mathsf Z}^{\u\v})\T & {\mathsf Z}^{\v\v} & {\mathsf Z}^{\v\y} \\
    \y & ({\mathsf Z}^{\x\y})\T & ({\mathsf Z}^{\u\y})\T & ({\mathsf Z}^{\v\y})\T & {\mathsf Z}^{\y\y}
    \end{pmatrix}& \in & \cD_{4n+1}.
    \end{array}
\]

Note that \eqref{D1A} consists of $5 n + 4$ linear equality constraints and a doubly nonnegative constraint in $\cS^{4n+1}$. First, we make several observations about feasible solutions of \eqref{D1A}.

\begin{lemma} \label{D1Aprops}
Let $\Zb \in \cS^{4n + 1}$ be (D1A($\rho$))-feasible. Then, the following relations hold:
\begin{eqnarray} \label{D1Arels}
\x \e\T - \Zb^{\x \u} & = & \Zb^{\x \v} \geq \Ob, \label{D1Arels1} \\
- \Zb^{\x\x} + \Zb^{\x\u} & = & \Zb^{\x\y} \geq \Ob, \label{D1Arels2} \\
\Zb^{\x\x} - \Zb^{\x\u} - (\Zb^{\x\u})\T + \Zb^{\u\u} & = & \Zb^{\y\y} \geq \Ob, \label{D1Arels3} \\
\e \e\T - \e \u\T - \u \e\T + \Zb^{\u\u} & = & \Zb^{\v\v} \geq \Ob, \label{D1Arels4} \\
-\e \x\T + \left(\Zb^{\x\u}\right)\T + \e \u\T - \Zb^{\u\u} & = & \Zb^{\v \y} \geq \Ob \,,  \label{D1Arels5} \\
\diag(\Zb^{\x\u}) & = & \x \, , \label{D1Arels6} \\
\diag(\Zb^{\x\v}) & = & \o \, , \label{D1Arels7} \\
\diag(\Zb^{\v\y}) & = & \o \, , \label{D1Arels8} \\
\diag(\Zb^{\v\v}) & = & \v \, , \label{D1Arels9} \\
\diag(\Zb^{\u\v}) & = & \o \, . \label{D1Arels10}
\end{eqnarray}
\end{lemma}
\begin{proof}
Let $\Zb \in \cS^{4n + 1}$ be \eqref{D1A}-feasible. By using the Schur complementation, we obtain
\[
\begin{pmatrix}
    {\mathsf Z}^{\x\x} & {\mathsf Z}^{\x\u} & {\mathsf Z}^{\x\v} & {\mathsf Z}^{\x\y} \\
    ({\mathsf Z}^{\x\u})\T & {\mathsf Z}^{\u\u} & {\mathsf Z}^{\u\v} & {\mathsf Z}^{\u\y} \\
    ({\mathsf Z}^{\x\v})\T & ({\mathsf Z}^{\u\v})\T & {\mathsf Z}^{\v\v} & {\mathsf Z}^{\v\y} \\
    ({\mathsf Z}^{\x\y})\T & ({\mathsf Z}^{\u\y})\T & ({\mathsf Z}^{\v\y})\T & {\mathsf Z}^{\y\y}
    \end{pmatrix} = \begin{pmatrix} \x \\ \u \\ \v \\ \y \end{pmatrix} \begin{pmatrix} \x \\ \u \\ \v \\ \y \end{pmatrix}\T + \sum\limits_{k \in K} \begin{pmatrix} {\a}^k \\ {\b}^k \\ {\c}^k \\ {\d}^k \end{pmatrix} \begin{pmatrix} {\a}^k \\ {\b}^k \\ {\c}^k \\ {\d}^k \end{pmatrix}\T,
\]
where $K$ is a finite set and $\{{\a}^k,{\b}^k, {\c}^k ,{\d}^k\} \subset \R^n$ for each $k \in K$. Using $\diag({\mathsf Z}^{\u\u}) + 2 \, \diag({\mathsf Z}^{\u\v}) + \diag({\mathsf Z}^{\v\v}) = \e$, we obtain
\begin{eqnarray*}
u_i^2 + \sum\limits_{k \in K} \left( b_i^k \right)^2 + 2 \left[ u_i v_i + \sum\limits_{k \in K} \left( b_i^k \right) \left( c_i^k \right) \right] + v_i^2 + \sum\limits_{k \in K} \left( c_i^k \right)^2 & = & \left( u_i + v_i \right)^2 + \sum\limits_{k \in K} \left( b_i^k + c_i^k \right)^2 \\
 & = & 1 + \sum\limits_{k \in K} \left( b_i^k + c_i^k \right)^2  =  1
\end{eqnarray*}
for each $i \in \{ 1,\ldots,n\}$, where we used $\u + \v = \e$ in the second equality. We conclude that
\begin{equation} \label{imp1}
    \b^k = - \c^k, \quad \mbox{for all }k \in K\, .
\end{equation}
In a similar manner, by using $\diag({\mathsf Z}^{\x\x}) + \diag({\mathsf Z}^{\y\y}) + \diag({\mathsf Z}^{\u\u}) + 2 \, \left[\diag({\mathsf Z}^{\x\y}) - \, \diag({\mathsf Z}^{\x\u}) - \, \diag({\mathsf Z}^{\u\y}) \right] = \o$, we arrive at
\[
\left( x_i +  y_i -  u_i \right)^2 + \sum\limits_{k \in K} \left(  a_i^k +  d_i^k -  b_i^k \right)^2 = 0
\]
for each $i \in\{ 1,\ldots,n\}$, which, together with $\x + \y = \u$, implies that
\begin{equation} \label{imp2}
\a^k + \d^k = \b^k, \quad  \mbox{for all } k \in K\, .    
\end{equation}
Therefore,
\[
\x \e\T - \Zb^{\x \u} = \x \e\T - \x \u\T - \sum\limits_{k \in K} \a^k \left( \b^k \right)\T = \x \v\T + \sum\limits_{k \in K} \a^k \left( \c^k \right)\T = \Zb^{\x \v} \geq \Ob,
\]
where we used $\u + \v = \e$, \eqref{imp1}, and $\Zb \in \cD_{4n+1}$. This establishes \eqref{D1Arels1}. Arguing similarly, 
\[
- \Zb^{\x\x} + \Zb^{\x\u} = - \x \x\T - \sum\limits_{k \in K} \a^k \left( \a^k \right)\T + \x \u\T + \sum\limits_{k \in K} \a^k \left( \b^k \right)\T = \x \y\T + \sum\limits_{k \in K} \a^k \left( \d^k \right)\T = \Zb^{\x\y} \geq \Ob,
\]
where we used $\x + \y = \u$, \eqref{imp2}, and $\Zb \in \cD_{4n+1}$. This establishes \eqref{D1Arels2}. Next,
\begin{eqnarray*}
\Zb^{\x\x} - \Zb^{\x\u} - (\Zb^{\x\u})\T + \Zb^{\u\u} & = & \x \x\T + \sum\limits_{k \in K} \a^k \left( \a^k \right)\T - \x \u\T - \sum\limits_{k \in K} \a^k \left( \b^k \right)\T \\
 &  & \quad - \u \x\T - \sum\limits_{k \in K} \b^k \left( \a^k \right)\T + \u \u\T + \sum\limits_{k \in K} \b^k \left( \b^k \right)\T \\
 & = & \y \y^T + \sum\limits_{k \in K} \d^k \left( \d^k \right)\T \\
 & = & \Zb^{\y\y} \geq \Ob,
\end{eqnarray*}
where we used $\x + \y = \u$, \eqref{imp2}, and $\Zb \in \cD_{4n+1}$. Similarly, 
\[
\e \e\T - \e \u\T - \u \e\T + \Zb^{\u\u} = \e \e\T - \e \u\T - \u \e\T + \u \u\T + \sum\limits_{k \in K} \b^k \left( \b^k \right)\T = \v \v\T + \sum\limits_{k \in K} \c^k \left( \c^k \right)\T = \Zb^{\v\v} \geq \Ob,
\]
where we used $\u + \v = \e$, \eqref{imp1}, and $\Zb \in \cD_{4n+1}$, establishing \eqref{D1Arels4}. Next, 
\begin{eqnarray*}
    -\e \x\T + \left(\Zb^{\x\u}\right)\T + \e \u\T - \Zb^{\u\u} & = & -\e \x\T + \u \x\T + \sum\limits_{k \in K} \b^k \left( \a^k \right)\T + \e \u\T - \u \u\T - \sum\limits_{k \in K} \b^k \left( \b^k \right)\T \\
    & = & - \v \x\T + \v \u\T + \sum\limits_{k \in K} \c^k \left( \d^k \right)\T \\
    & = & \Zb^{\v \y} \geq \Ob, 
\end{eqnarray*}
where we used $\u + \v = \e$, $\x + \y = \u$, \eqref{imp1}, \eqref{imp2}, and $\Zb \in \cD_{4n+1}$, establishing \eqref{D1Arels5}. 

Using $\diag({\mathsf Z}^{\u\u}) = \u$, \eqref{D1Arels1}, and \eqref{D1Arels5},
\begin{eqnarray*}
\o\le \diag \left( \x \e\T - {\mathsf Z}^{\x\u} \right) & = & \x - \diag({\mathsf Z}^{\x\u}) \quad \mbox{and}\\
\o\le \diag \left( -\e \x\T + ({\mathsf Z}^{\x\u})\T + \e \u\T - {\mathsf Z}^{\u\u} \right) & = & -\x + \diag({\mathsf Z}^{\x\u}) + \u - \diag({\mathsf Z}^{\u\u}) = -\x + \diag({\mathsf Z}^{\x\u})\, ,
\end{eqnarray*}
which together yield \eqref{D1Arels6}. By \eqref{D1Arels6} and $\diag({\mathsf Z}^{\u\u}) = \u$, it is easy to see that \eqref{D1Arels7} and \eqref{D1Arels8} follow from \eqref{D1Arels1} and \eqref{D1Arels5}, respectively. Using \eqref{D1Arels4}, $\diag({\mathsf Z}^{\u\u}) = \u$ and $\u + \v = \e$, we obtain
\[
\diag(\e \e\T - \e \u\T - \u \e\T + \Zb^{\u\u})  = 
 \e - 2 \, \u + \u = \v = \diag(\Zb^{\v\v}),
\]
which establishes \eqref{D1Arels9}. 
Finally,
\eqref{D1Arels10} follows from \eqref{D1Arels9}, $\diag({\mathsf Z}^{\u\u}) = \u$, $\diag({\mathsf Z}^{\u\u}) + 2 \, \diag({\mathsf Z}^{\u\v}) + \diag({\mathsf Z}^{\v\v}) = \e$, and $\u + \v = \e$. This completes the proof.
\end{proof}

In view of Lemma~\ref{D1Aprops}, let us now consider the following optimization problem:
\[
\begin{array}{lllrcl}
 \tag{D1B($\rho$)}\label{D1B}
 &\nu(\textrm{D1B($\rho$)}) := & \min\limits_{\Wb \in \cS^{2n + 1}} & \du{\Qb}{{\mathsf W}^{\x\x}}  & & \\
 && \textrm{s.t.}  & \e\T  \x & = & 1 \\ 
  && & \e\T  \u & = & \rho \\ 
 && & \du{\Eb}{{\mathsf W}^{\x\x}} & = & 1 \\ 
 && & \du{\Eb}{{\mathsf W}^{\u\u}} & = & \rho^2 \\ 
 && & \diag({\mathsf W}^{\u\u})  & = & \u\\
 && & \x \e\T - {\mathsf W}^{\x\u} & \geq & \Ob\\
&& & -{\mathsf W}^{\x\x} + {\mathsf W}^{\x\u} & \geq & \Ob\\
 && & {\mathsf W}^{\x\x} - {\mathsf W}^{\x\u} - ({\mathsf W}^{\x\u})\T + {\mathsf W}^{\u\u} & \geq & \Ob\\
 && & \e \e\T - \e \u\T - \u \e\T + {\mathsf W}^{\u\u}  & \geq & \Ob\\
  && & -\e \x\T + ({\mathsf W}^{\x\u})\T + \e \u\T - {\mathsf W}^{\u\u} & \geq & \Ob\\
 && & {\mathsf W}^{\x\x}  & \geq & \Ob\\
 && & \Wb:=\begin{pmatrix}
    1 & \x\T & \u\T \\
    \x & {\mathsf W}^{\x\x} & {\mathsf W}^{\x\u} \\
    \u & ({\mathsf W}^{\x\u})\T & {\mathsf W}^{\u\u} 
    \end{pmatrix}& \succeq & \Ob \,.
    \end{array}
\]

It is easy to see that \eqref{D1B} is a convex relaxation of \eqref{P1} since, for any feasible solution $(\x,\u) \in \R^n \times \R^n$ of \eqref{P1}, 
\[
\Wb = \begin{pmatrix}
    1 \\ \x \\ \u
    \end{pmatrix} \begin{pmatrix}
    1 \\ \x \\ \u
    \end{pmatrix}\T \in \cS^{2n + 1}
\]
is a feasible solution of \eqref{D1B} with the same objective function value. 
Note that \eqref{D1B} has $n + 4$ linear equality constraints, $(9/2) n^2 +  (3/2) n $ linear inequality constraints, and a PSD constraint in $\cS^{2n + 1}$. It is worth noticing that the inequality constraints of \eqref{D1B} are precisely given by all the RLT (reformulation-linearization technique) constraints obtained from the inequalities $\x \geq \o$, $\x \leq \u$, and $\u \leq \e$ (see, e.g.,~\cite{Sherali1999}). 

First, we present some properties of feasible solutions of \eqref{D1B}.

\begin{lemma} \label{D1Bprops}
Let $\Wb \in \cS^{2n+1}$ be \eqref{D1B}-feasible. Then, 
\begin{eqnarray} \label{D1Brels}
 \Wb & \in & \cD_{2n + 1} \, , \label{D1Brels1} \\
 -({\mathsf W}^{\x\u})\T + {\mathsf W}^{\u\u} & \geq & \Ob  \, , \label{D1Brels2} \\
 \u \e\T - {\mathsf W}^{\u\u} & \geq & \Ob \, , \label{D1Brels3}\\
\u & \leq & \e\, , \label{D1Brels4} \\
\x & \leq & \u \, . \label{D1Brels6}
\end{eqnarray}
\end{lemma}
\begin{proof}
Let $\Wb \in \cS^{2n+1}$ be \eqref{D1B}-feasible. Since $\Wb^{\x \x}\ge\Ob$ and $- \Wb^{\x \x} + \Wb^{\x \u} \ge \Ob$, we obtain $\Wb^{\x \u}\ge \Ob$. By $ \x \e\T - {\mathsf W}^{\x\u} \geq \Ob$, we arrive at $\x \geq \o$. Since $\diag({\mathsf W}^{\u\u}) =  \u$ and $\Wb \succeq \Ob$, we have $\u \geq \o$. Then, adding the inequalities
$-{\mathsf W}^{\x\x} + {\mathsf W}^{\x\u}  \geq \Ob$
and ${\mathsf W}^{\x\x} - {\mathsf W}^{\x\u} - ({\mathsf W}^{\x\u})\T + {\mathsf W}^{\u\u} \geq  \Ob$ yields \eqref{D1Brels2}, which, in turn, establishes \eqref{D1Brels1} since ${\mathsf W}^{\u\u} \geq (\Wb^{\x \u})\T \geq \Ob$. Similarly, adding the inequalities $\left(\x \e\T - {\mathsf W}^{\x\u}\right)\T = \e \x\T - ({\mathsf W}^{\x\u})\T \geq \Ob$ and $-\e \x\T + ({\mathsf W}^{\x\u})\T + \e \u\T - {\mathsf W}^{\u\u} \geq \Ob$, we arrive at \eqref{D1Brels3}. Adding \eqref{D1Brels3} to $\e \e\T - \e \u\T - \u \e\T + {\mathsf W}^{\u\u} \geq \Ob$, we get $\e(\e-\u)\T\ge \Ob$, which yields \eqref{D1Brels4}.
Similarly, using $\diag({\mathsf W}^{\u\u}) = \u$,
\begin{eqnarray*}
\o\le \diag \left( \x \e\T - {\mathsf W}^{\x\u} \right) & = & \x - \diag({\mathsf W}^{\x\u}) \quad \mbox{and}\\
\o\le \diag \left( -\e \x\T + ({\mathsf W}^{\x\u})\T + \e \u\T - {\mathsf W}^{\u\u} \right) & = & -\x + \diag({\mathsf W}^{\x\u}) + \u - \diag({\mathsf W}^{\u\u}) = -\x + \diag({\mathsf W}^{\x\u})\, ,
\end{eqnarray*}
which together yield $\diag({\mathsf W}^{\x\u}) = \x$. Combining this with \eqref{D1Brels2}, we get
\[\o\le 
\diag \left( -({\mathsf W}^{\x\u})\T + {\mathsf W}^{\u\u} \right) = - \diag({\mathsf W}^{\x\u}) + \diag({\mathsf W}^{\u\u}) = - \x + \u \, ,
\]
where we used $\diag({\mathsf W}^{\u\u}) = \u$. This establishes \eqref{D1Brels6} and completes the proof.
\end{proof}

Our next result shows that \eqref{D1B} is equivalent to the DNN relaxation (D1A($\rho$)) of \eqref{CP1A}. 

\begin{proposition} \label{DNNequiv1}
\eqref{D1A} and \eqref{D1B} are equivalent to each other. Therefore, $\nu(\textrm{D1A($\rho$)}) = \nu(\textrm{D1B($\rho$)})$.  
\end{proposition}
\begin{proof}
Let $\Wb \in \cS^{2n+1}$ be \eqref{D1B}-feasible. Let us define $\Zb \in \cS^{4n + 1}$ as follows:
\begin{equation} \label{defZfromW}
\Zb = \begin{pmatrix} 1 & \o\T & \o\T \\
\o & \Ib & \o \\ \o & \Ob & \Ib \\ \e & \Ob & - \Ib \\ \o & - \Ib & \Ib  \end{pmatrix}  \begin{pmatrix}
    1 & \x\T & \u\T \\
    \x & {\mathsf W}^{\x\x} & {\mathsf W}^{\x\u} \\
    \u & ({\mathsf W}^{\x\u})\T & {\mathsf W}^{\u\u} 
    \end{pmatrix} \begin{pmatrix} 1 & \o\T & \o\T \\
\o & \Ib & \o \\ \o & \Ob & \Ib \\ \e & \Ob & - \Ib \\ \Ob & - \Ib & \Ib \end{pmatrix}\T.  
\end{equation}
Therefore,
    \begin{equation} \label{Z_identity}
    \Zb = \begin{pmatrix}
    1 & \x\T & \u\T & \e\T - \u\T & -\x\T +\u\T  \\
    \x & {\mathsf W}^{\x\x} & {\mathsf W}^{\x\u} & \x \e\T - {\mathsf W}^{\x\u} & -{\mathsf W}^{\x\x} + {\mathsf W}^{\x\u} \\
    \u & ({\mathsf W}^{\x\u})\T & {\mathsf W}^{\u\u} & \u \e\T - {\mathsf W}^{\u\u} & -({\mathsf W}^{\x\u})\T + {\mathsf W}^{\u\u}  \\
    \e - \u & \e\x\T - ({\mathsf W}^{\x\u})\T & \e \u\T - {\mathsf W}^{\u\u} & \e \e\T - \e \u\T - \u \e\T + {\mathsf W}^{\u\u} & - \e \x\T + ({\mathsf W}^{\x\u})\T + \e \u\T - {\mathsf W}^{\u\u} \\ 
    -\x + \u & -{\mathsf W}^{\x\x} + ({\mathsf W}^{\x\u})\T & -{\mathsf W}^{\x\u} + {\mathsf W}^{\u\u} & -\x \e\T + {\mathsf W}^{\x\u} + \u \e\T - {\mathsf W}^{\u\u} & {\mathsf W}^{\x\x} - {\mathsf W}^{\x\u} - ({\mathsf W}^{\x\u})\T + {\mathsf W}^{\u\u} 
    \end{pmatrix}.
\end{equation}
We claim that $\Zb \in \cS^{4n + 1}$ is (D1A($\rho$))-feasible. By Lemma~\ref{D1Bprops}, \eqref{defZfromW}, and \eqref{Z_identity}, we conclude that $\Zb \in \cD_{4n + 1}$. Since $\v = \e -\u$, $\y = \u - \x$, $\Zb^{\x\x} = \Wb^{\x\x}$, and $\Zb^{\u\u} = \Wb^{\u\u}$, the first seven sets of equality constraints of (D1A($\rho$)) are satisfied. Considering the eighth set of equality constraints of (D1A($\rho$)) and substituting the corresponding submatrices in \eqref{Z_identity}, we obtain
\begin{eqnarray*}
\diag({\mathsf Z}^{\x\x}) & = & \diag({\mathsf W}^{\x\x}) \\
\diag({\mathsf Z}^{\y\y}) & = & \diag({\mathsf W}^{\x\x}) - 2 \, \diag({\mathsf W}^{\x\u}) + \diag({\mathsf W}^{\u\u}) \\
\diag({\mathsf Z}^{\u\u}) & = & \diag({\mathsf W}^{\u\u})\\
\diag({\mathsf Z}^{\x\y}) & = & - \diag({\mathsf W}^{\x\x}) + \diag({\mathsf W}^{\x\u}) \\
\diag({\mathsf Z}^{\x\u}) & = & \diag({\mathsf W}^{\x\u}) \\
\diag({\mathsf Z}^{\u\y}) & = & - \diag({\mathsf W}^{\x\u}) + \diag({\mathsf W}^{\u\u}),
\end{eqnarray*}
which implies that 
\[
\diag({\mathsf Z}^{\x\x}) + \diag({\mathsf Z}^{\y\y}) + \diag({\mathsf Z}^{\u\u}) + 2 \, \left[\diag({\mathsf Z}^{\x\y}) - \, \diag({\mathsf Z}^{\x\u}) - \, \diag({\mathsf Z}^{\u\y}) \right] = \o\, . 
\]
Finally, consider the last set of equality constraints of (D1A($\rho$)):
\[
\diag({\mathsf Z}^{\u\u}) + 2 \, \diag({\mathsf Z}^{\u\v}) + \diag({\mathsf Z}^{\v\v}) = \diag({\mathsf W}^{\u\u}) + 2 \, \left( \u - \diag({\mathsf W}^{\u\u}) \right) + \e - 2 \, \u + \diag({\mathsf W}^{\u\u}) = \e\, ,
\]
which implies that $\Zb \in \cS^{4n + 1}$ is (D1A($\rho$))-feasible and achieves the same objective function value as $\Wb$ in \eqref{D1B}.

Conversely, let $\Zb \in \cS^{4n + 1}$ be (D1A($\rho$))-feasible. Let $\Wb \in \cS^{2n + 1}$ be given by the top left $3 \times 3$-block of $\Zb$, i.e., 
\begin{equation} \label{defWfromZ}
\Wb = \begin{pmatrix}
    1 & \x\T & \u\T \\
    \x & {\mathsf Z}^{\x\x} & {\mathsf Z}^{\x\u} \\
    \u & ({\mathsf Z}^{\x\u})\T & {\mathsf Z}^{\u\u} 
    \end{pmatrix} \, .
\end{equation}
Clearly, $\Wb \in \cD_{2n + 1}$ and satisfies all linear equality constraints and $\Wb^{\x \x} = \Zb^{\x \x} \geq \Ob$ in \eqref{D1B}. By Lemma~\ref{D1Aprops}, we conclude that $\Wb$ satisfies each of the remaining linear inequality constraints and achieves the same objective function value as $\Zb$ in \eqref{D1A}. It follows that $\nu(\textrm{D1A($\rho$)}) = \nu(\textrm{D1B($\rho$)})$.
\end{proof}

Recall that \eqref{D1A} consists of $5 n + 4$ linear equality constraints and {\bf a doubly nonnegative constraint} in $\cS^{4n+1}$. In contrast, \eqref{D1B} has $n + 4$ linear equality constraints, $(9/2) n^2 +  (3/2) n $ linear inequality constraints, and {\bf a PSD constraint} in $\cS^{2n + 1}$. By Proposition~\ref{DNNequiv1}, we conclude that the lower bound arising from the DNN relaxation of \eqref{CP1A} can be computed by solving a conic optimization problem in a much smaller dimension.

\subsection{DNN relaxations of the complementarity constraint formulation}

In this section, we focus on the DNN relaxation of the copositive optimization problem \eqref{CP2A}. 
By replacing the intractable conic constraint $\Zb \in \cp{3n+1}$ in \eqref{CP2A} by $\Zb \in \cD_{3n+1}$, we obtain the following convex optimization problem:

\[
\begin{array}{lllrcl}
 \tag{D2A($\rho$)}\label{D2A}
 &\nu(\textrm{D2A($\rho$)}) := & \min\limits_{\Yb \in \cS^{3n + 1}} & \du{\Qb}{{\mathsf Y}^{\x\x}}  & & \\
 && \textrm{s.t.}  & \e\T  \x & = & 1 \\ 
  && & \e\T  \u & = & \rho \\
 && & \u + \v & = & \e \\
 && & \du{\Eb}{{\mathsf Y}^{\x\x}} & = & 1 \\ 
 && & \du{\Eb}{{\mathsf Y}^{\u\u}} & = & \rho^2 \\ 
 && & \diag({\mathsf Y}^{\u\u}) & = & \u \\
  && & \e\T \diag({\mathsf Y}^{\x\v}) & = & 0 \\
 && & \diag({\mathsf Y}^{\u\u}) + 2 \, \diag({\mathsf Y}^{\u\v}) + \diag({\mathsf Y}^{\v\v}) & = & \e \\
 && & \Yb:=\begin{pmatrix}
    1 & \x\T & \u\T & \v\T\\
    \x & {\mathsf Y}^{\x\x} & {\mathsf Y}^{\x\u} & {\mathsf Y}^{\x\v}\\
    \u & ({\mathsf Y}^{\x\u})\T & {\mathsf Y}^{\u\u} & {\mathsf Y}^{\u\v}\\
    \v & ({\mathsf Y}^{\x\v})\T & ({\mathsf Y}^{\u\v})\T & {\mathsf Y}^{\v\v}\\
    \end{pmatrix}& \in & \cD_{3n+1}.
    \end{array}
\]

Once again, it is easy to verify that \eqref{D2A} is a 
tractable convex relaxation of \eqref{P2A}. 
Note that \eqref{D2A} consists of $3 n + 5$ linear equality constraints and a doubly nonnegative constraint in $\cS^{3n+1}$. 

First, we establish several properties of feasible solutions of \eqref{D2A}.

\begin{lemma} \label{D2Aprops}
Let $\Yb \in \cS^{3n + 1}$ be \eqref{D2A}-feasible. Then, the following relations hold:
\begin{eqnarray} \label{D2Arels}
\x \e\T - \Yb^{\x \u} & = & \Yb^{\x \v} \geq \Ob, \label{D2Arels1} \\
\e \e\T - \e \u\T - \u \e\T + \Yb^{\u\u} & = & \Yb^{\v\v} \geq \Ob, \label{D2Arels2} \\
\u \e\T - \Yb^{\u\u}
& = & \Yb^{\u \v} \geq \Ob \,,  \label{D2Arels3} \\
\diag(\Yb^{\x\v}) & = & \o \, , \label{D2Arels4} \\
\diag(\Yb^{\x\u}) & = & \x \, , \label{D2Arels5} \\
\diag(\Yb^{\v\v}) & = & \v \, , \label{D2Arels6} \\
\diag(\Yb^{\u\v}) & = & \o \, . \label{D2Arels7}
\end{eqnarray}    
\end{lemma}
\begin{proof}
Let $\Yb \in \cS^{3n + 1}$ be \eqref{D2A}-feasible. By using the Schur complementation, we obtain
\[
\begin{pmatrix}
    {\mathsf Y}^{\x\x} & {\mathsf Y}^{\x\u} & {\mathsf Y}^{\x\v}\\
    ({\mathsf Y}^{\x\u})\T & {\mathsf Y}^{\u\u} & {\mathsf Y}^{\u\v}\\
    ({\mathsf Y}^{\x\v})\T & ({\mathsf Y}^{\u\v})\T & {\mathsf Y}^{\v\v} \end{pmatrix} = \begin{pmatrix} \x \\ \u \\ \v \end{pmatrix} \begin{pmatrix} \x \\ \u \\ \v \end{pmatrix}\T + \sum\limits_{k \in K} \begin{pmatrix} {\a}^k \\ {\b}^k \\ {\c}^k \end{pmatrix} \begin{pmatrix} {\a}^k \\ {\b}^k \\ {\c}^k \end{pmatrix}\T,
\]
where $K$ is a finite set and $\{{\a}^k,{\b}^k, {\c}^k \} \subset \R^n$ for each $k \in K$. Arguing similarly to the proof of Lemma~\ref{D1Aprops}, we obtain $\b^k = - \c^k$ for each $k \in K$. The relations \eqref{D2Arels1}, \eqref{D2Arels2}, and \eqref{D2Arels3} can be established in a very similar manner. Since $\e\T \diag({\mathsf Y}^{\x\v}) = 0$ and $\Yb \in \cD_{3n+1}$, we obtain \eqref{D2Arels4}. By using $\u + \v = \e$, $\b^k = - \c^k$ for each $k \in K$, and the decomposition above, it is easy to verify that $\diag({\mathsf Y}^{\x\u}) + \diag({\mathsf Y}^{\x\v}) = \x$, which establishes \eqref{D2Arels5} due to \eqref{D2Arels4}. Similarly, we obtain $\diag({\mathsf Y}^{\v\v}) = \e - 2 \, \u + \diag({\mathsf Y}^{\u\u}) = \e - \u = \v$ since $\diag({\mathsf Y}^{\u\u}) = \u$, which yields \eqref{D2Arels6}. Finally, we obtain \eqref{D2Arels7} from $\diag({\mathsf Y}^{\u\u}) = \u$, $ \diag({\mathsf Y}^{\u\u}) + 2 \, \diag({\mathsf Y}^{\u\v}) + \diag({\mathsf Y}^{\v\v}) = \e$, $\u + \v = \e$, and \eqref{D2Arels6}. This completes the proof.
\end{proof}

Once again, we can utilize Lemma~\ref{D2Aprops} to obtain the following optimization problem:

\[
\begin{array}{lllrcl}
 \tag{D2B($\rho$)}\label{D2B}
 &\nu(\textrm{D2B($\rho$)}) := & \min\limits_{\Sb \in \cS^{2n + 1}} & \du{\Qb}{{\mathsf S}^{\x\x}}  & & \\
 && \textrm{s.t.}  & \e\T  \x & = & 1 \\ 
  && & \e\T  \u & = & \rho \\
 && & \du{\Eb}{{\mathsf S}^{\x\x}} & = & 1 \\ 
 && & \du{\Eb}{{\mathsf S}^{\u\u}} & = & \rho^2 \\ 
 && & \diag({\mathsf S}^{\u\u}) & = & \u \\
  && & \diag({\mathsf S}^{\x\u}) & = & \x \\
 && & \x \e\T - \Sb^{\x \u} & \geq & \Ob \\
 && & \e \e\T - \e \u\T - \u \e\T + \Sb^{\u\u} & \geq & \Ob \\
 && & \u \e\T - \Sb^{\u\u} & \geq & \Ob \\  
 && & \Sb:=\begin{pmatrix}
    1 & \x\T & \u\T \\
    \x & {\mathsf S}^{\x\x} & {\mathsf S}^{\x\u} \\
    \u & ({\mathsf S}^{\x\u})\T & {\mathsf S}^{\u\u} \end{pmatrix} &  \in & \cD_{2n+1}.
    \end{array}
\]

Note that \eqref{D2B} has $2 n + 4$ linear equality constraints, $(5/2) n^2 +  (1/2) n $ linear inequality constraints, and a doubly nonnegative constraint in $\cS^{2n + 1}$. Once again, it is worth noticing that \eqref{D2B} contains all the RLT constraints obtained from the inequalities $\x \geq \o$, $\u \geq \o$, and $\u \leq \e$. 

Our next result establishes the equivalence between \eqref{D2A} and \eqref{D2B}.

\begin{proposition} \label{DNNequiv2}
\eqref{D2A} and \eqref{D2B} are equivalent to each other. Therefore, $\nu(\textrm{D2A($\rho$)}) = \nu(\textrm{D2B($\rho$)})$.  
\end{proposition}
\begin{proof}
The proof is very similar to that of Proposition~\ref{DNNequiv1}. 
Let $\Sb \in \cS^{2n+1}$ be \eqref{D2B}-feasible. Let us define $\Yb \in \cS^{3n + 1}$ as follows:
\begin{equation} \label{defYfromS}
\Yb = \begin{pmatrix} 1 & \o\T & \o\T \\
\o & \Ib & \o \\ \o & \Ob & \Ib \\ \e & \Ob & - \Ib\end{pmatrix}  \begin{pmatrix}
    1 & \x\T & \u\T \\
    \x & {\mathsf S}^{\x\x} & {\mathsf S}^{\x\u} \\
    \u & ({\mathsf S}^{\x\u})\T & {\mathsf S}^{\u\u} 
    \end{pmatrix} \begin{pmatrix} 1 & \o\T & \o\T \\
\o & \Ib & \o \\ \o & \Ob & \Ib \\ \e & \Ob & - \Ib \end{pmatrix}\T.  
\end{equation}
Therefore,
    \begin{equation} \label{Y_identity}
    \Yb = \begin{pmatrix}
    1 & \x\T & \u\T & \e\T - \u\T  \\
    \x & {\mathsf S}^{\x\x} & {\mathsf S}^{\x\u} & \x \e\T - {\mathsf S}^{\x\u} \\
    \u & ({\mathsf S}^{\x\u})\T & {\mathsf S}^{\u\u} & \u \e\T - {\mathsf S}^{\u\u} \\
    \e - \u & \e\x\T - ({\mathsf S}^{\x\u})\T & \e \u\T - {\mathsf S}^{\u\u} & \e \e\T - \e \u\T - \u \e\T + {\mathsf S}^{\u\u} \\ 
    \end{pmatrix}.
\end{equation}
We claim that $\Yb \in \cS^{3n + 1}$ is (D2A($\rho$))-feasible. Since $\diag({\mathsf S}^{\u\u}) = \u$, and $\Sb^{\u\u} - \u \u\T \succeq 0$, we obtain $\u \leq \e$. Therefore, by \eqref{defYfromS} and \eqref{Y_identity}, we conclude that $\Yb \in \cD_{4n + 1}$. Since $\v = \e -\u$, $\Yb^{\x\x} = \Sb^{\x\x}$, $\Yb^{\x\u} = \Sb^{\x\u}$, and $\Yb^{\u\u} = \Sb^{\u\u}$, the first six sets of equality constraints of (D2A($\rho$)) are satisfied. Consider the seventh set of equality constraints of (D2A($\rho$)):
\[
\e\T \diag(\Yb^{\x\v}) = \e\T \diag(\x \e\T - {\mathsf S}^{\x\u}) = \e\T \left[\x - \diag( {\mathsf S}^{\x\u} )\right] = 0,
\]
where we used $\diag( {\mathsf S}^{\x\u} ) = \x$.
Finally, the last set of equality constraints of (D2A($\rho$)) can be directly verified using \eqref{Y_identity}. Therefore, $\Yb \in \cS^{3n + 1}$ is (D2A($\rho$))-feasible and achieves the same objective function value as $\Sb$ in \eqref{D2B}.

Conversely, let $\Yb \in \cS^{3n + 1}$ be (D2A($\rho$))-feasible. Let $\Sb \in \cS^{2n + 1}$ be given by the top left $3 \times 3$-block of $\Zb$, i.e., 
\begin{equation} \label{defSfromY}
\Sb = \begin{pmatrix}
    1 & \x\T & \u\T \\
    \x & {\mathsf Y}^{\x\x} & {\mathsf Y}^{\x\u} \\
    \u & ({\mathsf Y}^{\x\u})\T & {\mathsf Y}^{\u\u} 
    \end{pmatrix} \, .
\end{equation}
Clearly, $\Sb \in \cD_{2n + 1}$ and satisfies all linear equality and inequality constraints in \eqref{D2B} by Lemma~\ref{D2Aprops}. Furthermore, $\Sb$ has the same objective function value in \eqref{D2B} as that of $\Yb$ in \eqref{D2A}. It follows that $\nu(\textrm{D2A($\rho$)}) = \nu(\textrm{D2B($\rho$)})$.
\end{proof}

A comparison of \eqref{D2A} and \eqref{D2B} reveals that the former 
consists of $3 n + 5$ linear equality constraints and {\bf a doubly nonnegative constraint} in $\cS^{3n+1}$ whereas the latter has $2n + 4$ linear equality constraints, $(5/2) n^2 +  (1/2) n $ linear inequality constraints, and {\bf a doubly nonnegative constraint} in $\cS^{2n + 1}$. Proposition~\ref{DNNequiv2} implies that the lower bound arising from the DNN relaxation of \eqref{CP2A} can be computed by solving a conic optimization problem in a smaller dimension.

\subsection{Comparing the two DNN relaxations}\label{comparDNN}

Now we compare the lower bounds arising from the two DNN relaxations of~\eqref{CP1A} and \eqref{CP2A}.

\begin{proposition} \label{D1AvsD2A}
\eqref{D1A} (or equivalently, \eqref{D1B}) is at least as tight as \eqref{D2A} (or equivalently, \eqref{D2B}), i.e., $$\nu(\textrm{D2A($\rho$)}) = \nu(\textrm{D2B($\rho$)}) \leq  \nu(\textrm{D1A($\rho$)}) = \nu(\textrm{D1B($\rho$)}) \leq \ell_\rho(\Qb)\, .$$ 
\end{proposition}
\begin{proof}
Let $\Zb \in \cS^{4n + 1}$ be \eqref{D1A}-feasible. Let $\Yb \in \cS^{3n + 1}$ be given by the top left $4 \times 4$-block of $\Zb$, i.e., 
\[
\Yb = \begin{pmatrix}
    1 & \x\T & \v\T & \u\T\\
    \x & {\mathsf Z}^{\x\x} & {\mathsf Z}^{\x\v} & {\mathsf Z}^{\x\u}\\
    \u & ({\mathsf Z}^{\x\u})\T & {\mathsf Z}^{\u\u} & {\mathsf Z}^{\u\v}\\
    \v & ({\mathsf Z}^{\x\v})\T & ({\mathsf Z}^{\u\v})\T & {\mathsf Z}^{\v\v}\\
    \end{pmatrix}.
\]
We claim that $\Yb \in  \cS^{3n + 1}$ is \eqref{D2A}-feasible. Clearly, $\Yb \in \cD_{3n+1}$ and all constraints of \eqref{D2A} are satisfied by Lemma~\ref{D1Aprops}. Therefore, for every \eqref{D1A}-feasible solution, there exists a corresponding \eqref{D2A}-feasible solution with the same objective function value. By Propositions~\ref{DNNequiv1} and \ref{DNNequiv2}, we conclude that $\nu(\textrm{D2A($\rho$)}) = \nu(\textrm{D2B($\rho$)}) \leq  \nu(\textrm{D1A($\rho$)}) = \nu(\textrm{D1B($\rho$)}) \leq \ell_\rho(\Qb)$.    
\end{proof}

Our next example illustrates that we can have $\nu(\textrm{D2A($\rho$)}) <  \nu(\textrm{D1A($\rho$)})$. i.e., \eqref{D2A} can be strictly weaker than \eqref{D1A}.

\begin{example} \label{D2Aweaker}
Consider the following instance, where $n = 6$ and $\rho = 3$:
    \[\Qb = \begin{pmatrix}
        2.6947 & -0.2028 & -1.1144 & -2.4230 & -2.1633 & 0.7710 \\
        -0.2028 & 5.1998 & 0.5005 & -2.0941 & 0.7828 & -3.3611 \\
        -1.1144 & 0.5005 & 5.2918 & 1.4119 & -1.1526 & -1.9723 \\
        -2.4230 & -2.0941 & 1.4119 & 4.1140 & -0.2025 & 0.3132 \\
        -2.1633 & 0.7828 & -1.1526 & -0.2025 & 7.6645 & -0.0170 \\
        0.7710 & -3.3611 & -1.9723 & 0.3132 & -0.0170 & 3.3040
    \end{pmatrix}.
    \]
For this instance, we have $\nu(\textrm{D2A($\rho$)}) = \nu(\textrm{D2B($\rho$)}) \approx 0.1320 < \nu(\textrm{D1A($\rho$)}) = \nu(\textrm{D1B($\rho$)}) \approx 0.1333$. In addition, we observe that the computed optimal solution $\Sb^*$ of (\textrm{D2B($\rho$)}) is not feasible for (\textrm{D1B($\rho$)}). For instance, $(S^{\x\u}_{31})^* - (S^{\x\x}_{31})^*   \approx -0.00137< 0$.
\end{example}

\section{Generating nontrivial instances of sparse StQPs}\label{generation}

In this section, we discuss procedures for generating nontrivial instances of the sparse StQP problem \eqref{sStQP}. To this end, we say that an instance of \eqref{sStQP} is {\em nontrivial} if 
the sparsity constraint is violated by each optimal solution of the corresponding instance of StQP given by \eqref{StQP} (or, equivalently, \eqref{sStQP} with $\rho = n$). For a nontrivial instance of the sparse StQP, it follows that the sparsity constraint cuts off all optimal solutions of the corresponding StQP instance without the sparsity constraints.

We aim to construct a nontrivial instance of \eqref{sStQP} in two steps. First, we discuss how to construct an instance of \eqref{StQP} with a prespecified \emph{unique} optimal solution. We then choose the sparsity parameter $\rho$ to be strictly smaller than the sparsity of the designated unique optimal solution of \eqref{StQP}. It follows that the resulting instance of \eqref{sStQP} is nontrivial since the unique optimal solution of \eqref{StQP} is cut off by the sparsity constraint. 

In Section~\ref{Sec_stqp_unique_sol}, we consider generating an instance of \eqref{StQP} with a unique optimal solution. We then present simple algorithms for generating nontrivial instances of \eqref{sStQP} in Section~\ref{SecNontrivialsStQP}.

\subsection{Generating StQP instances with a unique optimal solution} \label{Sec_stqp_unique_sol}

Let us next focus on generating an instance of \eqref{StQP} with a unique optimal solution. To that end, we first recall that an StQP can be reformulated as the following copositive optimization problem~\cite{BomzeDKRQT00}:
\[
\ell_n(\Qb) = \min\limits_{\x \in \R^n} \left\{\x\T  \Qb \x: \x \in F_n\right\} = \min\limits_{\Xb \in \cS^n} \left\{\langle \Qb, \Xb \rangle: \langle \Eb, \Xb \rangle\ = 1, \quad \Xb \in {\cal CP}_n \right\} \, ,
\]
where $F_n$ is given by \eqref{def_F_n} and $\Eb = \e \e\T \in \cS^n$.

The doubly nonnegative (DNN) relaxation of an StQP is therefore given by
\begin{equation} \label{stqp_dnn}
\mu(\Qb) = \min\limits_{\Xb \in \cS^n} \left\{\langle \Qb, \Xb \rangle: \langle \Eb, \Xb \rangle\ = 1, \quad \Xb \in {\cal D}_n \right\} \,.   
\end{equation}

Note that $\mu(\Qb) \leq \ell_n(\Qb)$. For a given instance of StQP, we say that its DNN relaxation is \emph{exact} if $\mu(\Qb) = \ell_n(\Qb)$ and \emph{inexact} otherwise.

We will review the following useful results, which will play a fundamental role in our instance construction procedures.

\begin{theorem} \label{stqp_opt_dnn}
Let $\x \in F_n$, where $F_n$ is given by \eqref{def_F_n}. 
\begin{enumerate}
    \item[(i)] $\x$ is a globally optimal solution of \eqref{StQP} if and only if there exist $\Mb \in \mathbf{bd} ~{\cal COP}_n$ and $\lambda \in \R$ such that $\x \in \mathbf{V}^\Mb$ and $\Qb = \Mb + \lambda \, \Eb$, where $\mathbf{bd} ~{\cal COP}_n$ and $\mathbf{V}^\Mb$ are given by \eqref{cop_bd} and \eqref{def_M_zeros}, respectively. Furthermore, in this case, $\lambda$ is the optimal value of \eqref{StQP}, and the set of optimal solutions of \eqref{StQP} is given by $\mathbf{V}^\Mb$.
    \item[(ii)] $\x$ is a globally optimal solution of \eqref{StQP} and its doubly nonnegative relaxation is exact if and only if there exist $\Rb \succeq \Ob$, $\Nb \geq \Ob$, and $\lambda \in \R$ such that $\x\T \Nb \x = 0$ and $\Qb = (\Ib - \e \x\T) \Rb (\Ib - \x \e\T) + \Nb + \lambda \, \Eb$. 
    \item[(iii)] $\x$ is a globally optimal solution of \eqref{StQP} and its doubly nonnegative relaxation is inexact if and only if there exist $\Mb \in \mathbf{bd} ~{\cal COP}_n \backslash {\cal SPN}_n$ and $\lambda \in \R$ such that $\x \in \mathbf{V}^\Mb$ and $\Qb = \Mb + \lambda \, \Eb$, where $\mathbf{bd} ~{\cal COP}_n$ and $\mathbf{V}^\Mb$ are given by \eqref{cop_bd} and \eqref{def_M_zeros}, respectively. 
\end{enumerate}
\end{theorem}
\begin{proof}
The reader is referred to \cite{ref:Bomze1997} for (i), and to \cite{Goek22} for (ii) and (iii).    
\end{proof}    

In view of Theorem~\ref{stqp_opt_dnn}, we first give a simple recipe to construct an instance of \eqref{StQP} such that it has a unique optimal solution and its doubly nonnegative relaxation is exact.

\begin{proposition} \label{nontrivial_exact}
Let $\x \in F_n$, where $F_n$ is given by \eqref{def_F_n}. Let $\A = \{j \in \{1,\ldots,n\}: \x_j > 0\}$ and $\B = \{j \in \{1,\ldots,n\}: \x_j = 0\}$. Then, for any $\Rb \succ \Ob$, any $\Nb \in \cS^n$ is such that $\Nb_{\A\A} = \Ob$, $\Nb_{\A\B} \geq \Ob$, and $\Nb_{\B\B} \geq \Ob$, and any $\lambda \in \R$, if $\Qb = (\Ib - \e \x\T) \Rb (\Ib - \x \e\T) + \Nb + \lambda \, \Eb$, then $\x$ is the unique globally optimal solution of \eqref{StQP} and its DNN relaxation is exact.  
\end{proposition}
\begin{proof}
It follows from the hypothesis and Theorem~\ref{stqp_opt_dnn}(ii) that $\x$ is a globally optimal solution of \eqref{StQP} and its DNN relaxation is exact. We next show the uniqueness. For any $\y \in F_n$, where $F_n$ is given by \eqref{def_F_n}, we have
\[
\y\T \Qb \y = (\y - \x)\T \Rb (\y - \x) + \y\T \Nb \y + \lambda \geq \lambda = \x\T \Qb \x,
\]
where we used $\x \in F_n$, $\y \in F_n$, $\Rb \succ \Ob$, $\y \geq \o$, and $\Nb \geq \Ob$. Since $\Rb \succ \Ob$, the inequality above holds with equality if and only if $\y = \x$. We conclude that $\x$ is the unique globally optimal solution of \eqref{StQP}. 
\end{proof}

In Proposition~\ref{nontrivial_exact}, if one chooses $\Nb = \Ob$ and $\lambda \geq \o$, we obtain $\Qb \succeq \Ob$, which implies that the resulting instance of \eqref{StQP} is a convex optimization problem. On the other hand, by choosing $\Nb$ and $\lambda$ in such a way that $\Qb \not \succeq \Ob$, Proposition~\ref{nontrivial_exact} allows us to construct a nonconvex instance of StQP that admits an exact DNN relaxation. We will utilize both of these observations in our computational experiments. 

Next, we consider constructing an instance of StQP with a unique optimal solution but an inexact DNN relaxation. In contrast with the previous case, such a construction is more involved and requires additional care. Our next result presents such a procedure. 

\begin{proposition} \label{nontrivial_inexact}
Let $\x \in F_n$, where $F_n$ is given by \eqref{def_F_n}. Let $\A = \{j \in \{1,\ldots,n\}: \x_j > 0\}$ and $\B = \{j \in \{1,\ldots,n\}: \x_j = 0\}$. Let $\Qb = (\Ib - \e \x\T) \Rb (\Ib - \x \e\T) + \Nb + \lambda \, \Eb$, where $\lambda \in \R$, $\Rb \in \cS^n$ is such that $\Rb_{\A\A} \succeq \Ob$, $\Rb_{\B\B} \in {\cal COP}_{|\B|}$, $\Rb_{\A\B} = \Ob$, and $\Nb \in \cS^n$ is such that $\Nb_{\A\A} = \Ob$, $\Nb_{\A\B} \geq \Ob$, and $\Nb_{\B\B} \geq \Ob$.
\begin{itemize}
    \item[(i)] $\x$ is a globally optimal solution of \eqref{StQP}. Furthermore, if $\Rb_{\A\A} \succ \Ob$, then $\x$ is the unique globally optimal solution.
    \item[(ii)] If $\Rb_{\B\B} \in {\cal COP}_{|\B|} \backslash {\cal SPN}_{|\B|}$ and $\Nb = \Ob$, there exists an $\epsilon > 0$ such that, for all $\Rb_{\A\A} \succ \Ob$ with $\| \Rb_{\A\A} \| < \epsilon$, where $\|\cdot\|$ denotes the operator norm, 
    $\x$ is the unique globally optimal solution of \eqref{StQP} and its doubly nonnegative relaxation is inexact.
\end{itemize}
\end{proposition}
\begin{proof}
\begin{enumerate}
    \item[(i)] Let $\Mb = (\Ib - \e \x\T) \Rb (\Ib - \x \e\T) + \Nb$. By Theorem~\ref{stqp_opt_dnn}(i), it suffices to show that $\Mb \in \mathbf{bd} ~{\cal COP}_n$ and $\x \in \mathbf{V}^\Mb$, where $\mathbf{V}^\Mb$ is given by \eqref{def_M_zeros}. For any $\y \in F_n$, 
\begin{equation} \label{yTMy}
\y\T \Mb \y = (\y - \x)\T \Rb (\y - \x) + \y\T \Nb \y = (\y_{\A} - \x_{\A})\T \Rb_{\A\A} (\y_{\A} - \x_{\A}) +  \y_{\B}\T \Rb_{\B\B} \y_{\B} + \y\T \Nb \y \geq 0 = \x\T \Mb \x,  
\end{equation}
where we used $\y \in F_n$ in the first equality, $\Rb_{\A\B} = \Ob$ and $\x_{\B} = \o$ in the second equality, $\Rb_{\A\A} \succeq \Ob$, $\Rb_{\B\B} \in {\cal COP}_{|\B|}$, $\Nb \geq \Ob$, and $\y \geq \o$ to derive the inequality, and $\x \in F_n$ in the last equality. We conclude that $\Mb \in \mathbf{bd} ~{\cal COP}_n$ and $\x \in \mathbf{V}^\Mb$. By Theorem~\ref{stqp_opt_dnn}(i), $\x$ is a globally optimal solution of \eqref{StQP}. If, in addition, $\Rb_{\A\A} \succ \Ob$, then it follows from \eqref{yTMy} that $\y\T \Mb \y = 0$ if and only if $(\y_{\A} - \x_{\A})\T \Rb_{\A\A} (\y_{\A} - \x_{\A}) =  \y_{\B}\T \Rb_{\B\B} \y_{\B} = \y\T \Nb \y = 0$. Since $\Rb_{\A\A} \succ \Ob$, we conclude that $\y_{\A} = \x_{\A}$. Since $\x \in F_n$, $\x_{\B} = \o$, and $\y_{\A} = \x_{\A}$, we obtain $1 = \e\T \x = \e_{\A} \T \x_{\A} + \e_{\B} \T \x_{\B} = \e_{\A} \T \x_{\A} = \e_{\A} \T \y_{\A}$, which implies that $\y_{\B} = \o$ since $\y \in F_n$. 
It follows that $\y = \x$. Therefore, $\mathbf{V}^\Mb = \{ \x \}$. By Theorem~\ref{stqp_opt_dnn}(i), $\x$ is the unique globally optimal solution of \eqref{StQP}. 

\item[(ii)] Let $\Mb = (\Ib - \e \x\T) \Rb (\Ib - \x \e\T)$. By Theorem~\ref{stqp_opt_dnn}(ii), it suffices to show that $\Mb \in \mathbf{bd} ~{\cal COP}_n \backslash {\cal SPN}_n$ and $\x \in \mathbf{V}^\Mb$. By (i), $\Mb \in \mathbf{bd} ~{\cal COP}_n$ and $\x \in \mathbf{V}^\Mb$. Since $\Rb_{\B\B} \in {\cal COP}_{|\B|} \backslash {\cal SPN}_{|\B|}$, ${\cal CP}_{|\B|}$ and $\cD_{|\B|}$ are the dual cones of ${\cal COP}_{|\B|}$ and ${\cal SPN}_{|\B|}$, respectively, it follows from \eqref{cone_inclusion} that there exists $\Tb \in \cD_{|\B|} \backslash {\cal CP}_{|\B|}$ such that $\langle \Rb_{\B\B}, \Tb \rangle = - \delta < 0$. Let $\Ub \in \cS^n$ be such that $\Ub_{\A\A} = \Ob$, $\Ub_{\A\B} = \Ob$, and $\Ub_{\B\B} = \Tb$. Clearly, $\Ub \in \cD_n$. Furthermore,
\begin{eqnarray*}
\langle \Mb, \Ub \rangle & = & \langle \Mb_{\B\B}, \Tb \rangle \\
 & = & \left\langle \left( \x_{\A}\T \Rb_{\A\A} \x_{\A} \right) \e_{\B} \e_{\B}\T + \Rb_{\B\B}, \Tb \right\rangle \\
 & = & \left( \x_{\A}\T \Rb_{\A\A} \x_{\A} \right) \left( \e_{\B}\T \Tb \e_{\B} \right) - \delta \\
 & \leq & \| \Rb_{\A\A} \| \mu - \delta,
\end{eqnarray*}
where $\mu = \e_{\B}\T \Tb \e_{\B} > 0$ and we used $\| \x_{\A} \| \leq 1$ to derive the last inequality. We conclude that $\langle \Mb, \Ub \rangle < 0$ provided that $\| \Rb_{\A\A} \| < \epsilon$, where $\epsilon = \delta / \mu > 0$. Therefore, for any $\Rb_{\A\A} \succeq \Ob$ such that $\| \Rb_{\A\A} \| < \epsilon$, we obtain $\Mb \in \mathbf{bd} ~{\cal COP}_n \backslash {\cal SPN}_n$ since $\Ub \in \cD_n$. By Theorem~\ref{stqp_opt_dnn}(ii), it follows that the DNN relaxation of \eqref{StQP} is inexact. If, in addition, $\Rb_{\A\A} \succ \Ob$, then $\x$ is the unique globally optimal solution of \eqref{StQP} by part (i). This concludes the proof. 
\end{enumerate}    
\end{proof}

\subsection{Two algorithms for generating nontrivial sparse StQP instances} \label{SecNontrivialsStQP}

In this section, we present two algorithms for generating nontrivial instances of \eqref{sStQP}. Given $\x \in F_n$ with $\|\x\|_0 \geq 2$, both algorithms initially generate an instance of \eqref{StQP} such that $\x$ is the unique optimal solution. Then, a nontrivial instance of \eqref{sStQP} is constructed by choosing $\rho \in \{1,\ldots,\|\x\|_0 - 1\}$. 

The details of our algorithms are presented in Algorithm~\ref{Alg1} and Algorithm~\ref{Alg2}. The two algorithms differ only in terms of whether the StQP instance generated in the first stage admits an exact or inexact DNN relaxation.

\begin{algorithm}[!htb]
\begin{algorithmic}[1]
\Require $n \geq 2$, $\x \in F_n$ such that $\|\x\|_0 \geq 2$, $\lambda \in \R$
\Ensure A nontrivial instance of \eqref{sStQP}
\State $\A \gets \{j \in \{1,\ldots,n\}: \x_j > 0\}$, $\B \gets \{j \in \{1,\ldots,n\}: \x_j = 0\}$
\State Choose an arbitrary $\Rb \in \cS^n$ such that $\Rb \succ \Ob$.
\State Choose an arbitrary $\Nb \in \cS^n$ such that $\Nb_{\A\A} = \Ob$, $\Nb_{\A\B} \geq \Ob$, and $\Nb_{\B\B} \geq \Ob$. 
\State $\Qb \gets (\Ib - \e \x\T) \Rb (\Ib - \x \e\T) + \Nb + \lambda \, \Eb$
\State Choose an arbitrary $\rho \in \{1,\ldots,\|\x\|_0 - 1\}$. 
\end{algorithmic}
\caption{Nontrivial instance of \eqref{sStQP} with \eqref{StQP} admitting an exact DNN relaxation}
\label{Alg1}
\end{algorithm}

Our first algorithm, given by Algorithm~\ref{Alg1}, requires as input $n \geq 2$, $\x \in F_n$ such that $\|\x\|_0 \geq 2$, and $\lambda \in \R$, and generates an instance of \eqref{sStQP}. By Proposition~\ref{nontrivial_exact}, the instance generated by Algorithm~\ref{Alg1} is a nontrivial instance of \eqref{sStQP} such that $\x$ is the unique optimal solution of  \eqref{StQP}, which admits an exact DNN relaxation. We remark that each step of Algorithm~\ref{Alg1} can be implemented in a fairly straightforward way. 

\begin{algorithm}[!htb]
\begin{algorithmic}[1]
\Require $n \geq 7$, $\x \in F_n$ such that $2 \leq \|\x\|_0 \leq n - 5$, $\lambda \in \R$
\Ensure A nontrivial instance of \eqref{sStQP}
\State $\A \gets \{j \in \{1,\ldots,n\}: \x_j > 0\}$, $\B \gets \{j \in \{1,\ldots,n\}: \x_j = 0\}$
\State Construct $\Rb \in \cS^n$ using the following steps.
\State $\Rb_{\A\B} \gets \Ob$
\State Choose an arbitrary $\Rb_{\B\B} \in {\cal COP}_{|\B|} \backslash {\cal SPN}_{|\B|}$.
\State Choose an arbitrary $\Tb \in \cD_{|\B|} \backslash {\cal CP}_{|\B|}$ such that $\langle \Rb_{\B\B}, \Tb \rangle < 0$.
\State $\delta \gets -\langle \Rb_{\B\B}, \Tb \rangle$, $\mu \gets \e_{\B}\T \Tb \e_{\B}$, $\epsilon \gets \frac{\delta}{\mu}$
\State Choose any $\Rb_{\A\A} \succ \Ob$ such that $\| \Rb_{\A\A} \| < \epsilon$. 
\State $\Qb \gets (\Ib - \e \x\T) \Rb (\Ib - \x \e\T) + \lambda \, \Eb$
\State Choose an arbitrary $\rho \in \{1,\ldots,\|\x\|_0 - 1\}$. 
\end{algorithmic}
\caption{Nontrivial instance of \eqref{sStQP} with \eqref{StQP} admitting an inexact DNN relaxation}
\label{Alg2}
\end{algorithm}

Our second algorithm, given by Algorithm~\ref{Alg2}, takes as input $n \geq 7$, $\x \in F_n$ such that $2 \leq \|\x\|_0 \leq n - 5$, and $\lambda \in \R$, and outputs an instance of \eqref{sStQP}. Similarly, it follows from Proposition~\ref{nontrivial_inexact} that the instance generated by Algorithm~\ref{Alg2} is a nontrivial instance of \eqref{sStQP} such that $\x$ is the unique optimal solution of \eqref{StQP}, which admits an inexact DNN relaxation. 

In contrast with Algorithm~\ref{Alg1}, the practical implementation of Algorithm~\ref{Alg2} merits further discussion. Step 4 of Algorithm~\ref{Alg2} requires $\Rb_{\B\B} \in {\cal COP}_{|\B|} \backslash {\cal SPN}_{|\B|}$. Such matrices are called \emph{exceptional}. Recall that ${\cal SPN}_n \subseteq {\cal COP}_n$, and ${\cal SPN}_n = {\cal COP}_n$ if and only if $n \leq 4$~(see \cite{ref:Diananda}). Therefore, such an exceptional matrix exists if and only if $|\B| \geq 5$, which is ensured by the choices of the input parameters of Algorithm~\ref{Alg2}. A well-known exceptional $5 \times 5$ matrix is the Horn matrix given by
\begin{equation} \label{def_horn}
\Hb = \begin{pmatrix} 1 & -1 & 1 & 1 & -1 \\-1 & 1 & -1 & 1 & 1 \\1 & -1 & 1 & -1 & 1 \\1 & 1 & -1 & 	1 & -1 \\ -1 & 1 & 1 & -1 & 1 \end{pmatrix} \in {\cal COP}_5 \backslash {\cal SPN}_5.
\end{equation}
Furthermore, $\Hb$ is an extreme ray of ${\cal COP}_5$~\cite{hall1963copositive}.

We next present a practical procedure for constructing an exceptional matrix in higher dimensions. Our procedure is inspired by a similar procedure outlined in \cite[Section 6.2]{Goek22}. 

Let $n \geq 7$ and let $\x \in F_n$ satisfy $\| \x \|_0 \leq n - 5$, which ensures that $|\B| \geq 5$. We now construct $\Rb_{\B\B} \in {\cal COP}_{|\B|} \backslash {\cal SPN}_{|\B|}$ as follows. 
Choose an arbitrary $\Bb \in {\cal COP}_{|\B| - 5}$, which can, for instance, be ensured by choosing $\Bb \in {\cal SPN}_{|\B| - 5}$ by \eqref{cone_inclusion} (or even $\Bb \geq \Ob$ or $\Bb \succeq \Ob$), and an arbitrary $\Cb \in \R^{(|\B| - 5) \times |\B|}$ such that $\Cb \geq \Ob$. Let us define 
\begin{equation} \label{def_R}
\Rb_{\B\B} = \begin{bmatrix} \Bb & \Cb \\ \Cb\T & \Hb \end{bmatrix} \in \cS^{|\B|}.
\end{equation}
It is easy to see that $\Rb_{\B\B} \in {\cal COP}_{|\B|}$ (see also~\cite[Lemma 3.4\,(a)]{shaked2016spn}). Since $\Hb \in {\cal COP}_5 \backslash {\cal SPN}_5$, there exists $\Fb \in  \cD_5 \backslash {\cal CP}_5$ such that $\langle \Hb, \Fb \rangle = - \delta < 0$. Let 
\begin{equation} \label{def_T}
\Tb = \begin{pmatrix}
    \Ob & \Ob \\
    \Ob & \Fb
\end{pmatrix} \, .
\end{equation}
Note that $\Tb \in {\cal D}_{|\B|}$ and $\langle \Rb_{\B\B}, \Tb \rangle = \langle \Hb, \Fb \rangle = - \delta < 0$, which implies that $\Rb_{\B\B} \not \in {\cal SPN}_{|\B|}$. We conclude that $\Rb_{\B\B} \in {\cal COP}_{|\B|} \backslash {\cal SPN}_{|\B|}$. We then define $\mu = \e_\B\T \Tb \e_\B > 0$ and $\epsilon = \delta / \mu$.  

While any choice of $\Fb \in  \cD_5 \backslash {\cal CP}_5$ such that $\langle \Hb, \Fb \rangle = - \delta < 0$ works in the above construction, it may be preferable in practice to choose $\Fb$ such that the corresponding value of $\epsilon$ is as large as possible. This can be easily achieved in practice by solving the following small semidefinite programming (SDP) problem:
\begin{equation}
    \label{separhorn}
\epsilon := \max\limits_{\Fb \in {\cal D}_5 \backslash \{\Ob\}} \frac{- \langle \Hb, \Fb \rangle}{\e_\B\T \Fb \e_\B} = \max \{ - \langle \Hb, \Fb \rangle: \e_\B\T\Fb\e_\B = 1\, ,\, \Fb\in {\cal D}_{5}\}\, ,
\end{equation}
where the second equality follows from the homogeneity of the objective function. An approximate optimal solution of \eqref{separhorn} is given by 
\begin{equation} \label{def_F}
\Fb = \frac{1}{78.2} \begin{pmatrix} 7 & 4.32 & 0 & 0 & 4.32 \\4.32 & 7 & 4.32 & 0 & 0 \\0 & 4.32 & 7 & 4.32 & 0 \\0 & 0 & 4.32 & 7 & 4.32 \\ 4.32 & 0 & 0 & 4.32 & 7 \end{pmatrix} \, ,
\end{equation}
which yields $\epsilon \approx 0.1049$.

\begin{remark}
The above construction works for any exceptional $\Rb_{\B\B}$ once we can control $\Tb$, the normal separating $\Rb_{\B\B}$ (corresponding to $\Hb$ above) from ${\cal SPN}_{|\B|}$,
in the sense to solve the (effectively small) SDP problem
\begin{equation}
    \label{separsdp}
\epsilon:= \max\limits_{\Tb \in {\cal D}_{|\B|} \backslash \{\Ob\}} \frac{- \langle \Rb_{\B\B}, \Tb \rangle}{\e_\B\T \Tb \e_\B} = \max \{ - \langle \Rb_{\B\B}, \Tb \rangle: \e_\B\T\Tb\e_\B = 1\, ,\, \Tb \in {\cal D}_{|\B|}\}\, .
\end{equation}
Observe that any extremal ray of ${\cal COP}_{|\B|}$ different from a rank-one matrix or a symmetric componentwise nonnegative matrix with exactly two positive entries -- which constitute the extremal rays of the cone of PSD matrices and the cone of componentwise nonnegative matrices -- is necessarily exceptional, by extremality. While for larger $|\B|$, these extremal rays are (yet) unknown, we can imitate the above extension strategy and build upon the known ones for $|\B| \in \{5,6\}$, e.g., using  Hildebrand matrices~\cite{afonin2021extreme,hildebrand2012extreme}. All these used in $\Rb_{\B\B}$ would render~\eqref{separsdp} feasible (recommended only for extremal rays of moderate order to keep this effort low), and proceeding as above with $\Hb$ and $\Fb$, this approach provides more nontrivial examples. Finally, note that extremality is used in this context only to ensure exceptionality; any non-extremal but exceptional matrix would work as well. 

\end{remark}
 
\subsection{Set of instances} \label{CRInst}

 To assess the impact of the choice of the instance set on the solution time of each exact model and each convex relaxation as well as on the quality of the lower bound, we conduct extensive experiments on a variety of carefully constructed instances of \eqref{sStQP}. 
 Let us now describe the set of instances in detail. 

Using Algorithm~\ref{Alg1} and Algorithm~\ref{Alg2}, we generated three sets of nontrivial instances of \eqref{sStQP} with different characteristics.

\begin{enumerate}
    \item[(i)] \emph{PSD Instances: } This set of instances of \eqref{sStQP} is constructed using Algorithm~\ref{Alg1} in such a way that $\Qb \succeq \Ob$, i.e., the corresponding instance of \eqref{StQP} is a convex optimization problem. For instance, such instances may arise in portfolio optimization problems. We remark that the resulting instances of \eqref{sStQP} are, in general, still NP-hard. In Step 2 of Algorithm~\ref{Alg1}, we construct $\Rb \succ \Ob$ so that all of its eigenvalues are uniformly distributed in $(0,3)$. By choosing $\Nb = \Ob$ and $\lambda = 0$, we ensure that $\Qb = (\Ib - \e \x\T) \Rb (\Ib - \x \e\T) \succeq \Ob$. Since $\x\T \Qb \x = 0$, $\Qb$ is PSD but not positive definite. By Proposition~\ref{nontrivial_exact}, each PSD instance satisfies
    \begin{equation} \label{psd_inst_rels}
    \ell_{\rho}(\Qb) > \ell_n (\Qb) = \mu(\Qb) = 0,
    \end{equation}
where $\mu(\Qb)$ is defined as in \eqref{stqp_dnn}.  

\item[(ii)] \emph{SPN Instances: }This set of instances of \eqref{sStQP} is generated using Algorithm~\ref{Alg1} in a similar manner to PSD instances, i.e., we set $\lambda = 0$ and use the same procedure to construct $\Rb$. However, instead of choosing $\Nb = \Ob$ as in PSD instances, each entry of $\Nb_{\A\B} \in \R^{|\A|\times|\B|}$ and $\Nb_{\B\B} \in \cS^{|\B|}$ is uniformly generated in $(0,3)$. In contrast with PSD instances, $\Qb = (\Ib - \e \x\T) \Rb (\Ib - \x \e\T) + \Nb$ is, in general, not PSD. In fact, we numerically verified that $\Qb$ was an indefinite matrix for each SPN instance. By Proposition~\ref{nontrivial_exact}, each SPN instance satisfies
    \begin{equation} \label{spn_inst_rels}
    \ell_{\rho}(\Qb) > \ell_n (\Qb) = \mu(\Qb) = 0.
    \end{equation}

\item[(iii)] \emph{COP Instances: }This set of instances of \eqref{sStQP} is generated by Algorithm~\ref{Alg2}. In Step 2, we construct $\Rb \in \cS^n$ using the procedure outlined in Section~\ref{SecNontrivialsStQP}. In particular, $\Rb_{\B\B}$ is constructed as in \eqref{def_R}, where $\Bb \succ \Ob$ is generated with eigenvalues uniformly distributed in $(0,3)$, and each entry of $\Cb \geq \Ob$ is uniformly distributed in $(0,1)$. Our choices are, in part, motivated by the observation that the eigenvalues of the Horn matrix $\Hb$ given by \eqref{def_horn} lie in $[-1.236,3.236]$ so that the eigenvalues of $\Bb \succ \Ob$ have similar magnitude. The separating matrix $\Tb$ is constructed as in \eqref{def_T}, where $\Fb$ is given by \eqref{def_F}, which yields $\epsilon \approx 0.1049$. Finally, $\Rb_{\A\A} \succ \Ob$ is constructed with eigenvalues uniformly distributed in $(0,0.99 \epsilon)$, ensuring that $\| \Rb_{\A\A} \| < \epsilon$. We then choose $\lambda = 0$ and set $\Qb = (\Ib - \e \x\T) \Rb (\Ib - \x \e\T)$. By Proposition~\ref{nontrivial_inexact}, this procedure ensures that $\Qb \not \succeq \Ob$ and each COP instance satisfies
\begin{equation} \label{cop_inst_rels}
    \ell_{\rho}(\Qb) > \ell_n (\Qb) = 0 > \mu(\Qb).
    \end{equation}
\end{enumerate}

In summary, each of the three sets of instances of \eqref{sStQP} is comprised of nontrivial instances that exhibit different characteristics. While each PSD instance and each SPN instance has a corresponding StQP instance that admits an exact doubly nonnegative (DNN) relaxation, they differ in terms of the convexity of the objective function. On the other hand, each COP instance has a nonconvex objective function with the additional property that the DNN relaxation of the corresponding StQP instance is inexact. 

\section{Computational results} \label{CompRes}

In this section, we report and discuss the results of our computational experiments. After describing 
the experimental setup in Section~\ref{CRSetup}, 
we report our results in detail in Sections~\ref{CRSolTime} and \ref{CRQLB}.

\subsection{Experimental setup} \label{CRSetup}

The set of instances was already described in Section~\ref{CRInst}. Here, we outline the details of our experimental setup. 

In our experiments, we chose $n \in \{25,50\}$. For each choice of $n$, we identified three sparsity levels for the designated optimal solution $\x \in F_n$ of \eqref{StQP}, denoted by $\rho_0 = \|\x\|_0$. In particular, we considered $\rho_0 \in \{\lfloor 0.25 n \rceil, \lfloor 0.5 n \rceil, \lfloor 0.75 n \rceil\}$, where $\lfloor \cdot \rceil$ denotes the nearest integer function. Finally, for each choice of $\rho_0$, we chose $\rho \in \{\lfloor 0.25 \rho_0 \rceil, \lfloor 0.5 \rho_0 \rceil, \lfloor 0.75 \rho_0 \rceil\}$. The parameters are summarized in Table~\ref{tab1}.

\begin{table}[!hbt] 
\begin{center}
\begin{tabular}{ |c|c|c| } 
\hline
$n$ & $\rho_0$ & $\rho$ \\
\hline
\hline
\multirow{3}{1em}{25} & 6 & \{2,3,4\} \\ 
& 12 & \{3,6,9\} \\  
& 19 & \{5,10,14\} \\ 
\hline
\multirow{3}{1em}{50} & 12 & \{3,6,9\} \\ 
& 25 & \{6,12,19\} \\  
& 38 & \{10,19,28\} \\ 
\hline
\end{tabular}
\end{center}
\caption{Choices of the parameters $n, \rho_0$, and $\rho$}
\label{tab1}
\end{table}

For each choice of $n$, $\rho_0$, and $\rho$ presented in Table~\ref{tab1}, we generated 25 instances from each of the three sets of instances described in Section~\ref{CRInst}, which gave rise to a total of 1,350 nontrivial instances of \eqref{sStQP} comprising of 450 PSD instances, 450 SPN instances, and 450 COP instances. For each instance of \eqref{sStQP}, we solved \eqref{P1}, \eqref{P2}, \eqref{D1A}, \eqref{D1B}, \eqref{D2A}, and \eqref{D2B}.

Algorithm~\ref{Alg1} and Algorithm~\ref{Alg2} were implemented in Julia version 1.8.5~\cite{BEKS17}. Each instance of \eqref{P1} and \eqref{P2} was solved by {\tt Gurobi} version 11.0.0~\cite{Gurobi} via the modeling language{\tt JuMP} v1.19.0~\cite{Lubin2023}. We solved \eqref{D1A}, \eqref{D1B}, \eqref{D2A}, and \eqref{D2B}
by {\tt MOSEK} version 10.1.24~\cite{MOSEK} via {\tt JuMP}.  
Our computational
experiments were carried out on a 64-bit HP workstation with 24 threads (2 sockets, 6 cores
per socket, 2 threads per core) running {\tt Ubuntu Linux} with 96 GB of RAM and Intel Xeon
CPU E5-2667 processors with a clock speed of 2.90 GHz. We imposed a time limit of 600 seconds on each solve. The default settings were used for all of the other parameters of {\tt Gurobi} and {\tt MOSEK}.

\subsection{Solution time} \label{CRSolTime}

Now we focus on the solution time of each of the six models \eqref{P1}, \eqref{P2}, \eqref{D1A}, \eqref{D1B}, \eqref{D2A}, and \eqref{D2B}.

\begin{figure}[!hbt]
    \centering
    \begin{subfigure}{0.495\linewidth}
		\includegraphics[width=\linewidth]{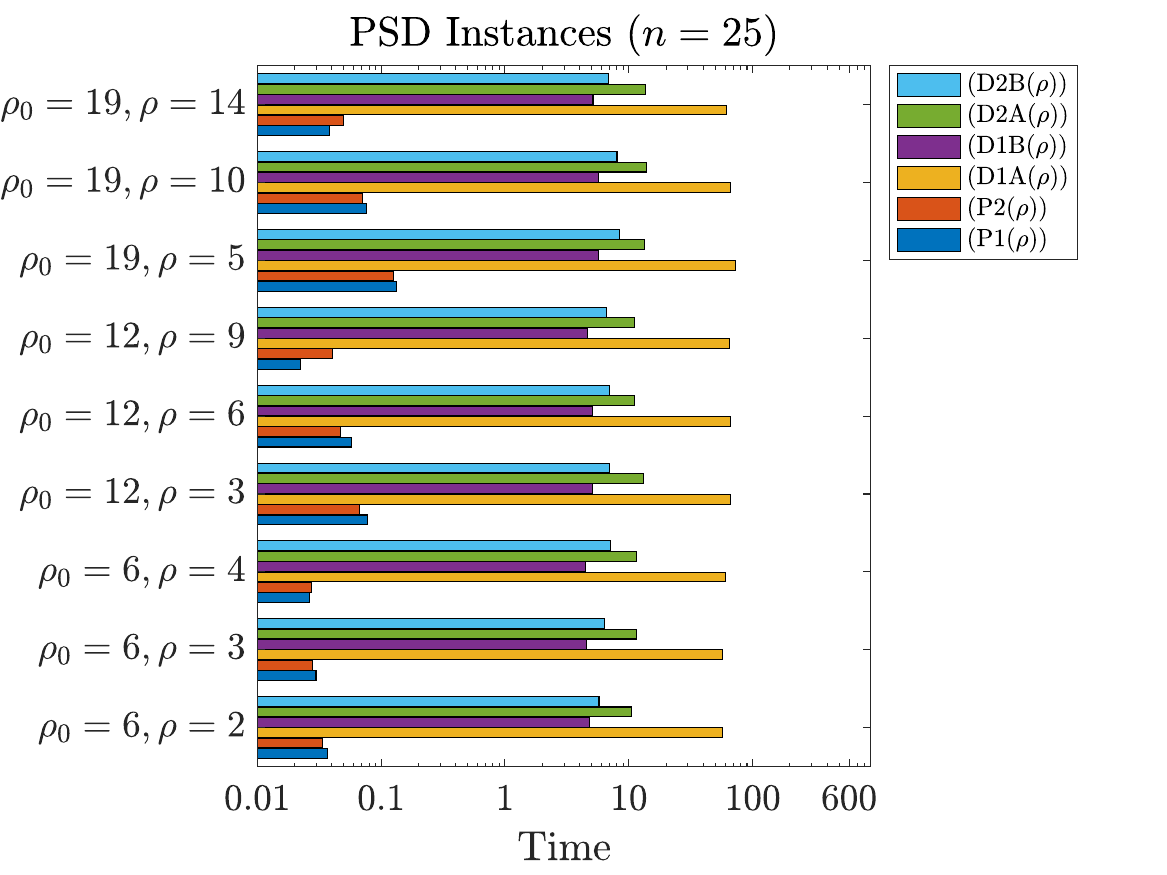}
		\caption{PSD Instances ($n = 25)$}
		\label{fig1apsd}
	\end{subfigure}
	\begin{subfigure}{0.495\linewidth}
		\includegraphics[width=\linewidth]{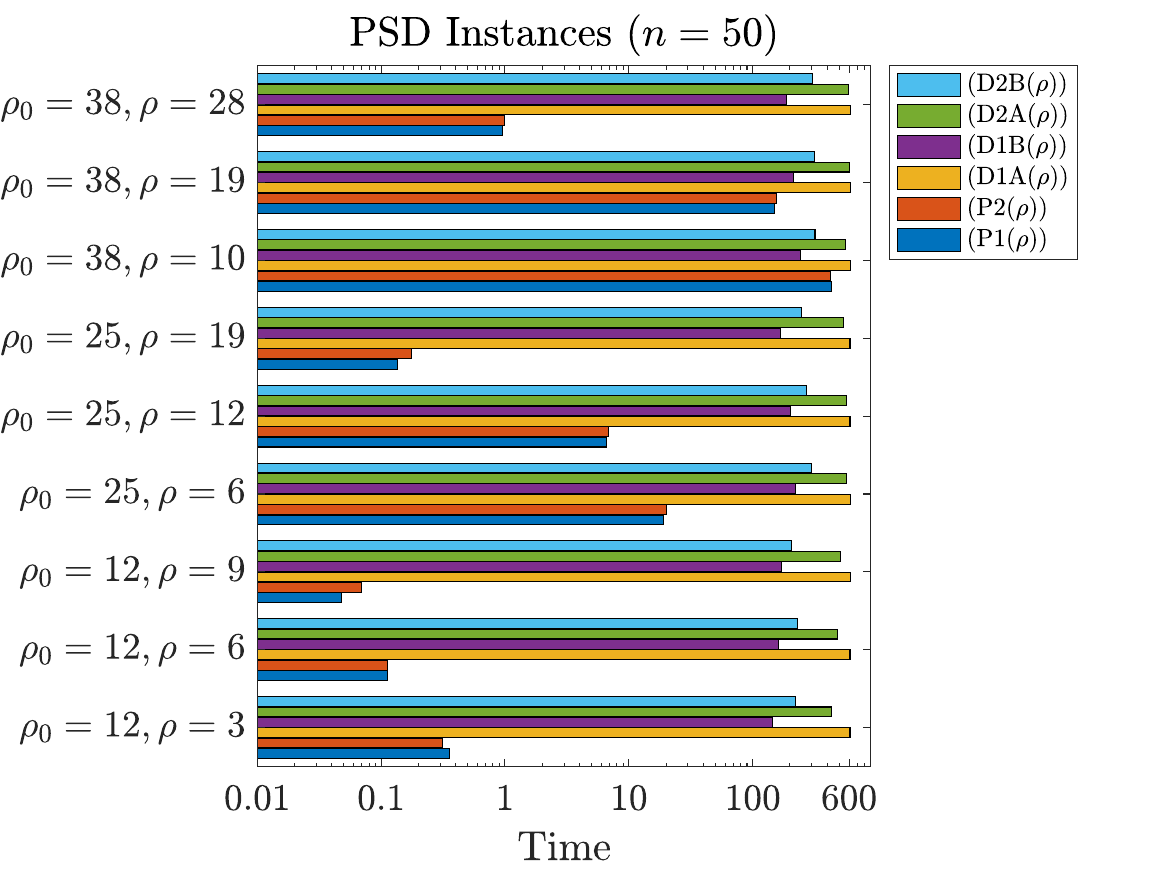}
		\caption{PSD Instances ($n = 50)$}
		\label{fig1bpsd}
	\end{subfigure}
 \\
 \begin{subfigure}{0.495\linewidth}
		\includegraphics[width=\linewidth]{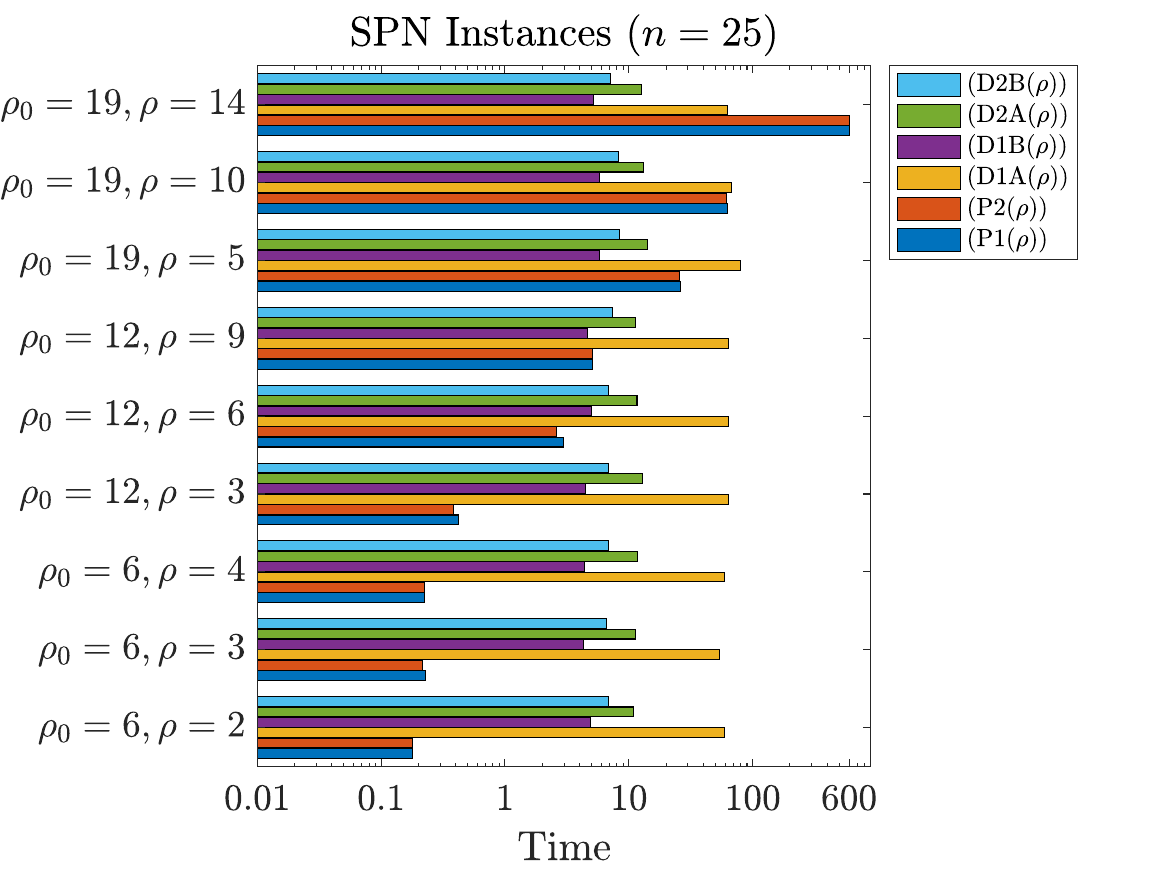}
		\caption{SPN Instances ($n = 25)$}
		\label{fig1aspn}
	\end{subfigure}
	\begin{subfigure}{0.495\linewidth}
		\includegraphics[width=\linewidth]{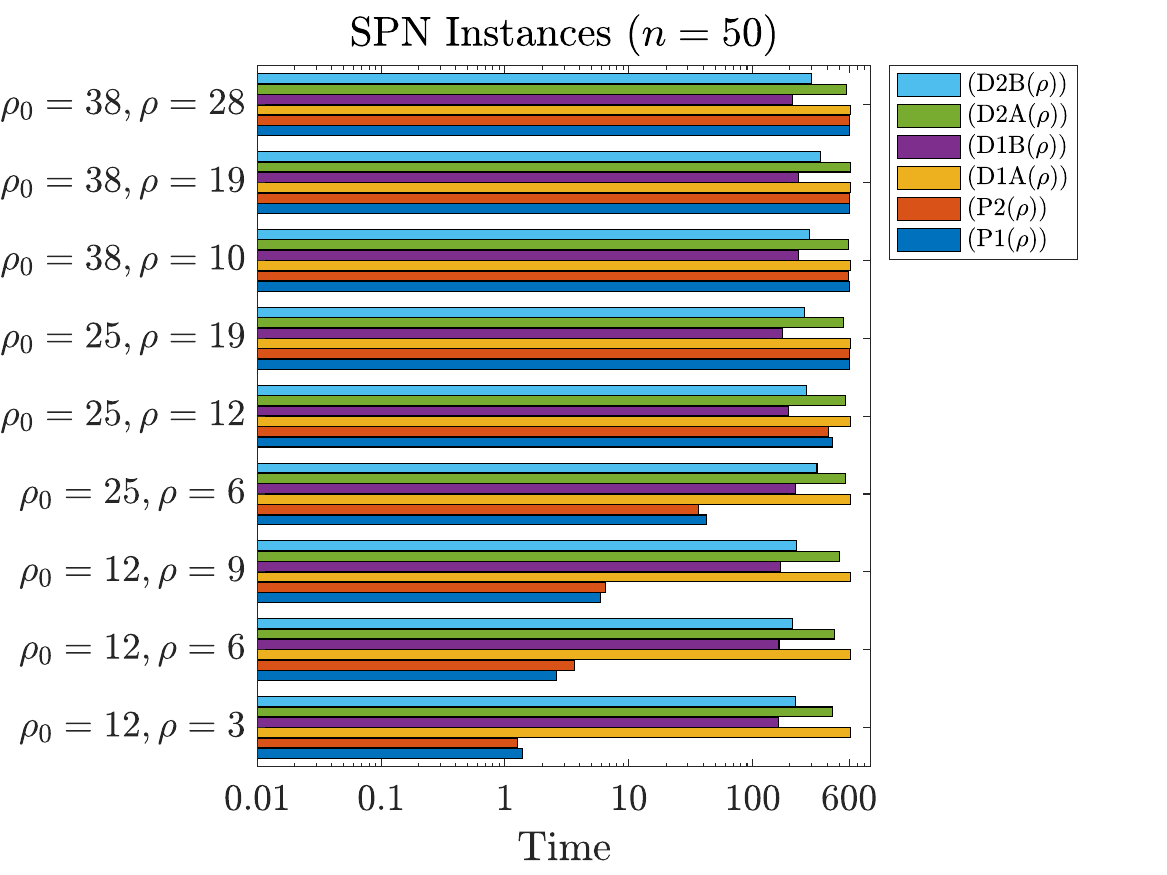}
		\caption{SPN Instances ($n = 50)$}
		\label{fig1bspn}
	\end{subfigure}
 \\
 \begin{subfigure}{0.495\linewidth}
		\includegraphics[width=\linewidth]{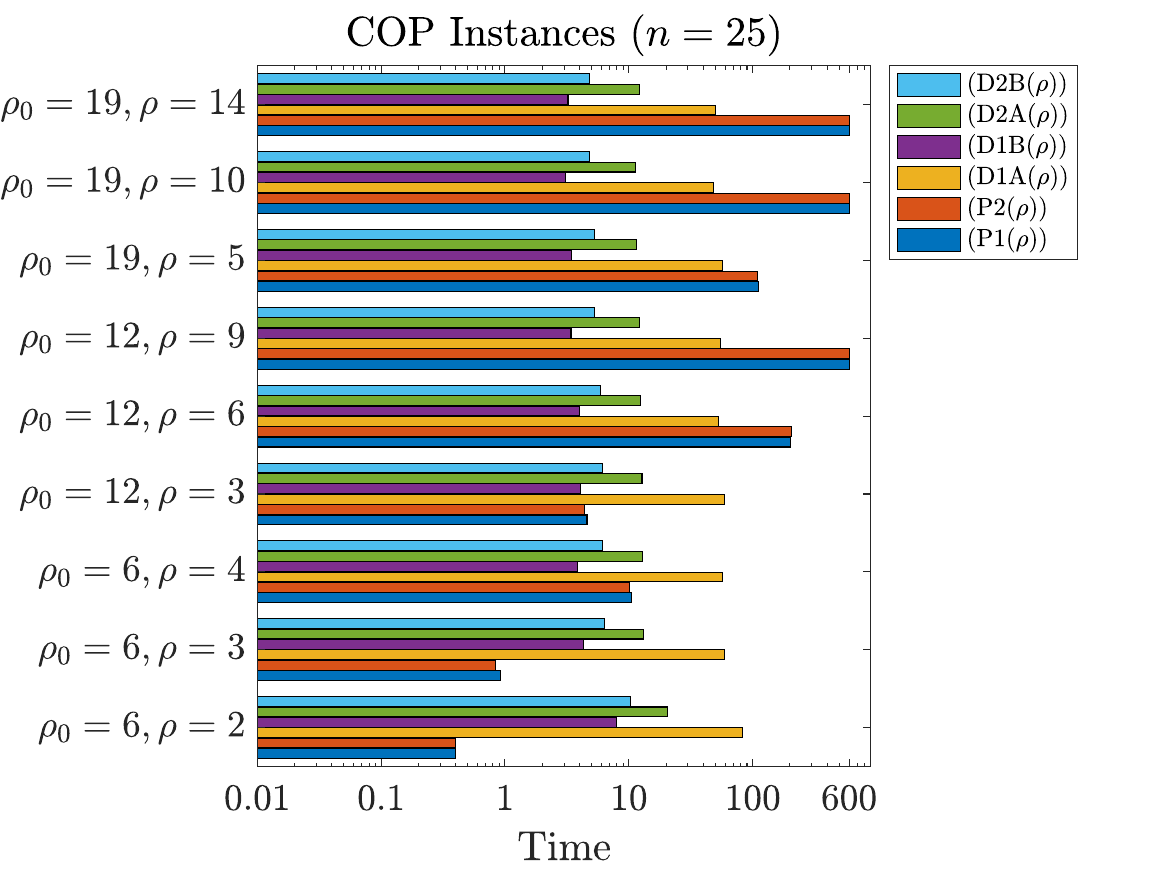}
		\caption{COP Instances ($n = 25)$}
		\label{fig1acop}
	\end{subfigure}
	\begin{subfigure}{0.495\linewidth}
		\includegraphics[width=\linewidth]{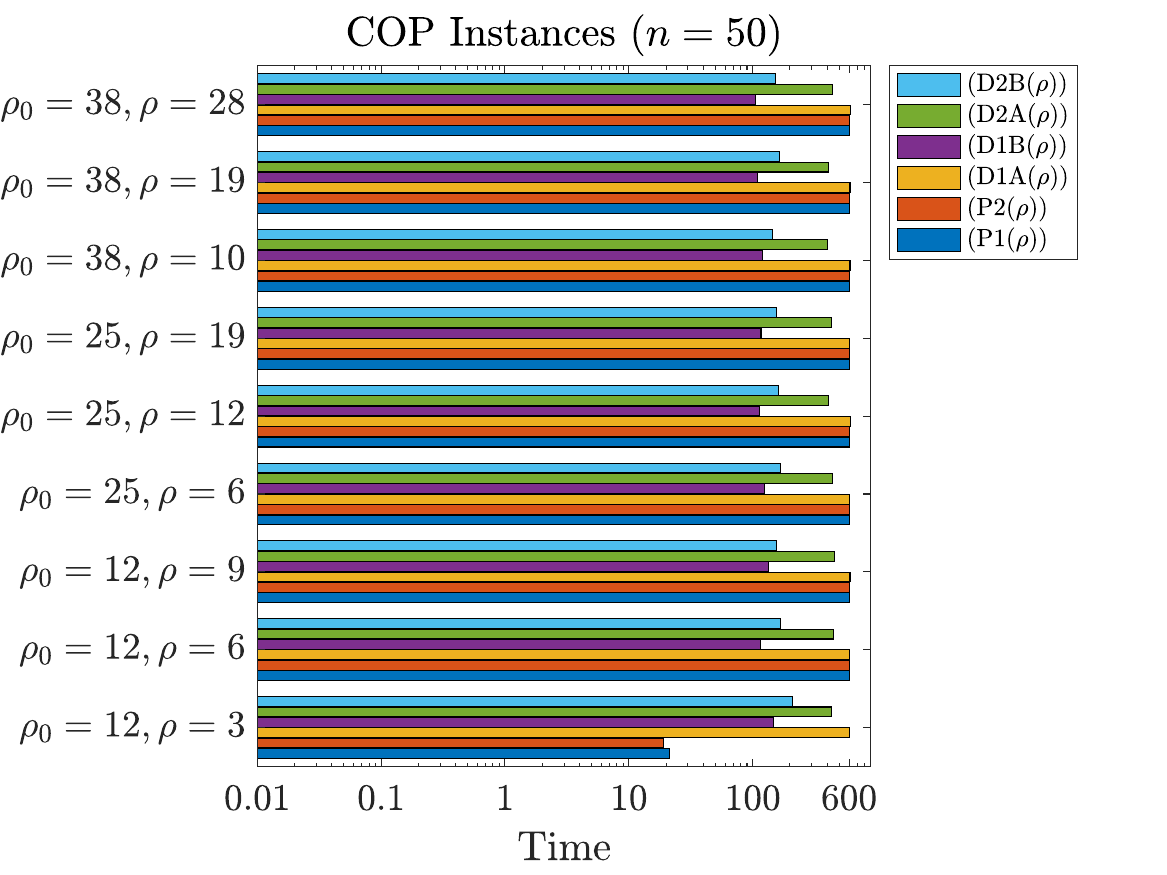}
		\caption{COP Instances ($n = 50)$}
		\label{fig1bcop}
	\end{subfigure}
    \caption{Average solution times (in seconds) of \eqref{P1}, \eqref{P2}, \eqref{D1A}, \eqref{D1B}, \eqref{D2A}, and \eqref{D2B} for PSD, SPN, and COP instances}
    \label{Fig1-PSD-COP}
\end{figure}

Figure~\ref{Fig1-PSD-COP} depicts the average solution times of each of the six models on each of the three instance sets. The results for PSD instances, SPN instances, and COP instances are presented in the first, second, and third rows of Figure~\ref{Fig1-PSD-COP}, respectively. In each row, the results for $n = 25$ and $n = 50$ are given in the first and second column, respectively. In each graph, the horizontal axis denotes the average solution time (in seconds) presented in logarithmic scale. We used identical axis limits in each graph to facilitate a better comparison across different graphs. The vertical axis is comprised of nine sets of bar charts, each of which represents a particular choice of the tuple $(\rho_0,\rho)$ corresponding to the choice of $n$ as outlined in Table~\ref{tab1}. Finally, for each choice of $(\rho_0,\rho)$, the corresponding bar chart represents the average solution times of the six models \eqref{P1}, \eqref{P2}, \eqref{D1A}, \eqref{D1B}, \eqref{D2A}, and \eqref{D2B} over 25 instances generated using the procedure in Section~\ref{CRInst}. Note that the solution times of instances that were terminated due to the time limit of 600 seconds were also included in the average solution times. Therefore, each bar chart represents the average computational requirement of the corresponding model. We recall that \eqref{P1} and \eqref{P2} were solved by {\tt Gurobi} whereas \eqref{D1A}, \eqref{D1B}, \eqref{D2A}, and \eqref{D2B} by {\tt MOSEK}. 

Figure~\ref{Fig1-PSD-COP} reveals several interesting relations about average solution times. We first outline our observations concerning the exact MIQP models \eqref{P1} and \eqref{P2}: 

\begin{itemize}
    \item[(i)] Average solution times of the two exact models exhibit very similar behavior for each fixed choice of the instance set and the triple $(n, \rho_0, \rho)$. However, there seems to be a very strong correlation between the average solution time and the choices of $(n,\rho_0,\rho)$ as well as the instance set. 
    
    \item[(ii)] Considering each row of Figure~\ref{Fig1-PSD-COP} separately, the average solution time of each of the two exact models across all choices of the tuple $(\rho_0,\rho)$ increases as $n$ increases. However, the rate of increase seems to be highly dependent on the choice of the instance set. 
\end{itemize}

We next present the corresponding observations for the average solution times of the convex relaxations 
\eqref{D1A}, \eqref{D1B}, \eqref{D2A}, and \eqref{D2B}:

\begin{itemize}
    \item[(i)] In contrast with the exact models, for each choice of $n$, Figure~\ref{Fig1-PSD-COP} reveals that the average solution times of the convex relaxations 
exhibit very similar behavior across all three sets of instances and all choices of $(\rho_0,\rho)$. 

\item[(ii)] The average solution time of our reduced formulation \eqref{D1B} consistently achieves the lowest average solution times, followed by the reduced formulation \eqref{D2B}, which, in turn, is followed by \eqref{D2A} and \eqref{D1A}, respectively. 

\item[(iii)] The average solution time of our reduced formulation \eqref{D1B} exhibits at least an order of magnitude improvement over that of \eqref{D1A}. On the other hand, the corresponding improvement of \eqref{D2B} over \eqref{D2A} seems to be less pronounced, especially in smaller dimensions.

 \item[(iv)] As expected, the average solution time of each relaxation increases with $n$. While the rate of increase does not seem to be influenced by the specific instance set, we observe that it does depend on the specific relaxation.
\end{itemize}


\begin{figure}[!hbt]
    \centering
    \begin{subfigure}{0.495\linewidth}
		\includegraphics[width=\linewidth]{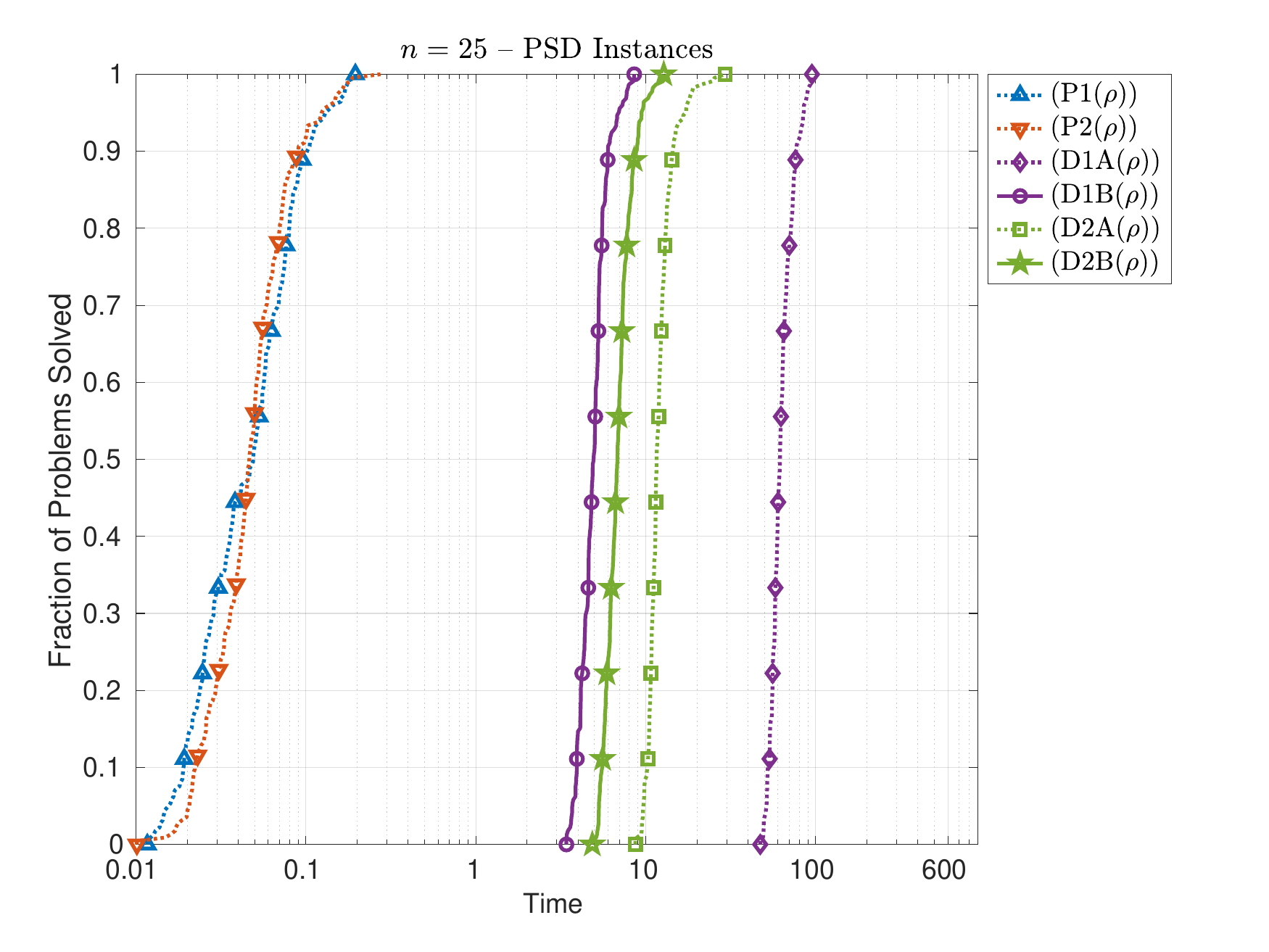}
		\caption{PSD Instances ($n = 25$)}
		\label{fig3a}
	\end{subfigure}
	\begin{subfigure}{0.495\linewidth}
  \includegraphics[width=\linewidth]{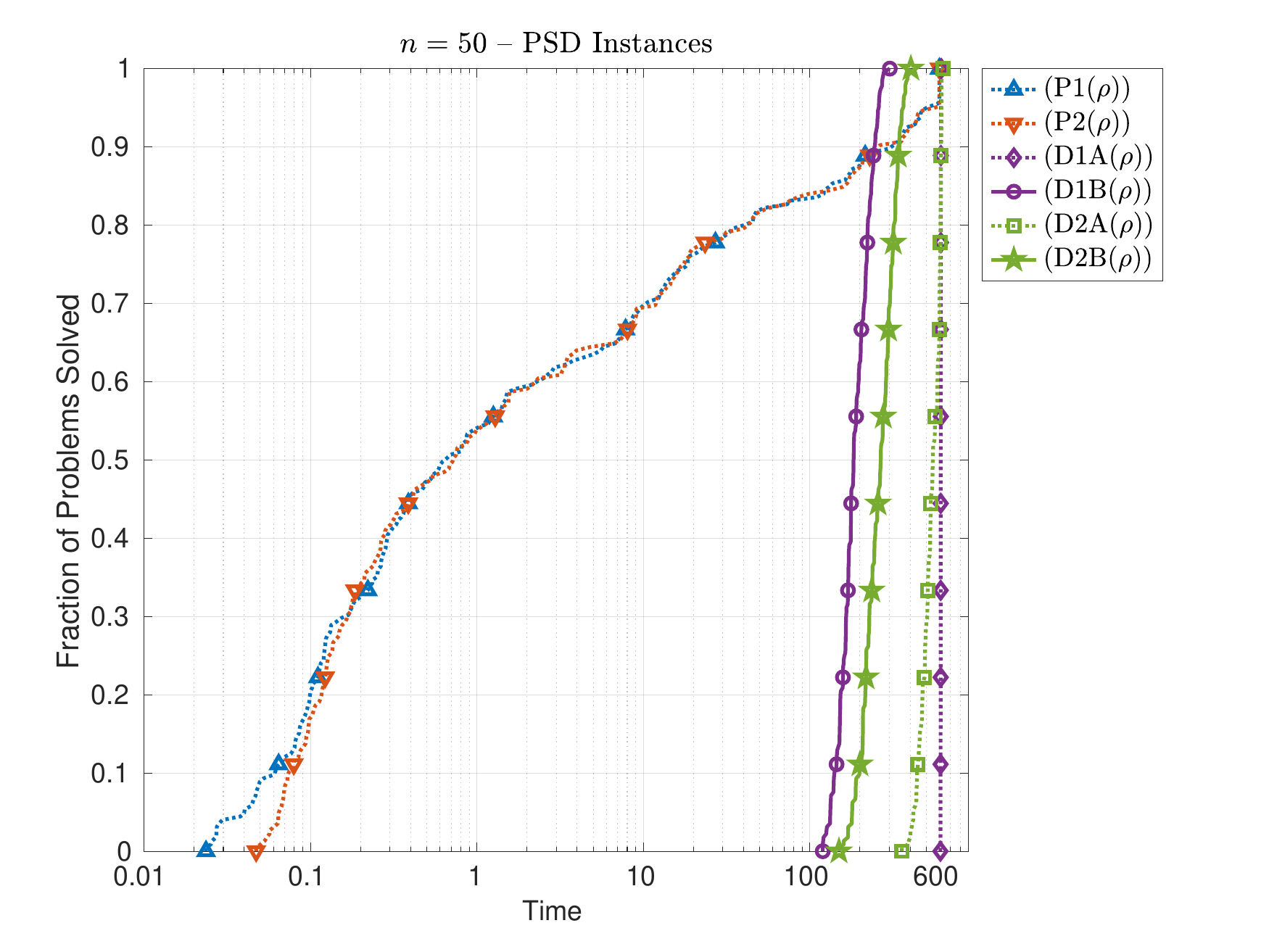}
		\caption{PSD Instances ($n = 50$)}
		\label{fig3b}
	\end{subfigure}
 \\
 \begin{subfigure}{0.495\linewidth}
		\includegraphics[width=\linewidth]{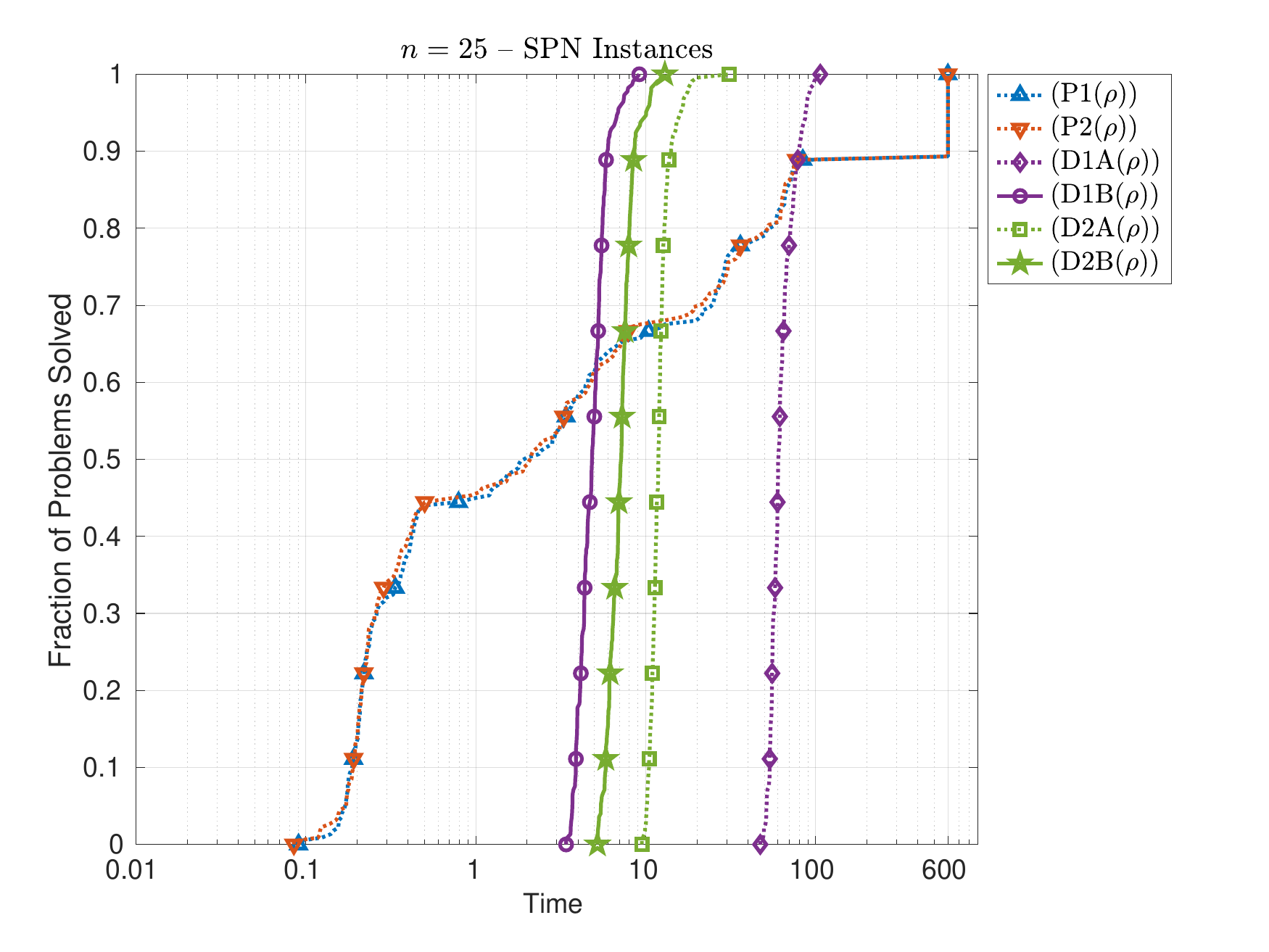}
		\caption{SPN Instances ($n = 25$)}
		\label{fig3c}
	\end{subfigure}
	\begin{subfigure}{0.495\linewidth}
		\includegraphics[width=\linewidth]{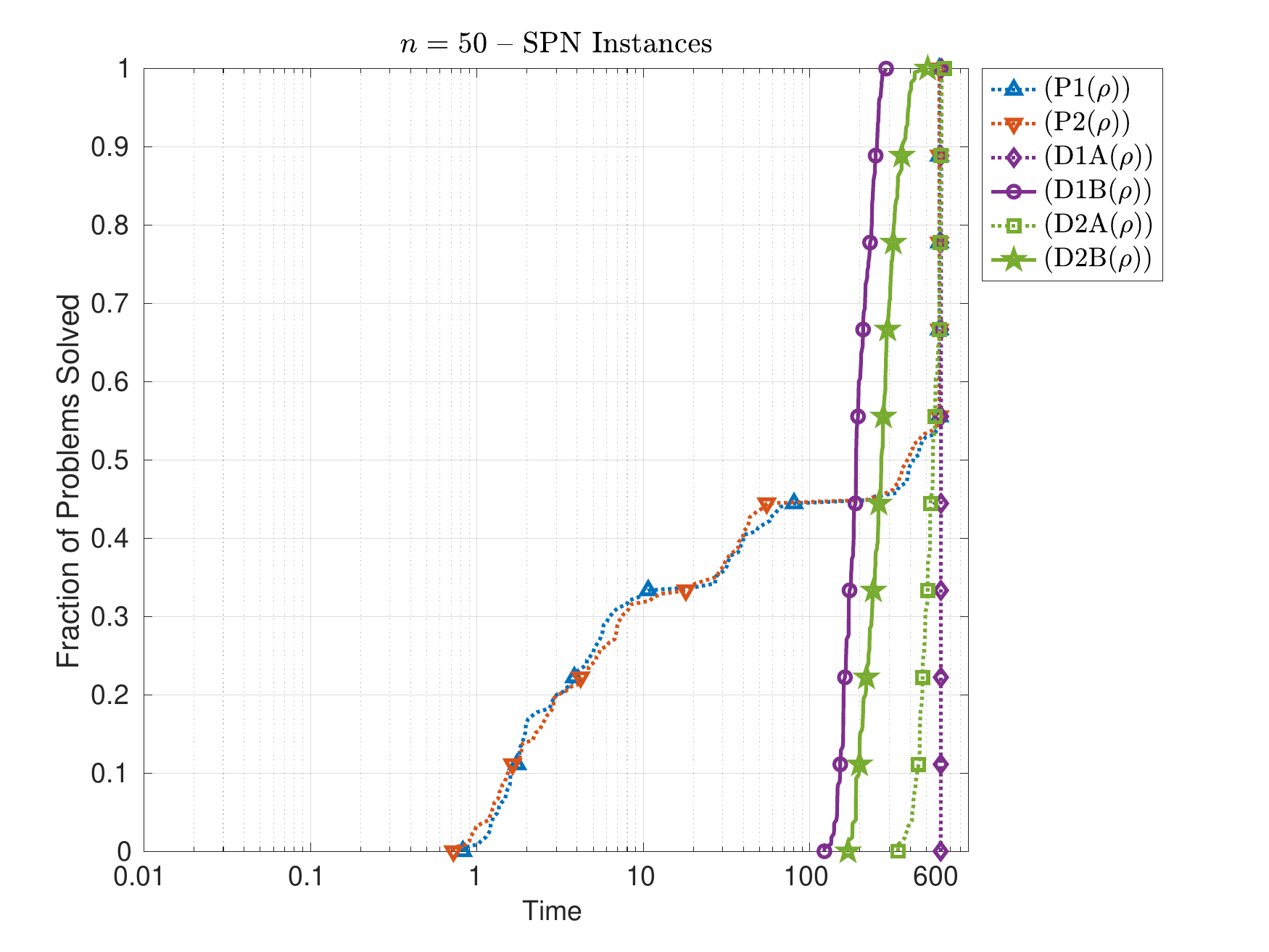}
		\caption{SPN Instances ($n = 50$)}
		\label{fig3d}
	\end{subfigure}
 \\
  \begin{subfigure}{0.495\linewidth}
		\includegraphics[width=\linewidth]{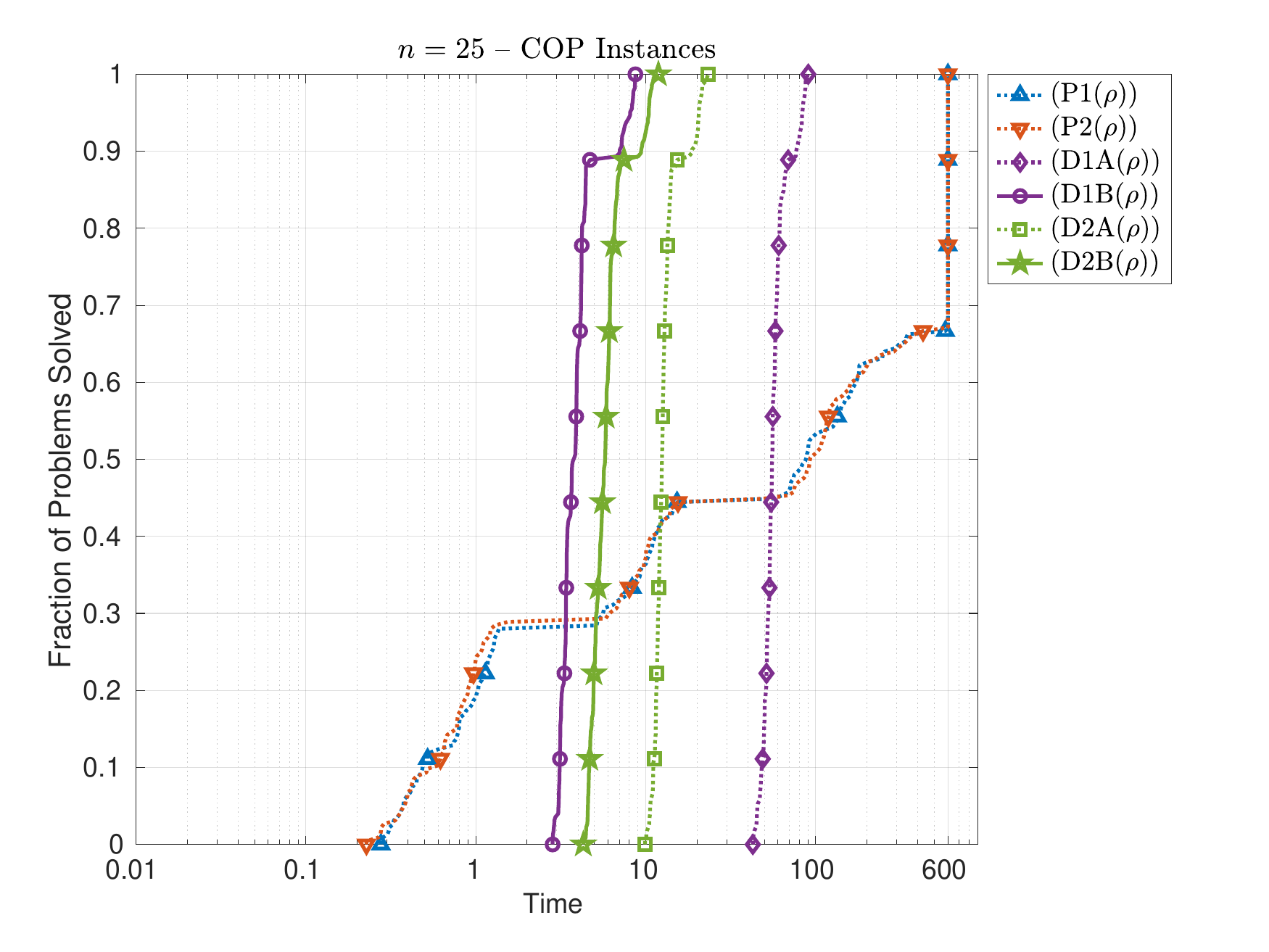}
		\caption{COP Instances ($n = 25$)}
		\label{fig3e}
	\end{subfigure}
	\begin{subfigure}{0.495\linewidth}
		\includegraphics[width=\linewidth]{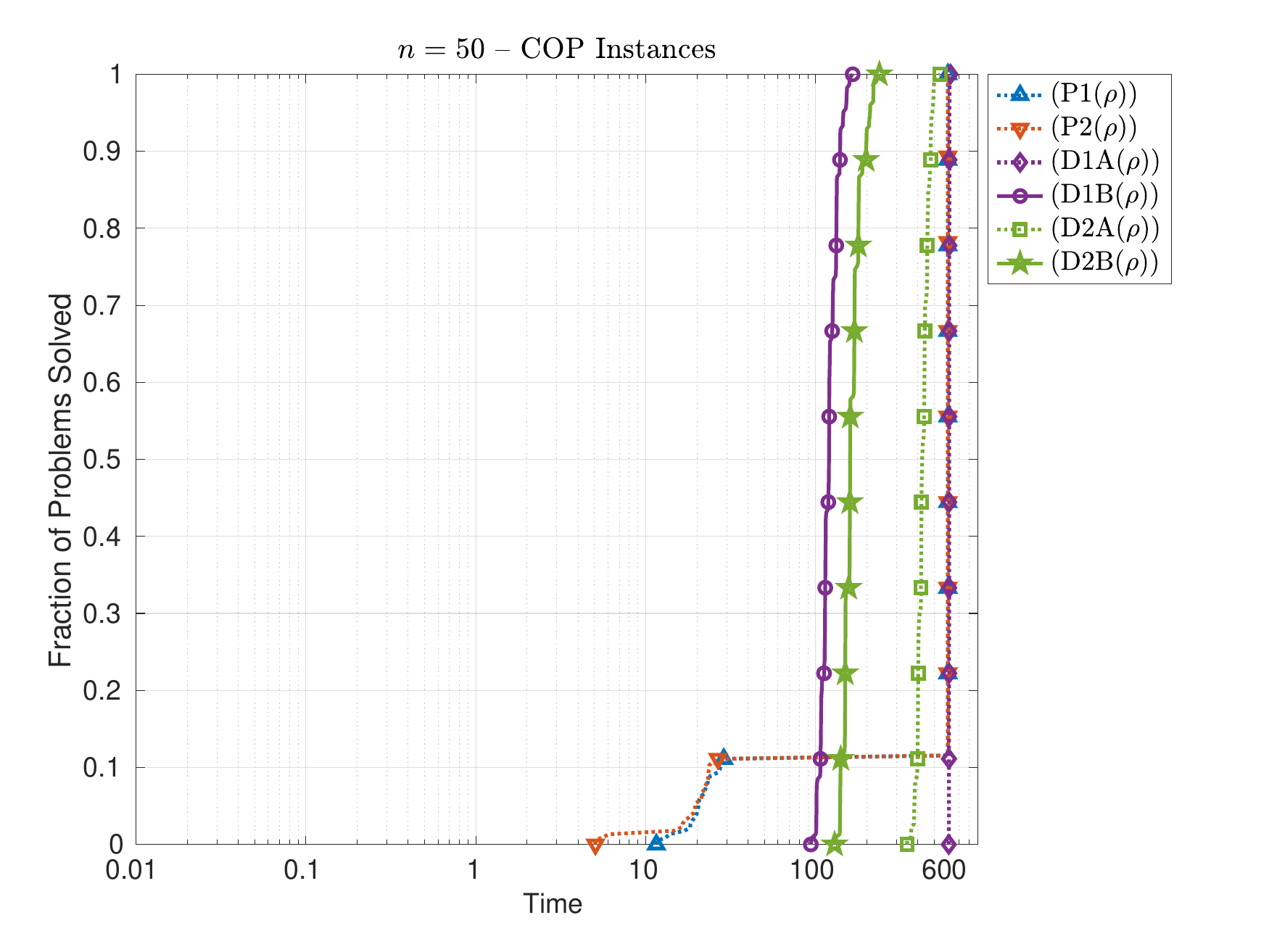}
		\caption{COP Instances ($n = 50$)}
		\label{fig3f}
	\end{subfigure}
    \caption{Empirical cumulative distribution functions of solution times of \eqref{P1},  \eqref{P2}, \eqref{D1A}, \eqref{D1B}, \eqref{D2A}, and \eqref{D2B}}
    \label{Fig3-D1D2}
\end{figure}

Since Figure~\ref{Fig1-PSD-COP} is based only on average solution times, it cannot capture the variability in the data. In an attempt to shed more light on the distribution of the solution times across all instance sets, we present the empirical cumulative distribution functions of all six models in Figure~\ref{Fig3-D1D2}, which is organized similarly to Figure~\ref{Fig1-PSD-COP}. The three rows display the results for PSD, SPN, and COP instances, respectively, whereas the two columns are devoted to the results for $n = 25$ and $n = 50$, respectively. Each of the six graphs presents six plots corresponding to the empirical cumulative distribution functions of solution times of each of the six models on all 225 instances for the corresponding choice of the instance set and $n$ (see Table~\ref{tab1}). The horizontal axis denotes the solution time (in seconds) on a logarithmic scale, and the vertical axis represents the fraction of the instances solved. The markers represent the data points for every $25^{\textrm{th}}$ instance. Once again, we employed identical axis limits in each of the six graphs.

Figure~\ref{Fig3-D1D2} clearly reveals the difference between the distributions of solution times of the exact models \eqref{P1} and \eqref{P2} and those of the convex relaxations \eqref{D1A}, \eqref{D1B}, \eqref{D2A}, and \eqref{D2B}. We outline our observations below:

\begin{itemize}
    \item[(i)] The solution times of the two exact models, in general, exhibit a very similar distribution but a significant variability across different instance sets. 
    
    \item[(ii)] In contrast, we observe that the solution times of the convex relaxations tend to have a much smaller variance and exhibit very similar distributions across different instance sets.
\end{itemize}

Therefore, in contrast with the exact models \eqref{P1} and \eqref{P2}, we conclude that the solution time of each convex relaxation seems to be very robust with respect to the choice of the instance set and the choice of $(\rho_0,\rho)$ in our setting. 

In the sequel, we present a more detailed discussion of the behavior of the solution times of the exact models and convex relaxations. 

\subsubsection{Exact MIQP models} \label{CRExact}

As illustrated by Figures~\ref{Fig1-PSD-COP} and~\ref{Fig3-D1D2}, the solution time of the exact models is highly dependent on the choices of the instance set and on the triple $(n,\rho_0,\rho)$. 
Here, we aim to shed more light on this dependence. 


\begin{sidewaysfigure}
    \centering
    \begin{subfigure}{0.32\linewidth}
  \includegraphics[width=\linewidth]{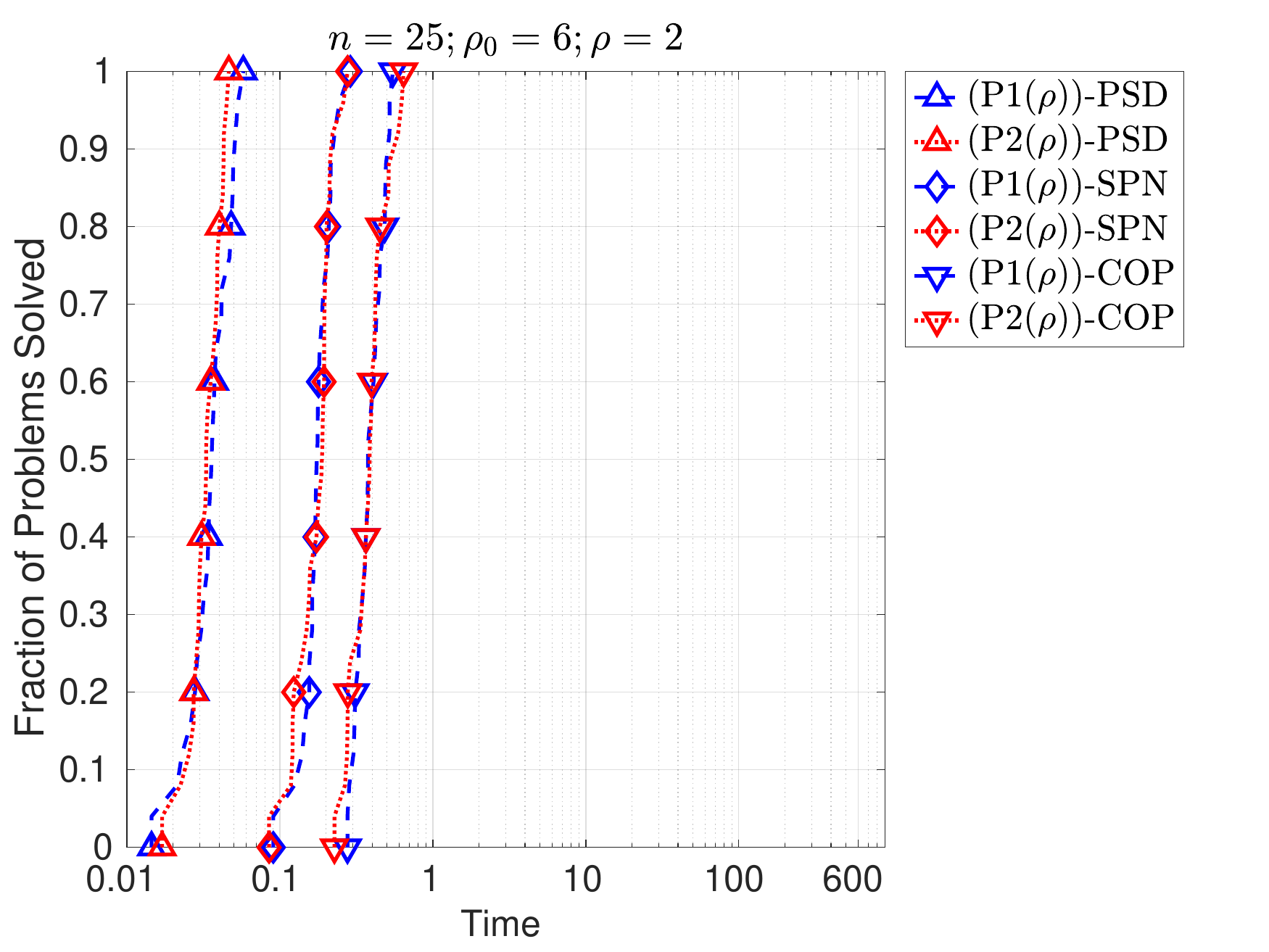}
		\caption{$n = 25, \rho_0 = 6, \rho = 2$}
		\label{fig1a}
	\end{subfigure}
	\begin{subfigure}{0.32\linewidth}
		\includegraphics[width=\linewidth]{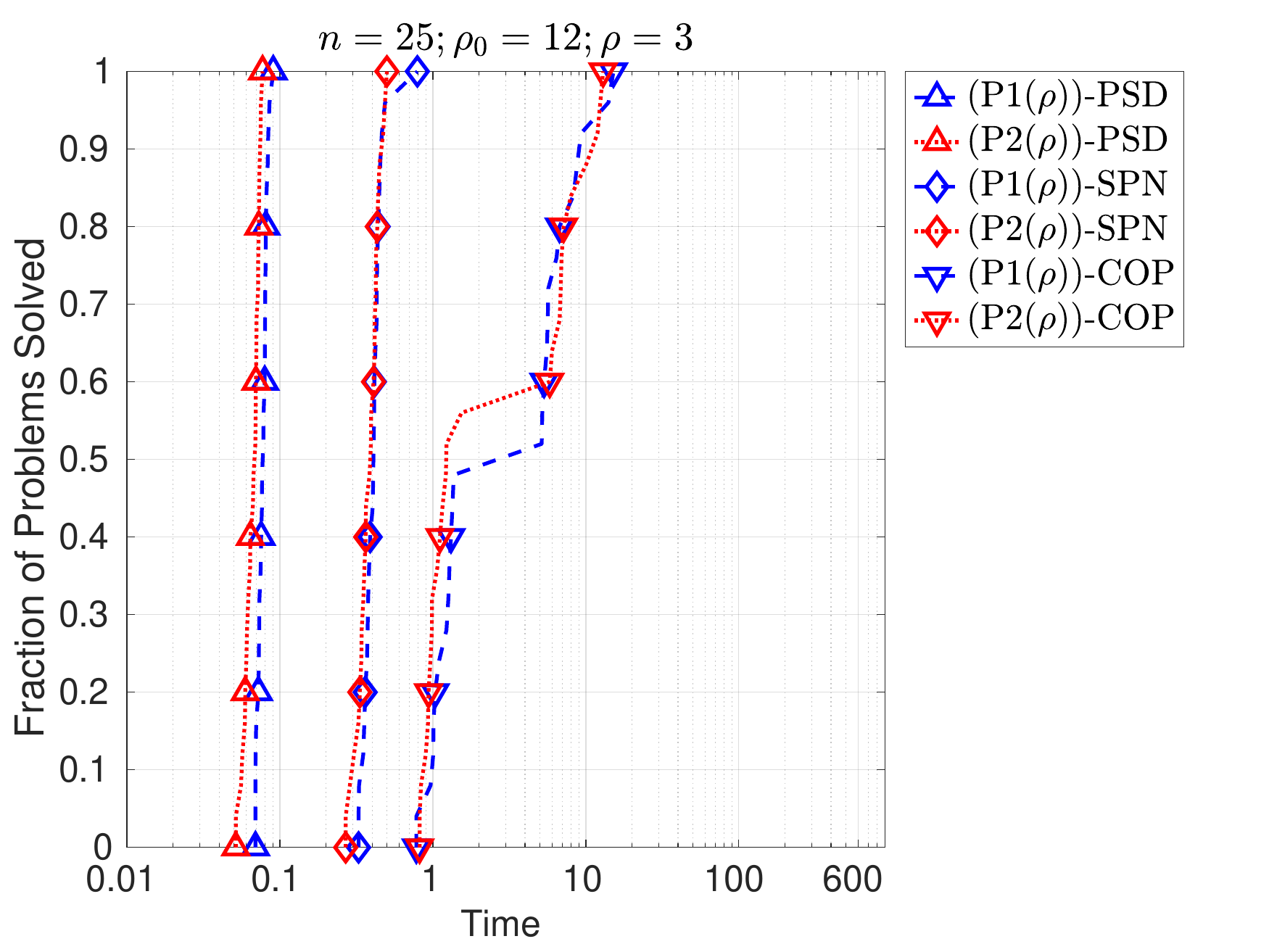}
		\caption{$n = 25, \rho_0 = 12, \rho = 3$}
		\label{fig1b}
	\end{subfigure}
 \begin{subfigure}{0.32\linewidth}
		\includegraphics[width=\linewidth]{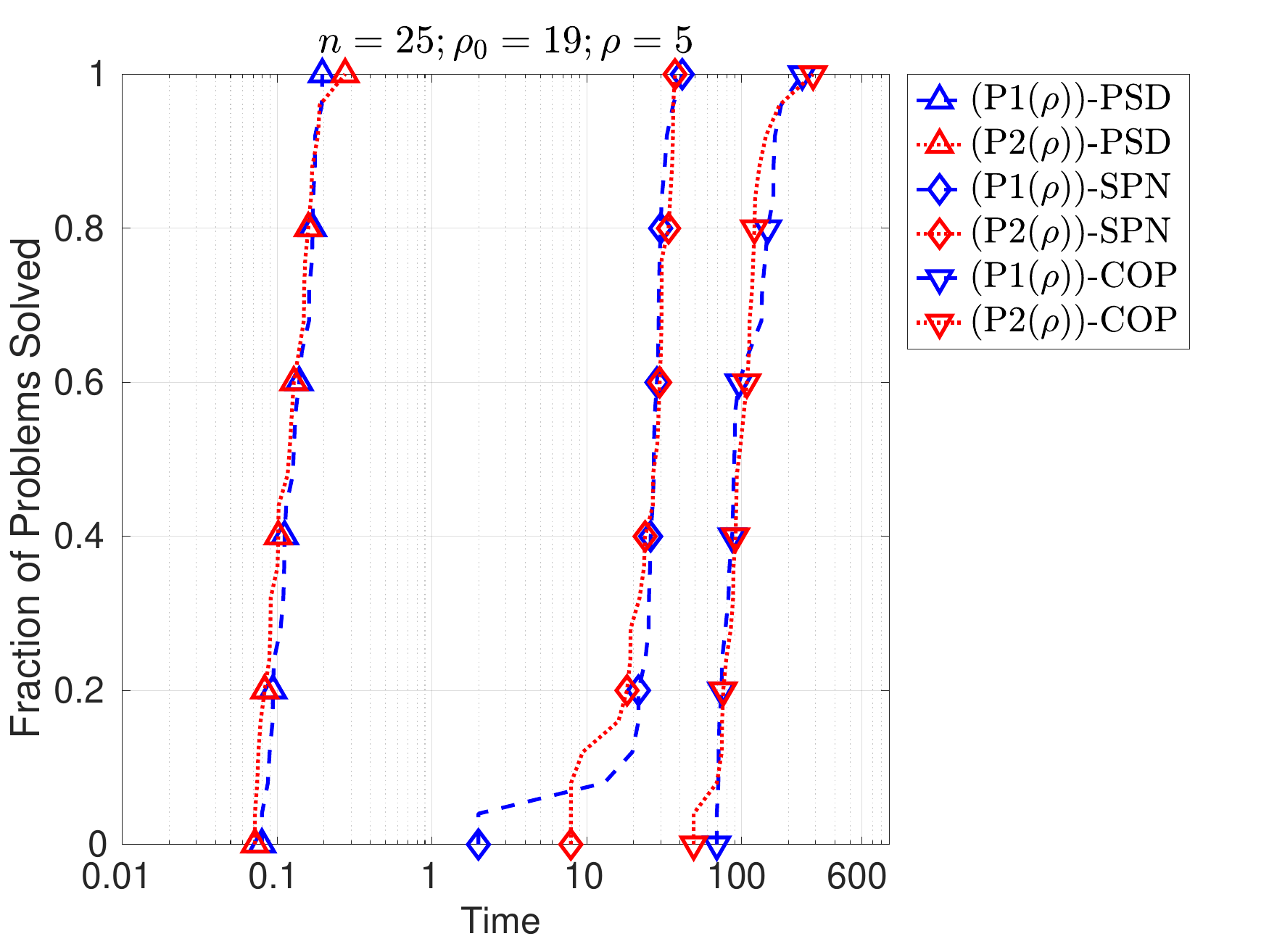}
		\caption{$n = 25, \rho_0 = 19, \rho = 5$}
		\label{fig1c}
	\end{subfigure}
 \\
 \begin{subfigure}{0.32\linewidth}
 \includegraphics[width=\linewidth]{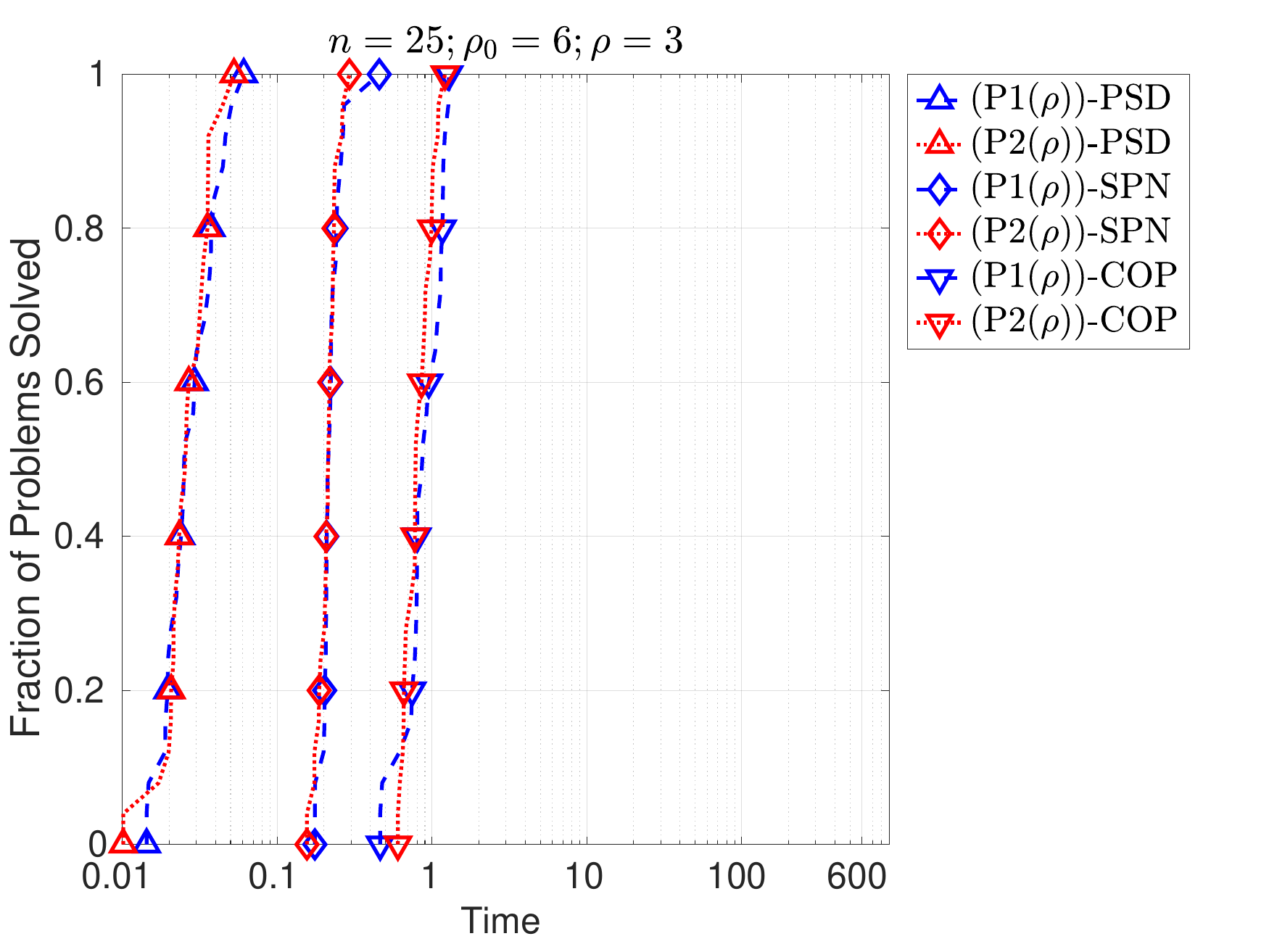}
		\caption{$n = 25, \rho_0 = 6, \rho = 3$}
		\label{fig1d}
	\end{subfigure}
	\begin{subfigure}{0.32\linewidth}
		\includegraphics[width=\linewidth]{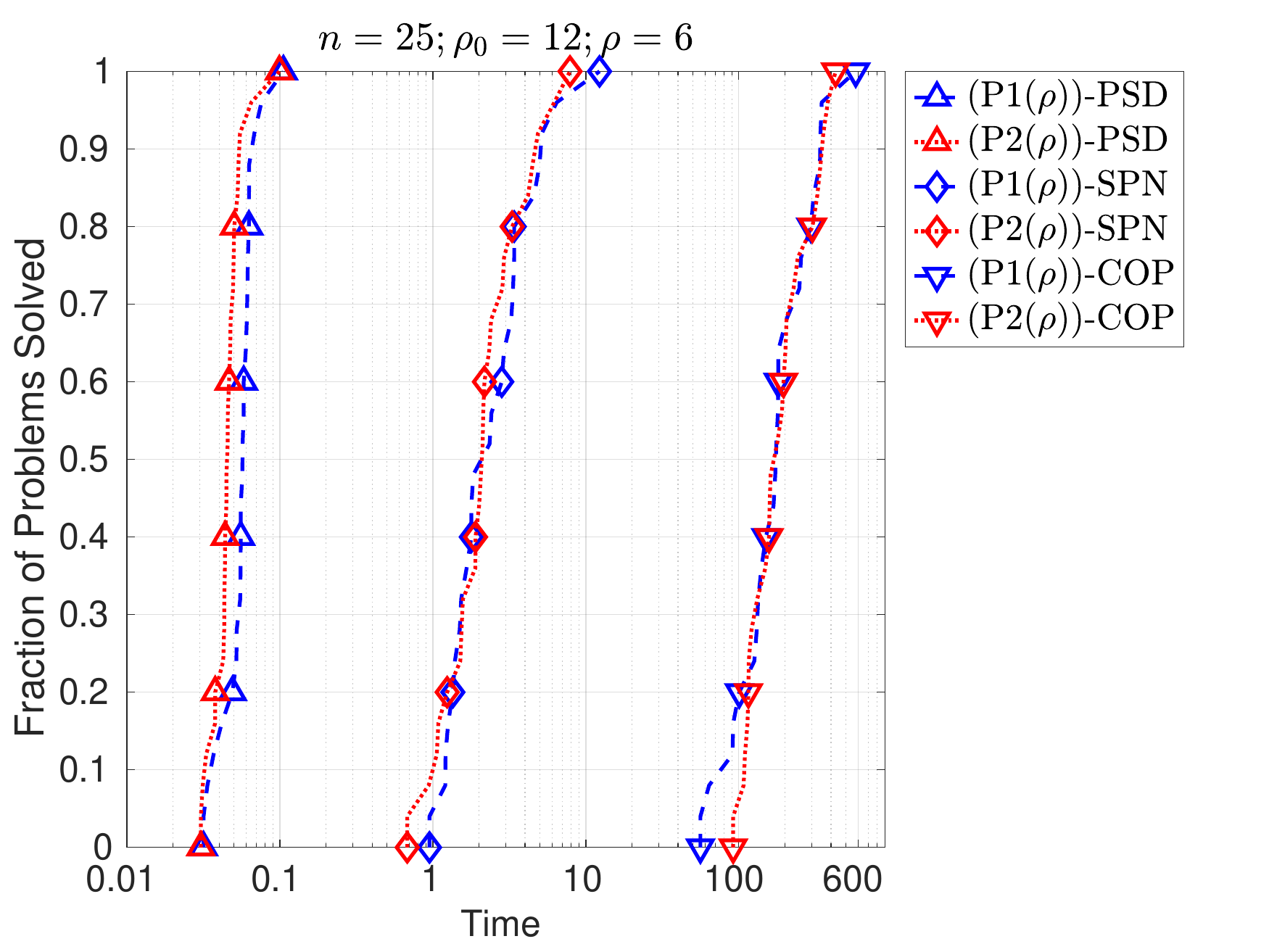}
		\caption{$n = 25, \rho_0 = 12, \rho = 6$}
		\label{fig1e}
	\end{subfigure}
 \begin{subfigure}{0.32\linewidth}
		\includegraphics[width=\linewidth]{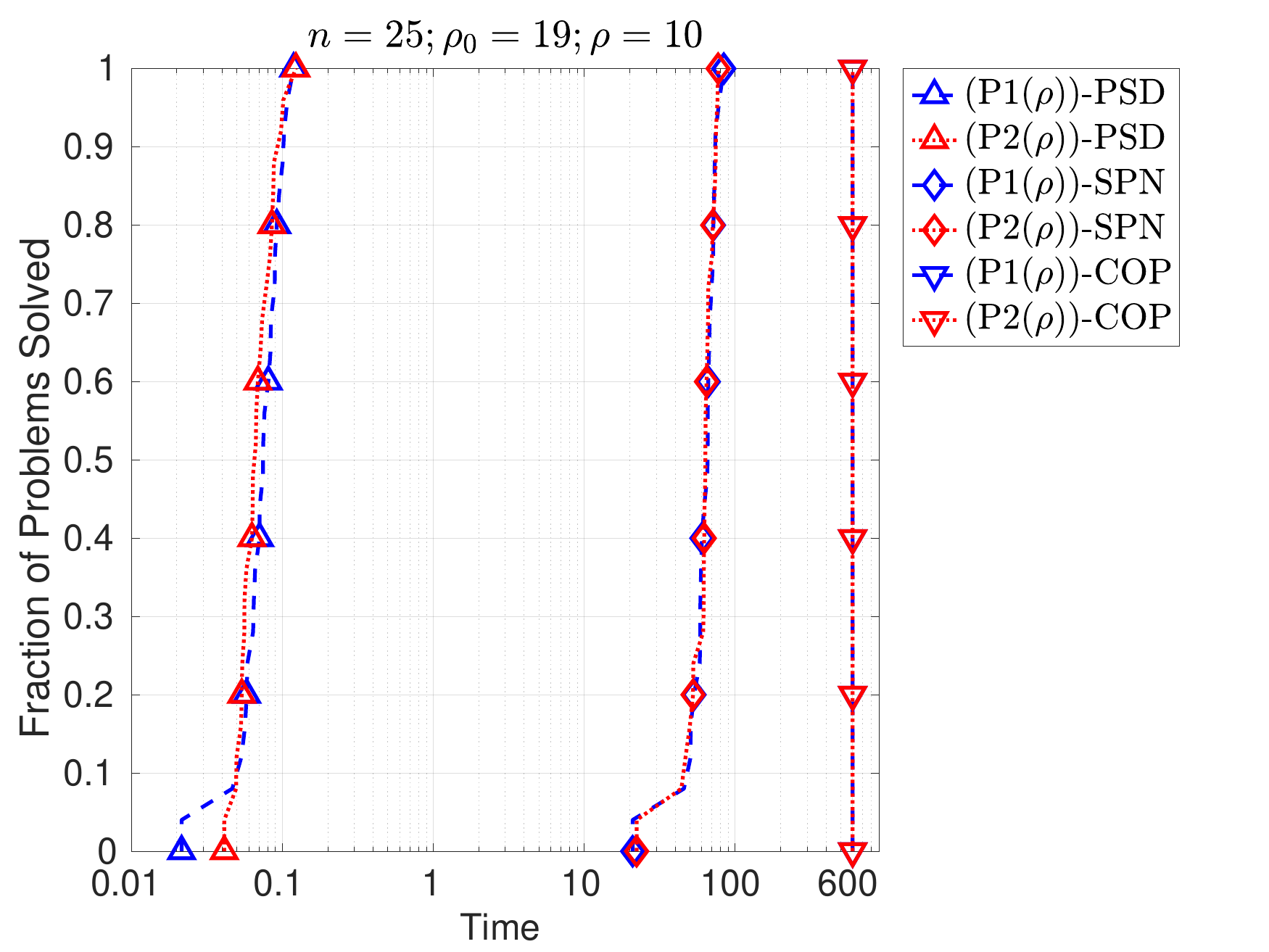}
		\caption{$n = 25, \rho_0 = 19, \rho = 10$}
		\label{fig1f}
	\end{subfigure}
 \\
  \begin{subfigure}{0.32\linewidth}
 \includegraphics[width=\linewidth]{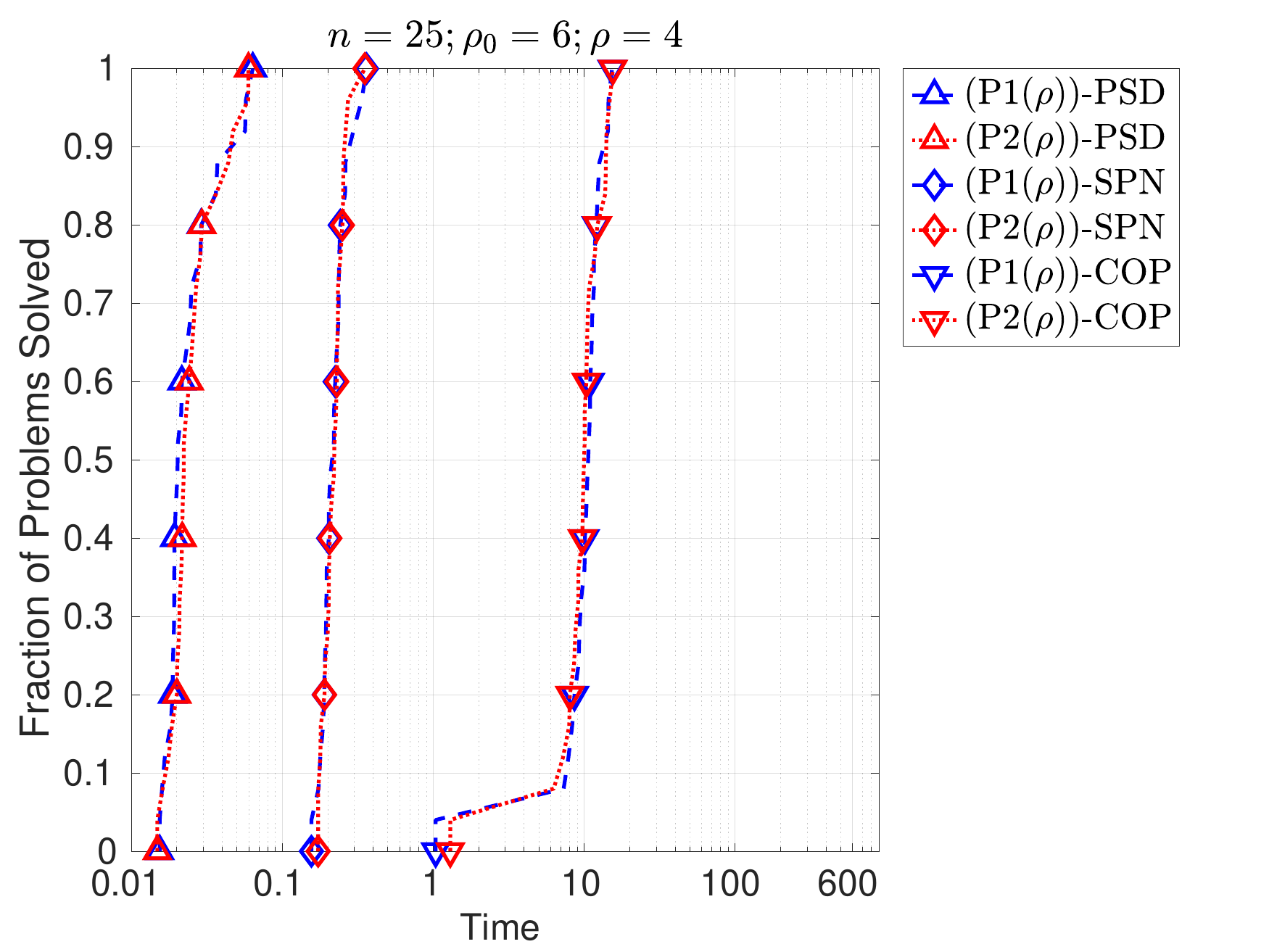}
		\caption{$n = 25, \rho_0 = 6, \rho = 4$}
		\label{fig1g}
	\end{subfigure}
	\begin{subfigure}{0.32\linewidth}
		\includegraphics[width=\linewidth]{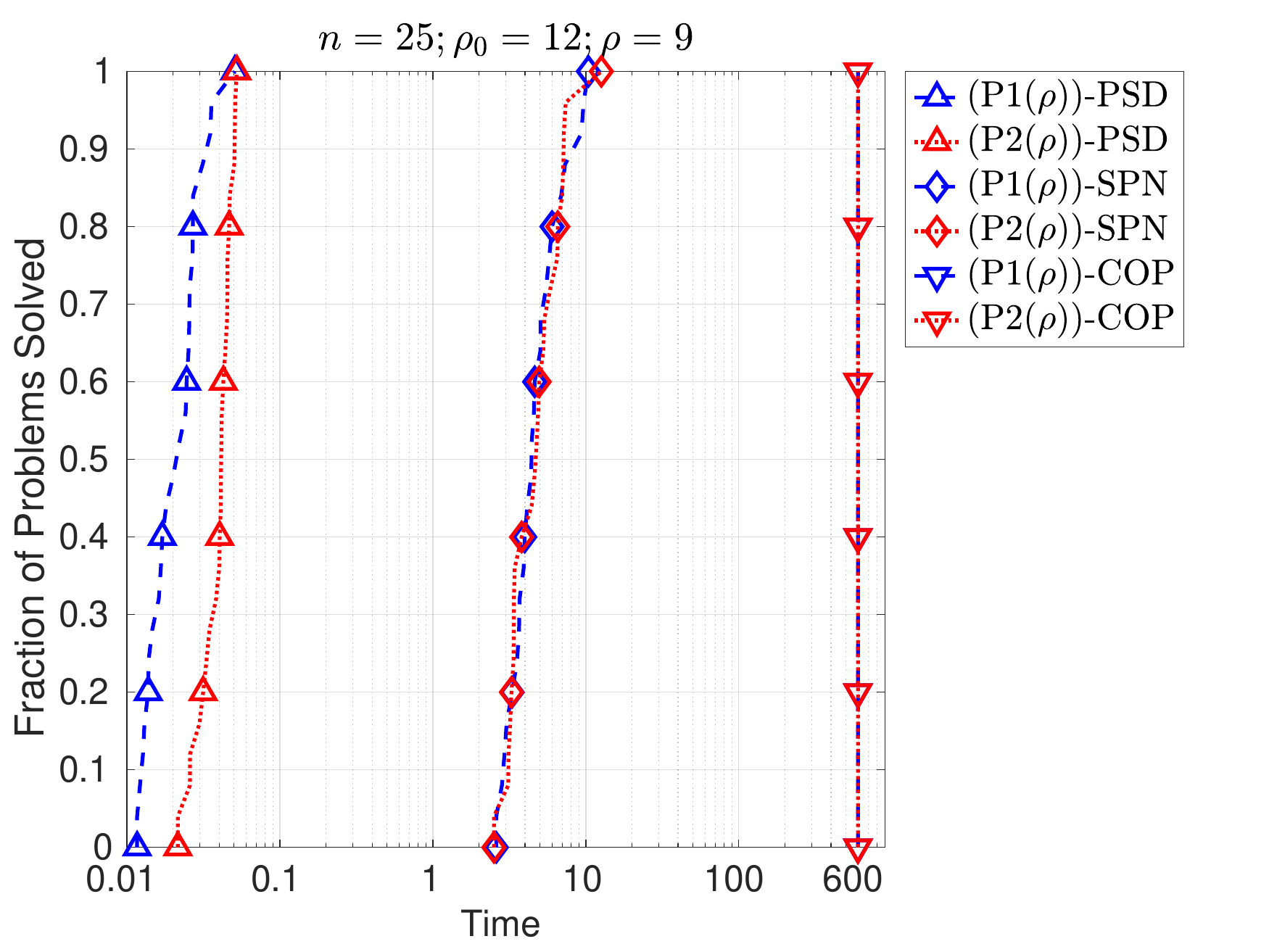}
		\caption{$n = 25, \rho_0 = 12, \rho = 9$}
		\label{fig1h}
	\end{subfigure}
 \begin{subfigure}{0.32\linewidth}
		\includegraphics[width=\linewidth]{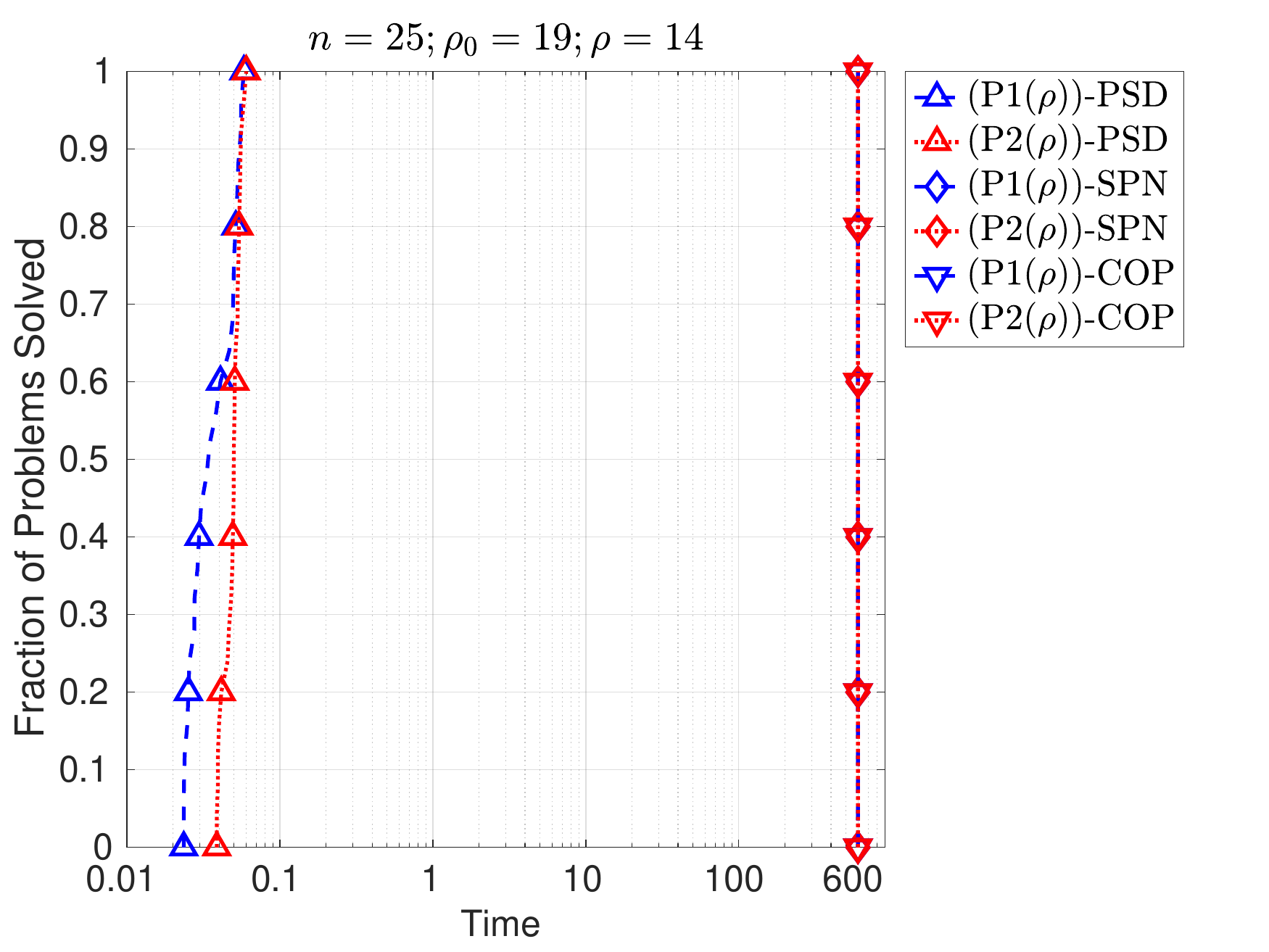}
		\caption{$n = 25, \rho_0 = 19, \rho = 14$}
		\label{fig1i}
	\end{subfigure}
    \caption{Empirical cumulative distribution functions of solution times of \eqref{P1} and \eqref{P2} for all instances with $n = 25$}
    \label{Fig1-n25-P1P2}
\end{sidewaysfigure}


\begin{sidewaysfigure}
    \centering
    \begin{subfigure}{0.32\linewidth}
  \includegraphics[width=\linewidth]{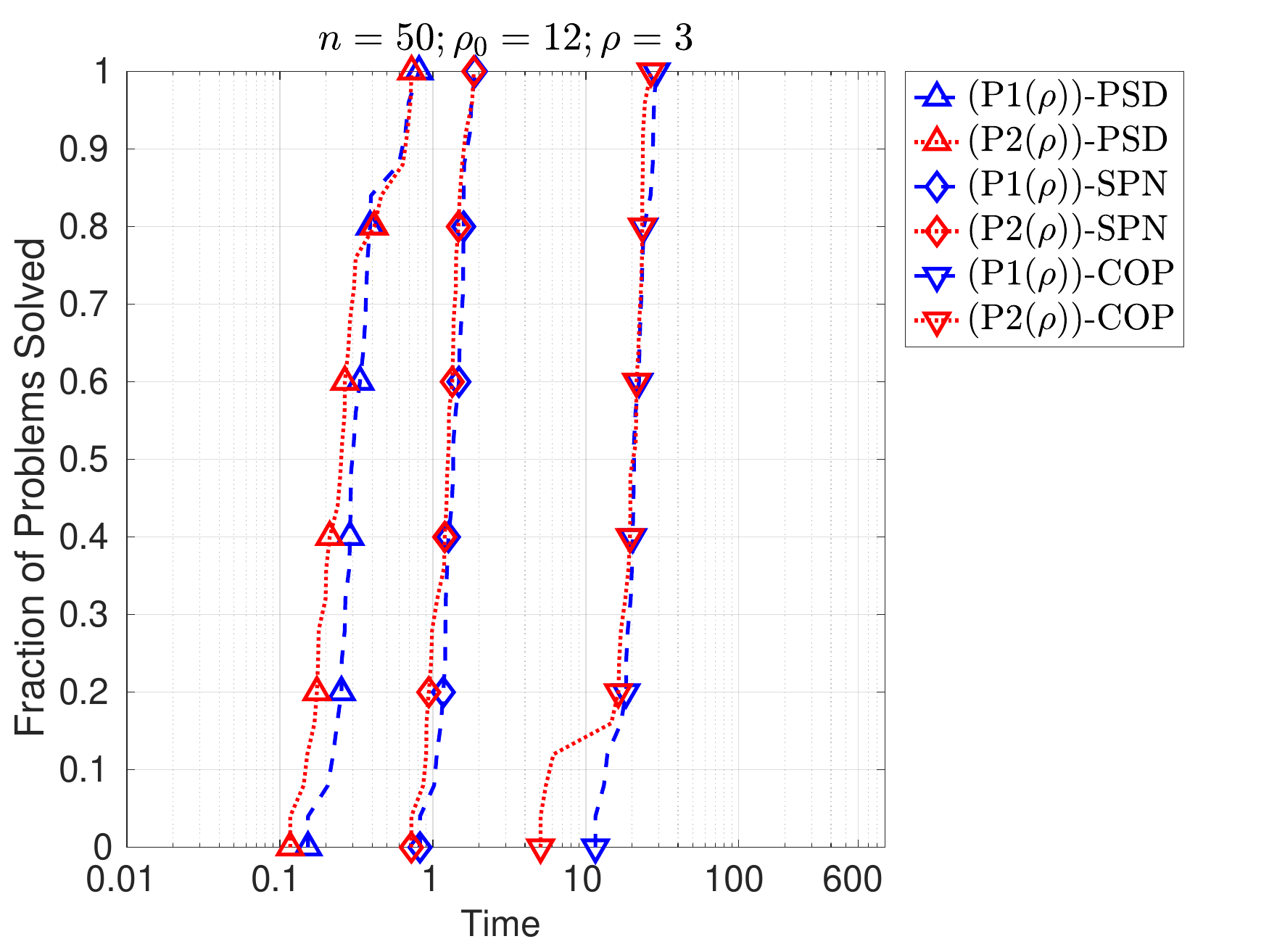}
		\caption{$n = 50, \rho_0 = 12, \rho = 3$}
		\label{fig2a}
	\end{subfigure}
	\begin{subfigure}{0.32\linewidth}
		\includegraphics[width=\linewidth]{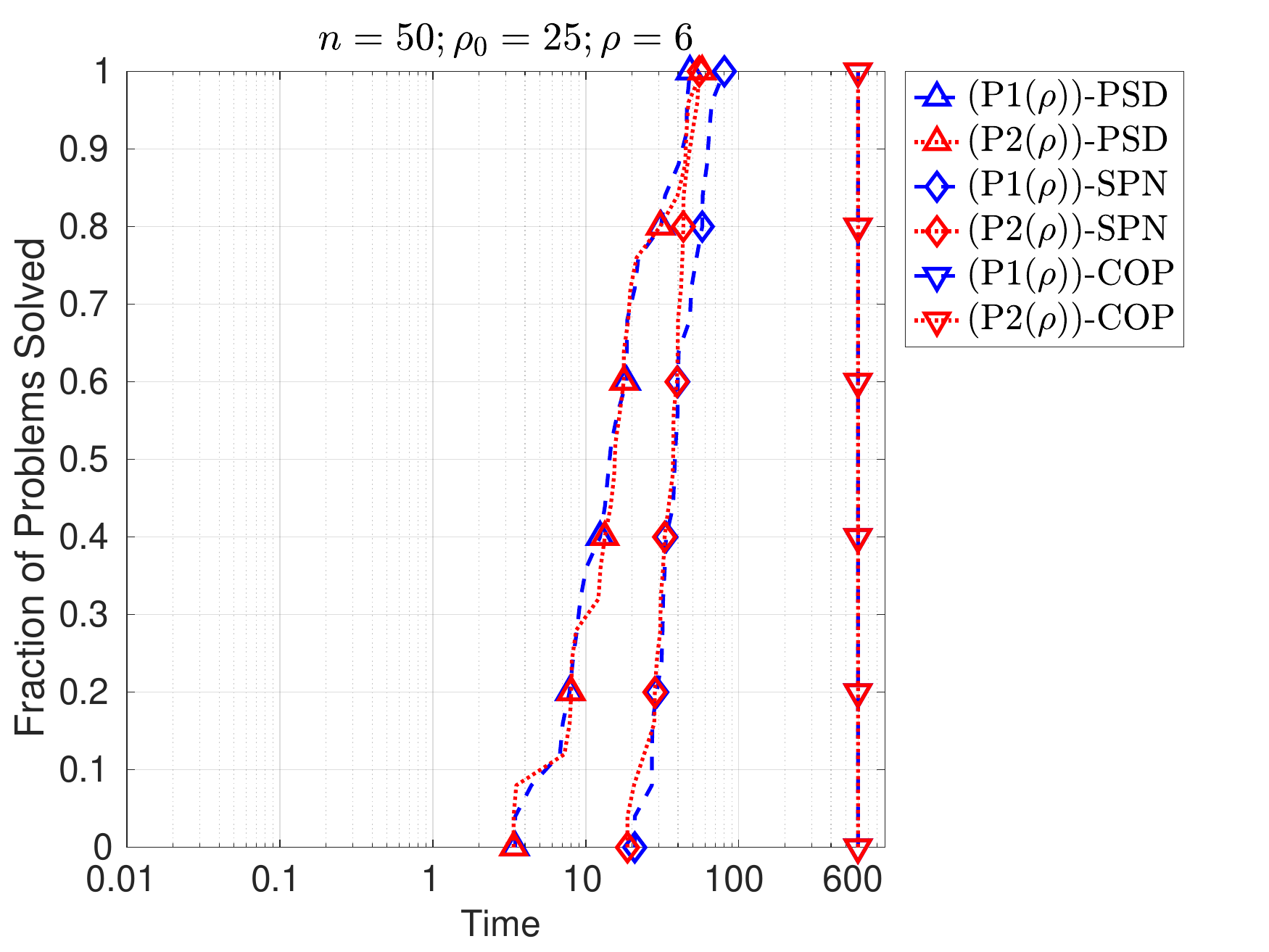}
		\caption{$n = 50, \rho_0 = 25, \rho = 6$}
		\label{fig2b}
	\end{subfigure}
 \begin{subfigure}{0.32\linewidth}
		\includegraphics[width=\linewidth]{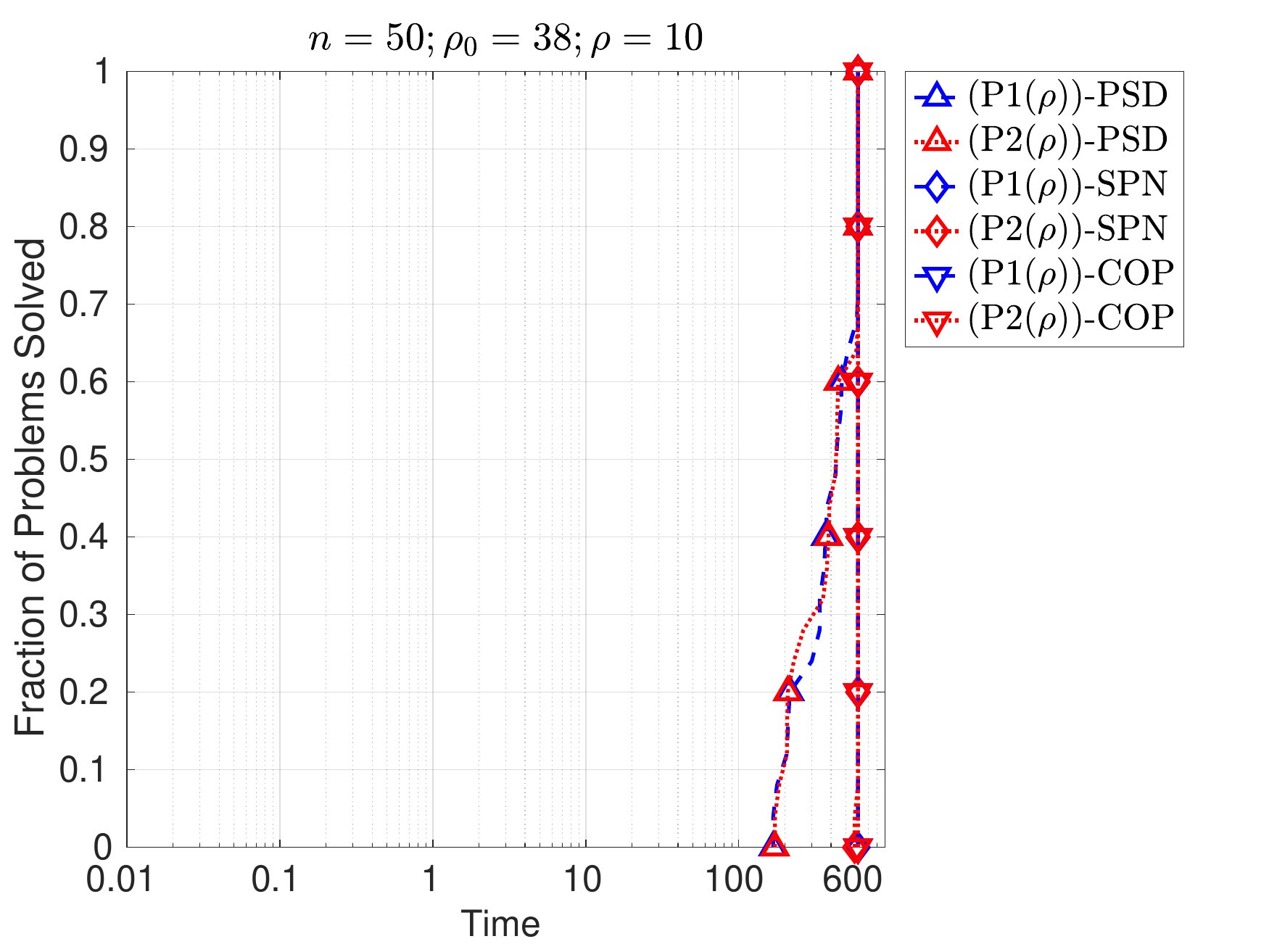}
		\caption{$n = 50, \rho_0 = 38, \rho = 10$}
		\label{fig2c}
	\end{subfigure}
 \\
 \begin{subfigure}{0.32\linewidth}
 \includegraphics[width=\linewidth]{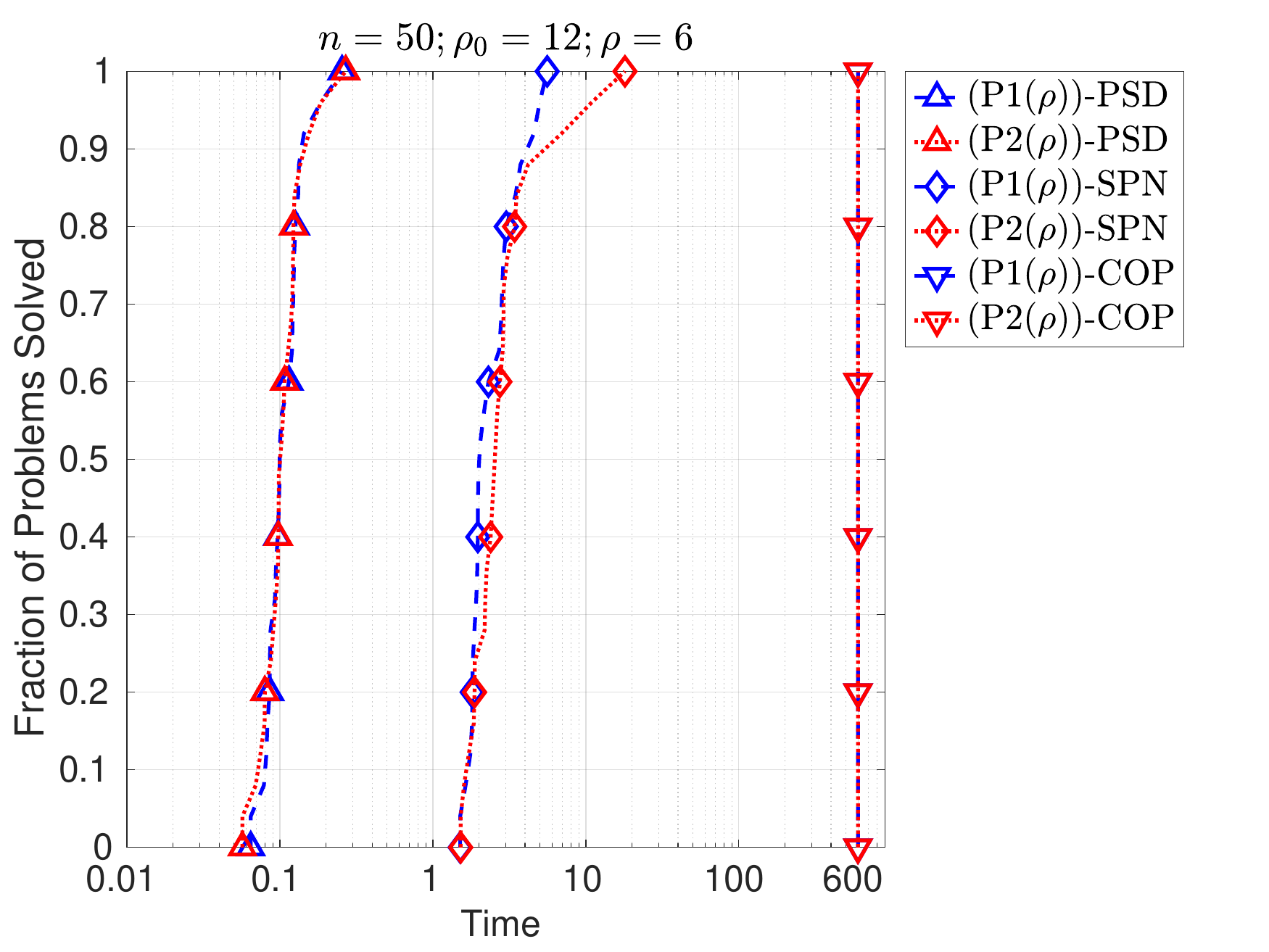}
		\caption{$n = 50, \rho_0 = 12, \rho = 6$}
		\label{fig2d}
	\end{subfigure}
	\begin{subfigure}{0.32\linewidth}
		\includegraphics[width=\linewidth]{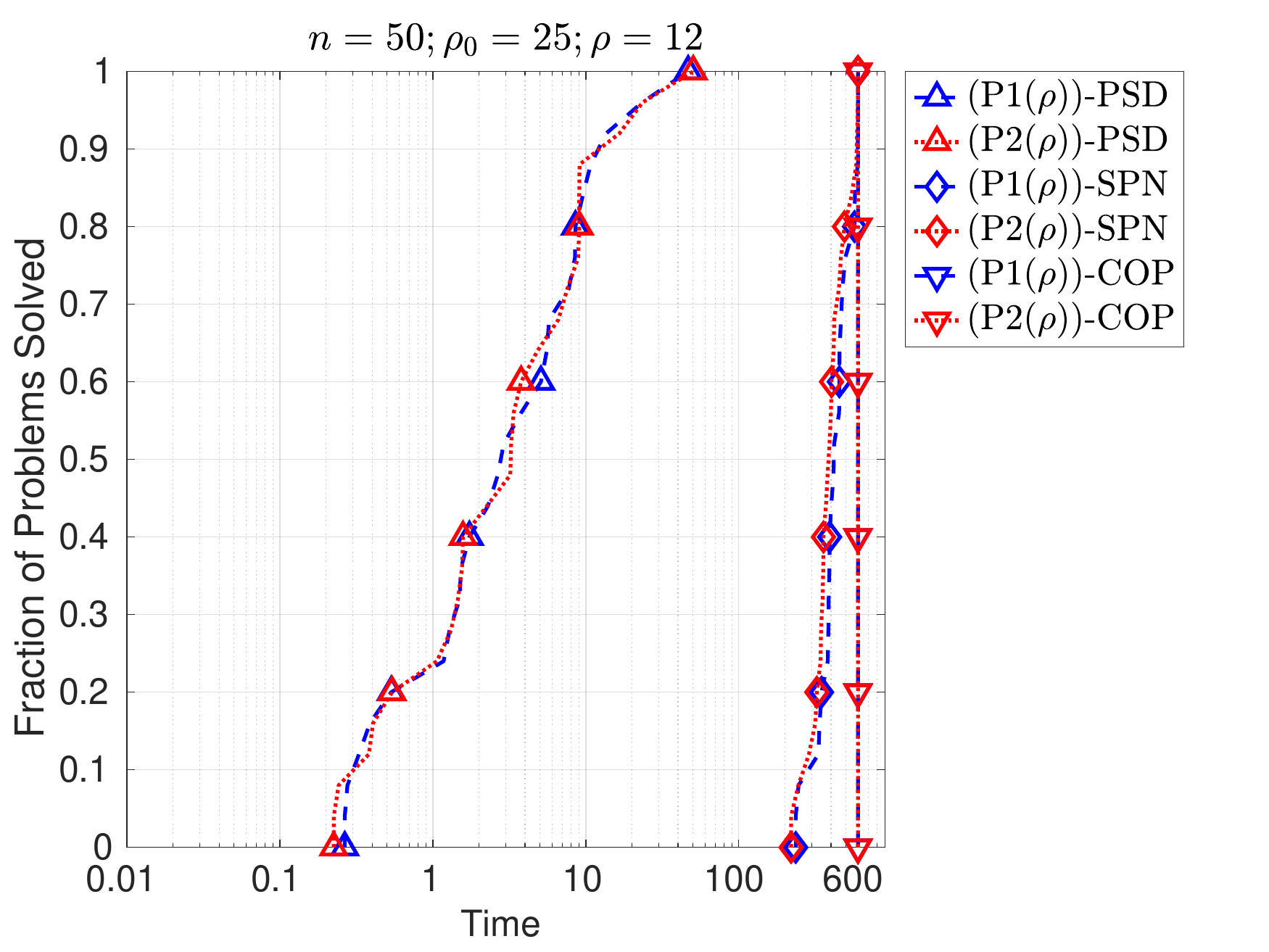}
		\caption{$n = 50, \rho_0 = 25, \rho = 12$}
		\label{fig2e}
	\end{subfigure}
 \begin{subfigure}{0.32\linewidth}
		\includegraphics[width=\linewidth]{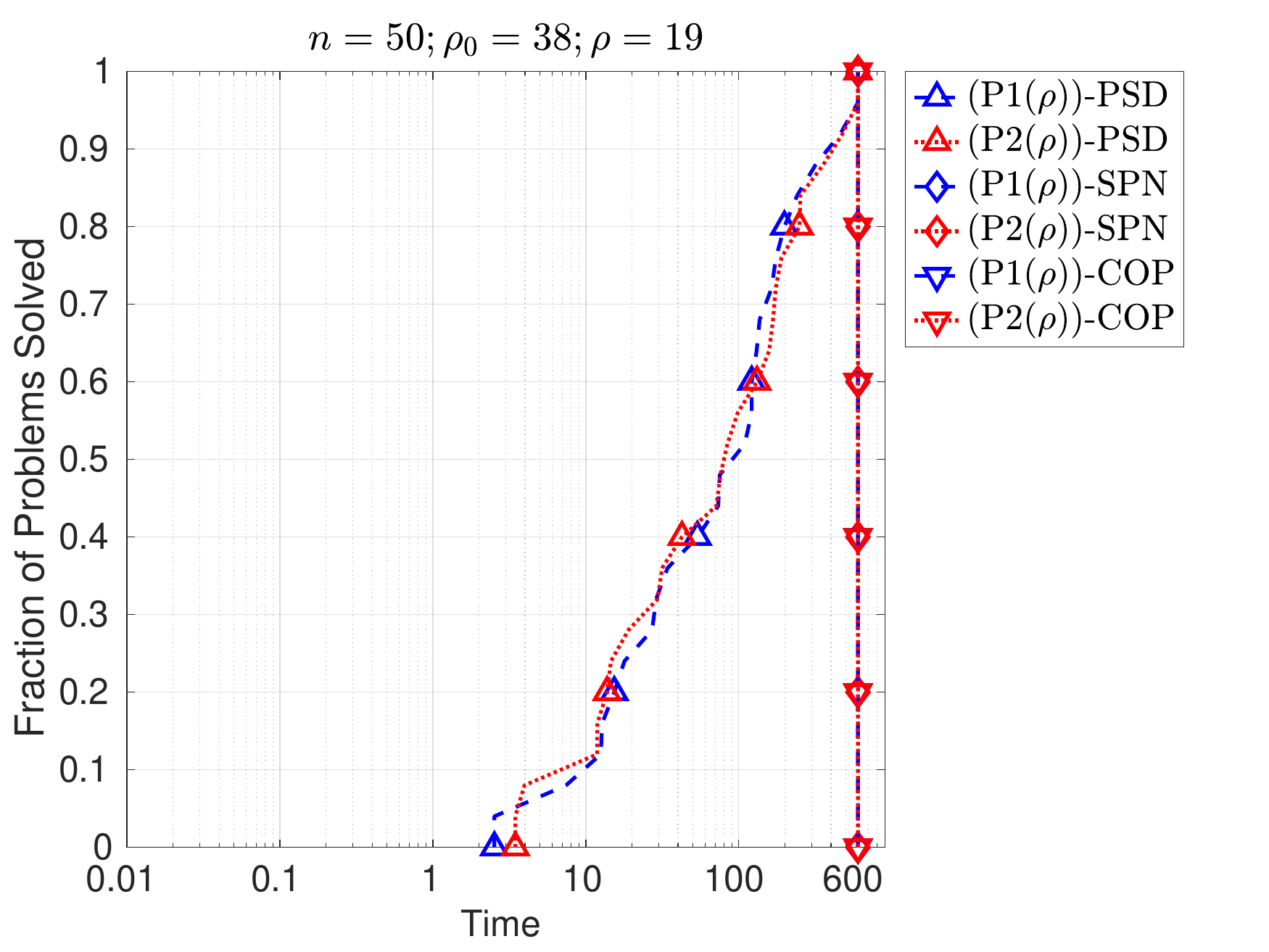}
		\caption{$n = 50, \rho_0 = 38, \rho = 19$}
		\label{fig2f}
	\end{subfigure}
 \\
  \begin{subfigure}{0.32\linewidth}
 \includegraphics[width=\linewidth]{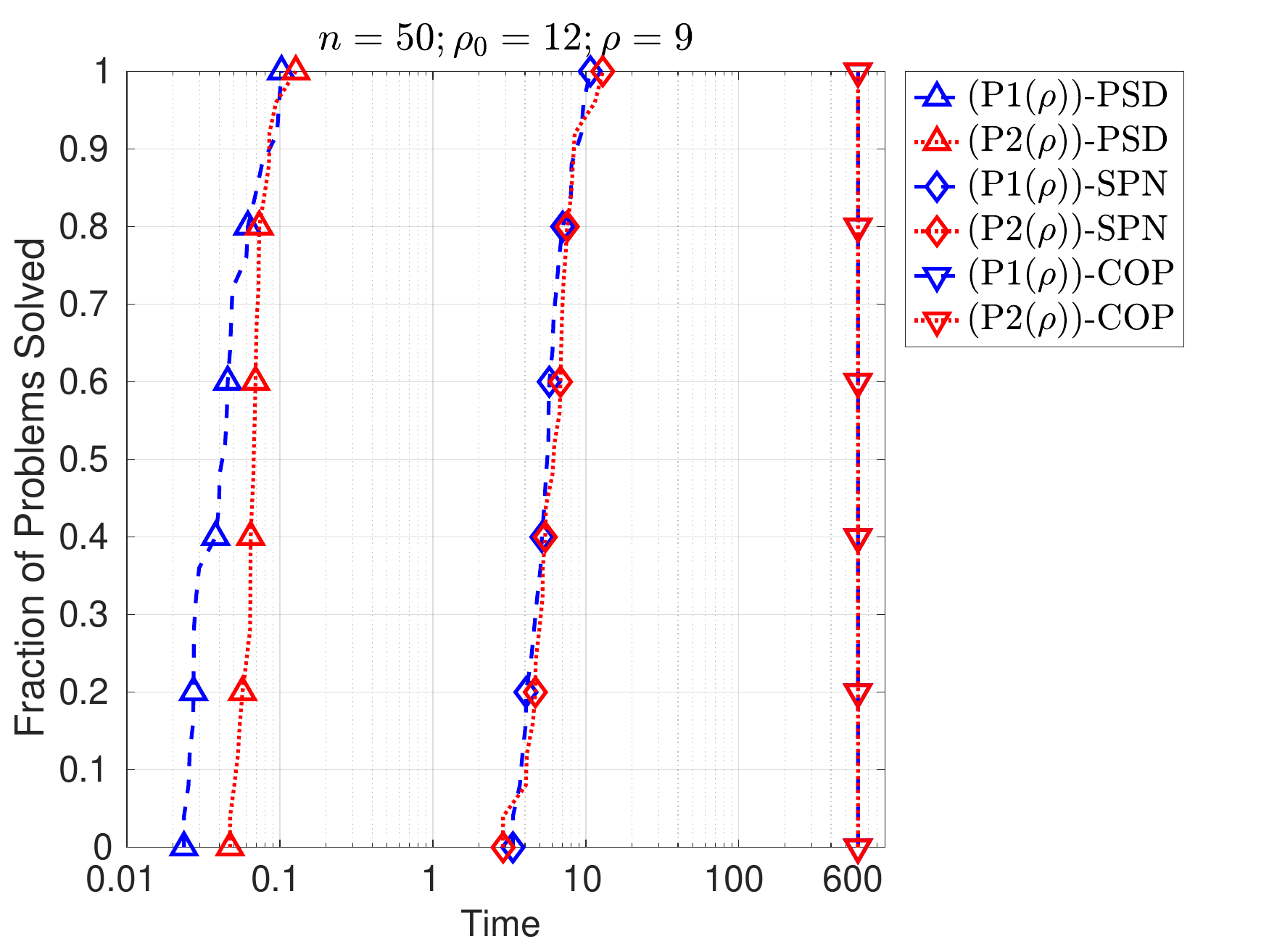}
		\caption{$n = 50, \rho_0 = 12, \rho = 9$}
		\label{fig2g}
	\end{subfigure}
	\begin{subfigure}{0.32\linewidth}
		\includegraphics[width=\linewidth]{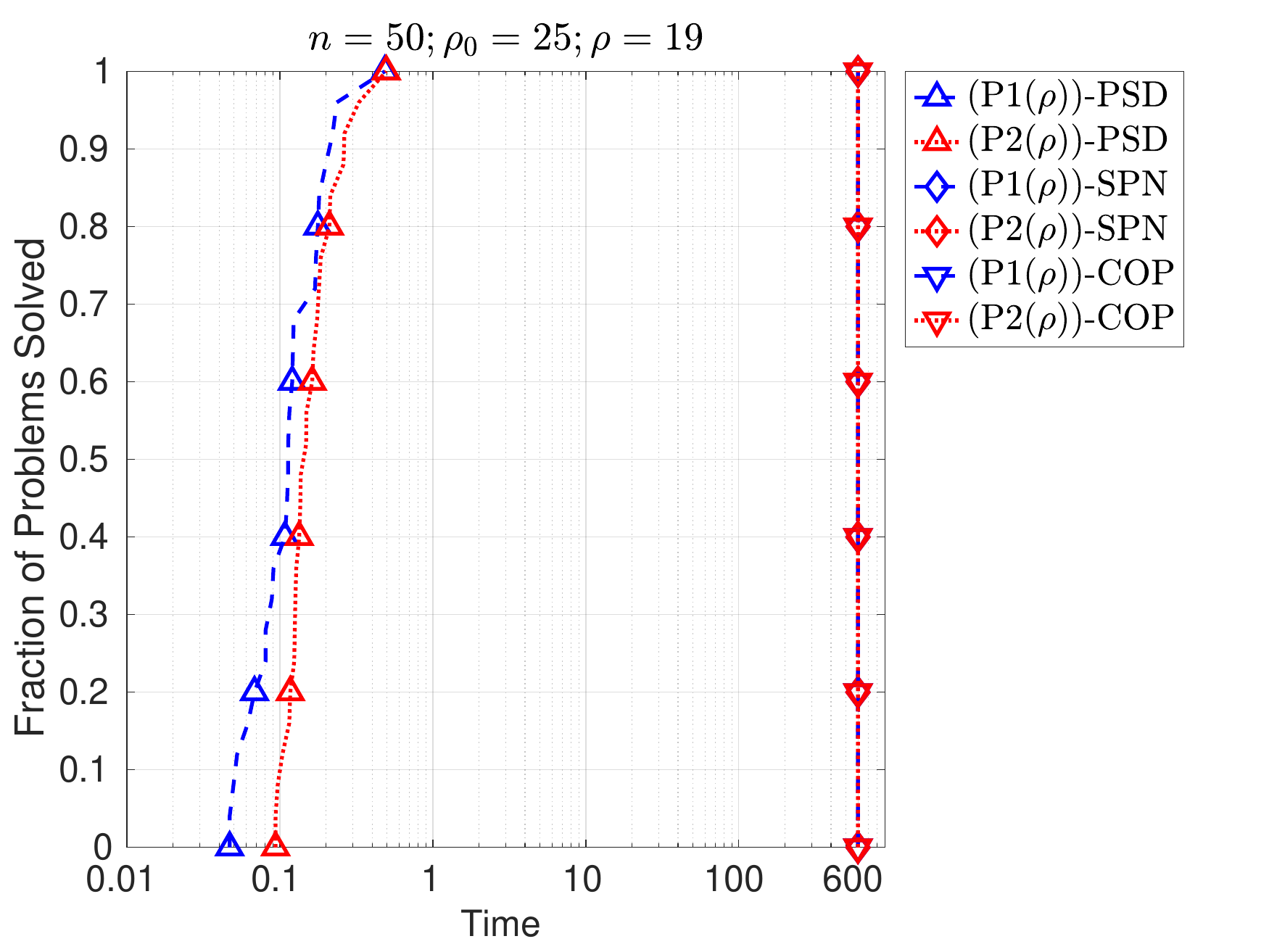}
		\caption{$n = 50, \rho_0 = 25, \rho = 19$}
		\label{fig2h}
	\end{subfigure}
 \begin{subfigure}{0.32\linewidth}
		\includegraphics[width=\linewidth]{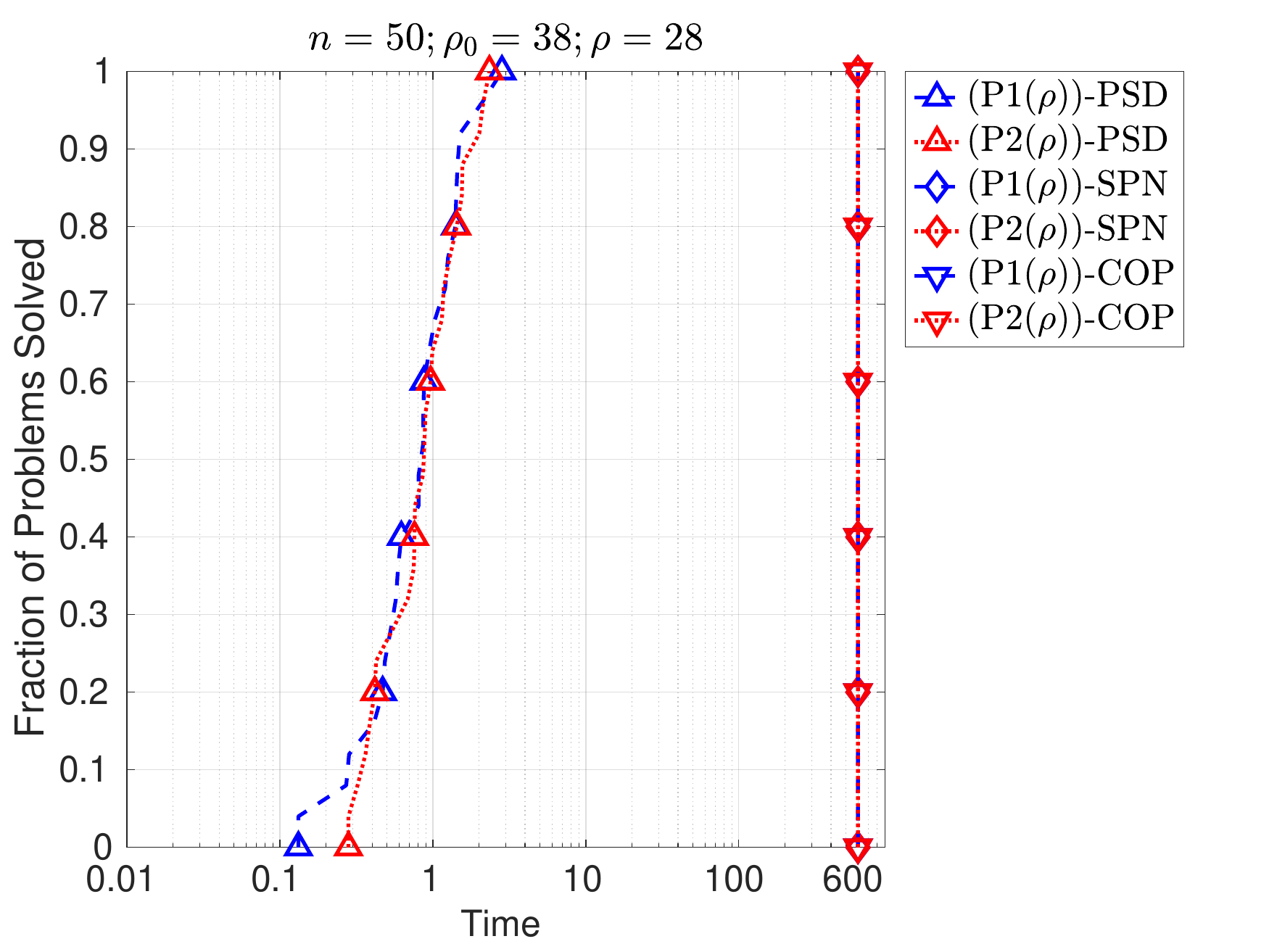}
		\caption{$n = 50, \rho_0 = 38, \rho = 28$}
		\label{fig2i}
	\end{subfigure}
    \caption{Empirical cumulative distribution functions of solution times of \eqref{P1} and \eqref{P2} for all instances with $n = 50$}
    \label{Fig1-n50-P1P2}
\end{sidewaysfigure}

In Figures~\ref{Fig1-n25-P1P2} and~\ref{Fig1-n50-P1P2}, we present the empirical cumulative distribution functions of solution times of \eqref{P1} and \eqref{P2} for all instances with $n = 25$ and $n = 50$, respectively. Each graph in each figure consists of six plots, each of which corresponds to the solution times of \eqref{P1} and \eqref{P2} on each of PSD, SPN, and COP instances for a fixed choice of $(n,\rho_0,\rho)$. In each column of each figure, we fix the tuple $(n,\rho_0)$ and present the empirical cumulative distributions of the solution times of 25 instances corresponding to the three different choices of $\rho$ as presented in Table~\ref{tab1}. On the other hand, each row corresponds to a different ratio of $\rho/\rho_0 \in \{0.25, 0.5, 0.75\}$. Again, we use a logarithmic scale for the solution time and ensure that all axis limits are identical for a meaningful comparison across different graphs. Finally, we note that the markers represent the data points at every $5^{\textrm{th}}$ instance.  

A close examination of Figures~\ref{Fig1-n25-P1P2} and~\ref{Fig1-n50-P1P2} reveal the following observations about the two exact models \eqref{P1} and \eqref{P2}:

\begin{itemize}
    \item[(i)] The two models, in general, exhibit very similar distributions across all parameter choices, except for some small differences for a few choices of $(n,\rho_0,\rho)$ on PSD instances, where \eqref{P1} seems to be solved slightly faster than \eqref{P2}.

    \item[(ii)] We clearly observe an empirical first-order stochastic dominance among PSD, SPN, and COP instances. For each choice of $(n,\rho_0,\rho)$, PSD instances tend to achieve the lowest solution times, followed by SPN instances, whereas COP instances exhibit the highest solution times.  

    \item[(iii)] For each fixed $(n,\rho_0)$, we observe that the solution time tends to decrease as $\rho$ increases for PSD instances, especially for $n = 50$, whereas it increases for each of SPN and COP instances. 
    
    \item[(iv)] If we fix $n$ and the ratio $\rho/\rho_0 \in \{0.25, 0.5, 0.75\}$, the solution time increases as $\rho_0$ increases across all three sets of instances. For fixed $n$, we therefore conclude that the solution time of the exact models seems to be also highly influenced by the choice of the ratio $\rho/\rho_0$. 

    \item[(v)] For $n = 25$, both exact models can be solved very quickly for all PSD instances. In contrast, for SPN and COP instances, each of \eqref{P1} and \eqref{P2} is terminated due to the time limit of 600 seconds on certain subsets of the instances. Therefore, we observe that SPN and COP instances can be particularly challenging for \eqref{P1} and \eqref{P2}, even for $n = 25$ with particular choices of $(\rho_0,\rho)$. For $n = 50$, we observe that even some PSD instances were terminated due to the time limit. Recall that each instance in this set has a convex objective function. For each choice of $(n,\rho_0,\rho)$, a comparison of PSD, SPN, and COP instances reveals that the number of instances terminated due to the time limit exhibits a monotonically increasing behavior. 
\end{itemize}

In summary, the solution time of exact models \eqref{P1} and \eqref{P2} seem to be highly dependent on each of the choices of the instance set, $(n, \rho_0,\rho)$, as well as on the ratio $\rho/\rho_0$. COP instances tend to be the most challenging for the exact models, followed by SPN instances, which, in turn, are followed by PSD instances. These results illustrate that Algorithms~\ref{Alg1} and~\ref{Alg2} are capable of generating instances of~\eqref{sStQP} that are particularly challenging for {\tt Gurobi}, even in relatively small dimensions.

We close this section with a brief discussion of why we choose not to report optimality gaps for \eqref{P1} and \eqref{P2} on instances that were terminated by the time limit. {\tt Gurobi} computes the optimality gap using the formula $|z_P - z_D|/|z_P|$, where $z_P$ and $z_D$ denote the best objective function value and the best lower bound, respectively. In our settings, the optimal value is usually very close to zero, which leads to very large optimality gaps due to the division by a very small number. Therefore, optimality gaps reported by {\tt Gurobi} tend to be extremely large and do not seem to reflect the quality of the solution accurately. We will, however, discuss the quality of the solutions relying directly on the lower bounds in Section~\ref{CRQLB}. 

\subsubsection{Convex DNN relaxations} \label{CRConvex}

Now let us focus on solution times of the four convex relaxations \eqref{D1A}, \eqref{D1B}, \eqref{D2A}, and \eqref{D2B}. 

In contrast with the exact models, Figure~\ref{Fig3-D1D2} illustrates that the solution times of convex relaxations do not exhibit a strong correlation with the choices of the instance set and the tuple $(\rho_0,\rho)$. As such, we do not present the counterparts of Figures~\ref{Fig1-n25-P1P2} and~\ref{Fig1-n50-P1P2} for the convex relaxations due to space considerations. 

Based on Figure~\ref{Fig3-D1D2}, we make the following observations about the convex relaxations \eqref{D1A}, \eqref{D1B}, \eqref{D2A}, and \eqref{D2B}:

\begin{itemize}
    \item[(i)] The distributions of the solution times clearly illustrate the computational advantages of the reduced formulations \eqref{D2A} and \eqref{D2B} over their counterparts. 
    
    \item[(ii)] We observe an empirical first-order stochastic dominance among the distributions of the solution times of \eqref{D1A}, \eqref{D1B}, \eqref{D2A}, and \eqref{D2B} on each of the six graphs. In particular, we observe that the reduced formulation \eqref{D1B} consistently achieves the lowest solution times, followed by the reduced formulation \eqref{D2B}, which, in turn, is followed by \eqref{D2A} and \eqref{D1A}, respectively. In addition to being the theoretically tightest relaxation, it is worth noticing that \eqref{D1B} also outperforms \eqref{D2B} in terms of the solution time.  

    \item[(iii)] For $n = 25$, each of the four relaxations can be solved to optimality within the time limit for each instance set and each choice of $\rho_0$ and $\rho$. For $n = 50$, we observe that \eqref{D1A} was terminated due to the time limit on every instance whereas \eqref{D2A} was terminated due to the time limit on some subsets of the PSD and SPN instances. In contrast, each of the two reduced formulations \eqref{D1B} and \eqref{D2B} was solved to optimality within the time limit on all instances. These observations clearly demonstrate the computational advantages of the reduced formulations \eqref{D2A} and \eqref{D2B} over their counterparts, especially in higher dimensions. 

    \item[(iv)] For PSD instances with $n = 25$, the distributions of the solution times of each of the two exact models exhibit a first-order stochastic dominance on the corresponding distribution of each of the four convex relaxations. In contrast, we observe that the solution times of the convex relaxations achieve better performance on a larger subset of instances for each instance set with a larger $n$. For fixed $n$, a comparison of PSD, SPN, and COP instances reveals an increasingly better performance of the convex relaxations in comparison with the exact models. 
    \end{itemize}

Table~\ref{tab3} reports the number of instances on which \eqref{D2A} is terminated due to the time limit. Note that we omit \eqref{D1A} since all instances with $n = 50$ were terminated due to the time limit. It is worth noticing that \eqref{D2A} was solved to optimality within the time limit on all COP instances with $n = 50$. In contrast with $n = 25$, we therefore conclude that the solution time of \eqref{D2A} for $n = 50$ seems to be somewhat sensitive to the set of instances and the choices of the parameters $(\rho_0,\rho)$. 

\begin{table}[!hbt] 
\begin{center}
\begin{tabular}{ |c|c|c|c||c| } 
\hline
Instance Set & $n$ & $\rho_0$ & $\rho$ & \eqref{D2A} \\
\hline
\hline
\multirow{7}{1em}{PSD} & \multirow{7}{1em}{50} & 
 12 & 9 & 1 \\ 
 \cline{3-5}
 & & 25 & 6 & 14 \\ 
 & & 25 & 12 & 10 \\
 & & 25 & 19 & 3 \\
 \cline{3-5}
 & & 38 & 10 & 8 \\
 & & 38 & 19 & 22 \\
 & & 38 & 28 & 18 \\ 
\hline
\hline
\multirow{7}{1em}{SPN} & \multirow{7}{1em}{50} & 12 & 3 & 1 \\
\cline{3-5}
 & & 25 & 6 & 12 \\
 & & 25 & 12 & 8 \\
 & & 25 & 19 & 2 \\
 \cline{3-5}
 & & 38 & 10 & 20 \\
 & & 38 & 19 & 24 \\
 & & 38 & 28 & 10 \\
 \hline
 \hline
 COP & 50 & -- & -- & -- \\
 \hline
\end{tabular}
\end{center}
\caption{Number of instances (out of 25) terminated due to the time limit (excluding \eqref{D1A} that was terminated due to the time limit on \emph{all} instances with $n = 50$)}
\label{tab3}
\end{table}

Finally, we recall that optimality gaps are not reported on instances terminated due to the time limit since they are not particularly meaningful in our setting. 

In conclusion, the solution time of each relaxation seems to be very robust with respect to the choice of the instance set and the choice of $(\rho_0,\rho)$ in our setting. The reduced formulations \eqref{D2A} and \eqref{D2B} yield significant computational advantages over their counterparts, especially in higher dimensions. 

\subsection{Quality of lower bounds} \label{CRQLB}

In this section, we discuss the quality of the lower bounds arising from the convex relaxations. To that end, we compare the absolute gaps of the exact models with those of the convex relaxations. For an instance, we define the MIQP gap as the difference between the best upper bound and the best lower bound obtained from \eqref{P1} and \eqref{P2}. Similarly, we define the DNN gap to be the difference between the best upper bound obtained from the two exact models, \eqref{P1} and \eqref{P2}, and the best lower bound obtained from the reduced formulations \eqref{D2A} and \eqref{D2B}. Note that we use the lower bounds from the reduced formulations since all of them can be solved to optimality on all instances.

If at least one of \eqref{P1} and \eqref{P2} can solve an instance to optimality, then the MIQP gap will be very close to zero. On such an instance, we deem that our relaxation is \emph{exact} if the DNN gap is of similar magnitude to that of the MIQP gap. If each of \eqref{P1} and \eqref{P2} is terminated due to the limit on an instance, then the MIQP gap will be sufficiently away from zero. In this case, if the DNN gap is smaller than or equal to the MIQP gap, we say that our relaxations are better than the exact modes. Otherwise, our relaxations are \emph{worse} than the exact models. 

In Figure~\ref{Fig4-Gaps}, which is organized similarly to Figure~\ref{Fig3-D1D2}, we plot the DNN and MIQP gaps. To ease reading, we have accumulated all instances of the same dimensions from each instance set across all different parameter constellations $(\rho_0,\rho)$, yielding 225 such instances in total for every graph (a)-(f), depicting the situation in the different hardness classes of the generated instances. In each graph, the horizontal axis represents the instances ordered in nondecreasing MIQP gaps, which are represented by the blue curve, and the vertical axis denotes the absolute gap. Finally, the markers indicate the data points at every $25^{\textrm{th}}$ instance.

We outline our observations:

\begin{itemize}
    \item[(i)] For PSD instances, our relaxations are exact on the vast majority of all instances solved to optimality by an exact model and are better than the exact models on all instances with a positive MIQP gap. 
    
    \item[(ii)] For SPN instances, our relaxations are exact on a relatively smaller proportion of all instances solved to optimality by an exact model and are better than the exact models on all instances with a positive MIQP gap. Furthermore, it is worth noticing that our relaxations continue to be exact on several instances with a positive MIQP gap.

    \item[(iii)] On COP instances, our relaxations are worse than the exact models on the vast majority of instances. However, we remark that the DNN gap is smaller on a small subset of instances with $n = 50$ that admit a positive MIQP gap.
\end{itemize}

The results show clearly that the DNN relaxations provide tight lower bounds on a large number of PSD and SPN instances and that there is a significant number of instances where the MIQP gap is larger than the gaps achieved by the DNN relaxations, sometimes drastically so (one or even two orders of magnitude). The dominance of MIQP models on the extremely difficult COP instances in graphs (e) and (f) may be explained by the fact that additional cuts may be needed to tighten the DNN gaps without resorting to higher-order relaxations of the (intractable) exact conic reformulations.

We close this 
part by reporting that we have not observed a significant difference between the quality of the lower bounds arising from the provably tighter relaxation \eqref{D1B} and the weaker \eqref{D2B}. However, we still recommend using the tighter \eqref{D1B} since it also outperforms \eqref{D2B} in terms of the solution time on our instance sets.



\begin{figure}[!hbt]
    \centering
    \begin{subfigure}{0.495\linewidth}
		\includegraphics[width=\linewidth]{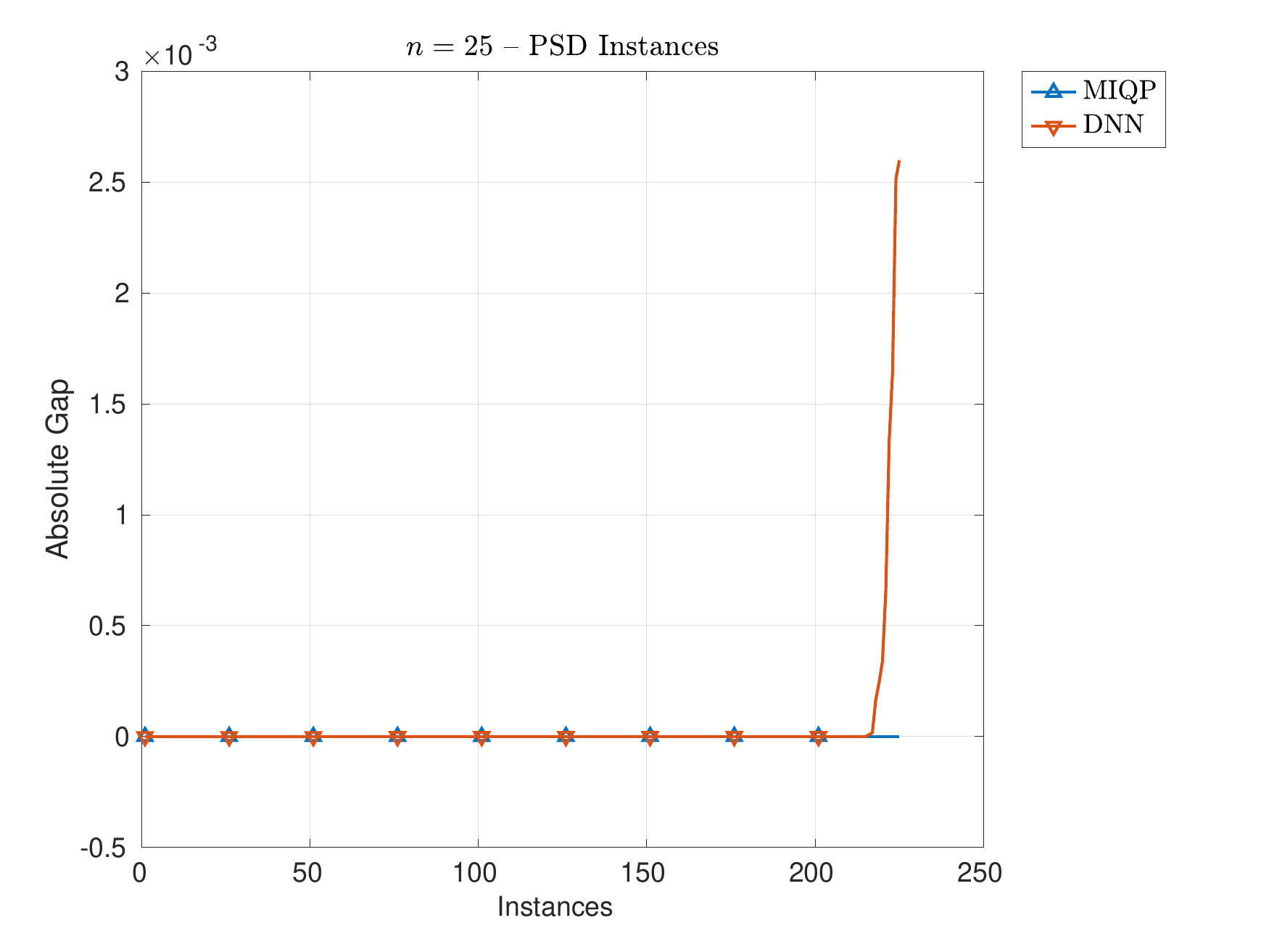}
		\caption{PSD Instances ($n = 25$)}
		\label{fig4a}
	\end{subfigure}
	\begin{subfigure}{0.495\linewidth}
  \includegraphics[width=\linewidth]{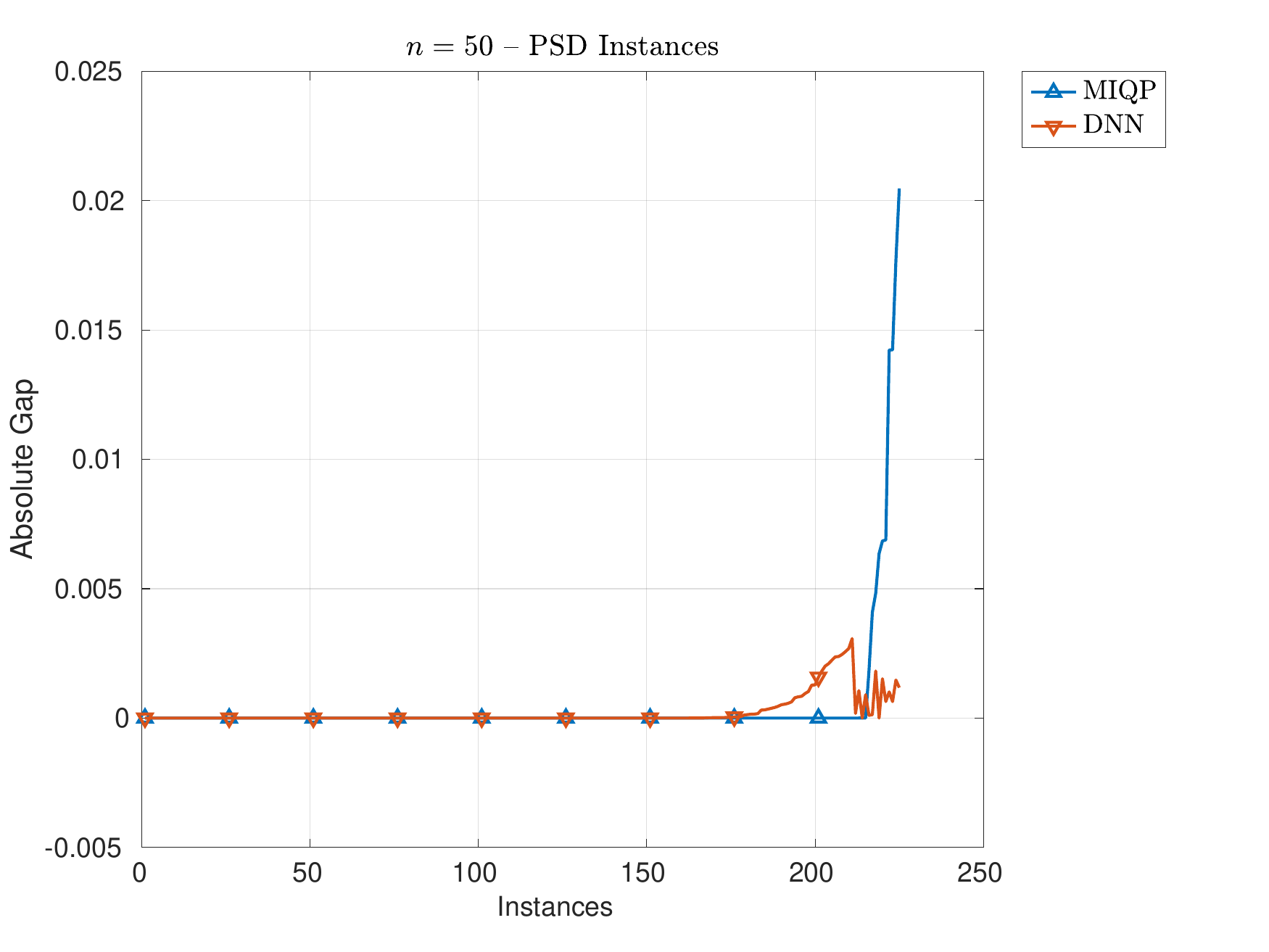}
		\caption{PSD Instances ($n = 50$)}
		\label{fig4b}
	\end{subfigure}
 \\
 \begin{subfigure}{0.495\linewidth}
		\includegraphics[width=\linewidth]{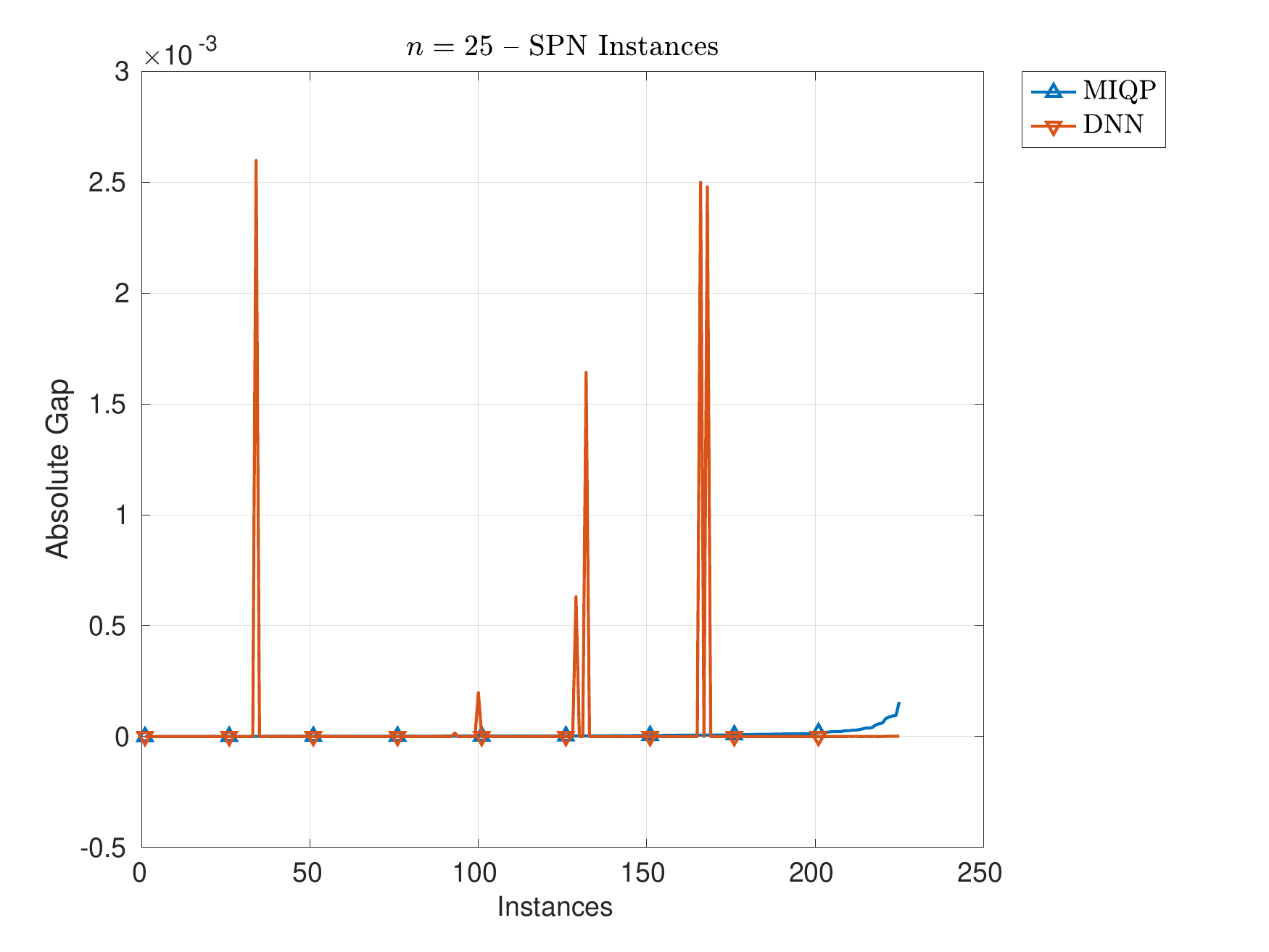}
		\caption{SPN Instances ($n = 25$)}
		\label{fig4c}
	\end{subfigure}
	\begin{subfigure}{0.495\linewidth}
		\includegraphics[width=\linewidth]{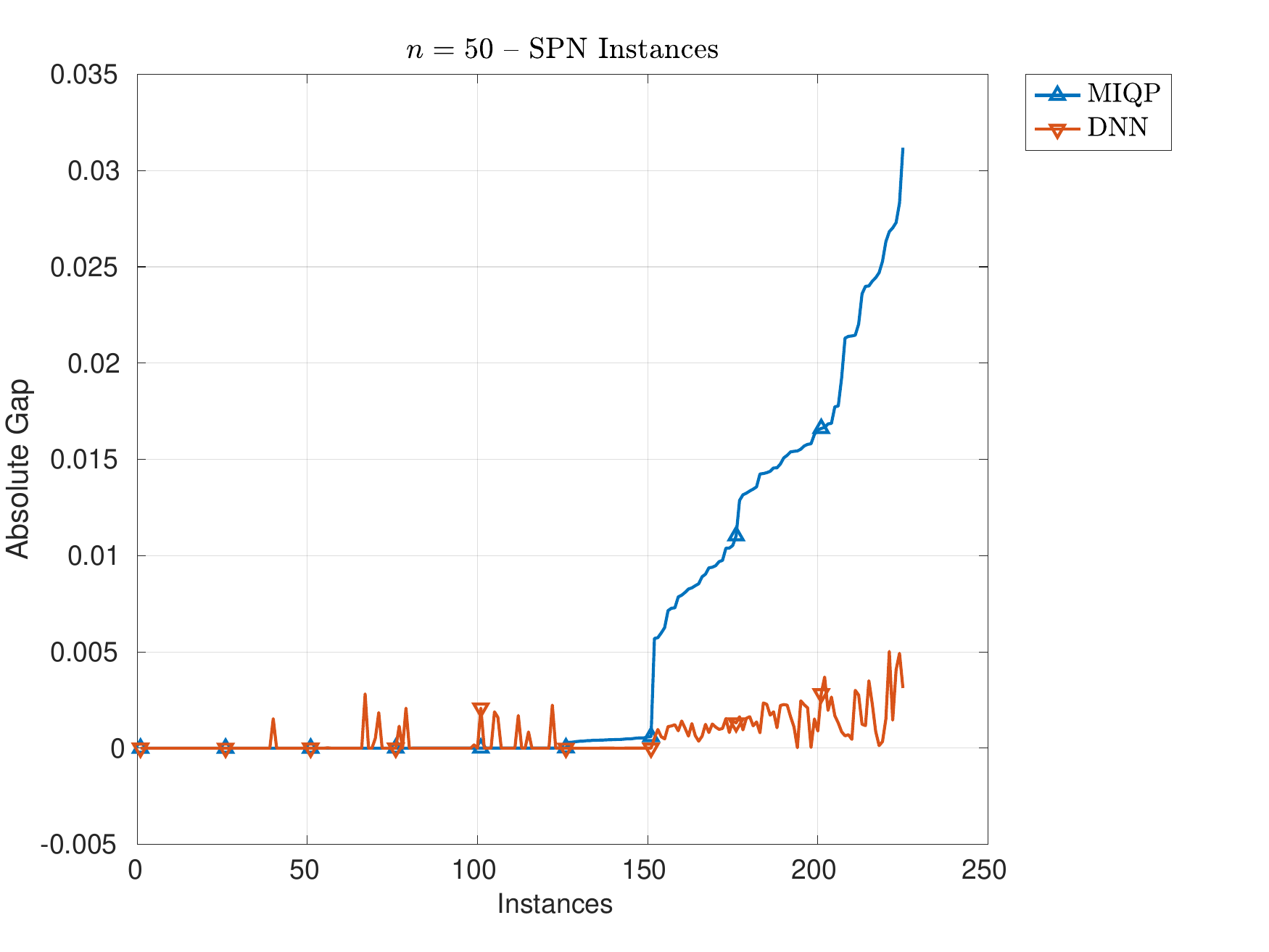}
		\caption{SPN Instances ($n = 50$)}
		\label{fig4d}
	\end{subfigure}
 \\
  \begin{subfigure}{0.495\linewidth}
		\includegraphics[width=\linewidth]{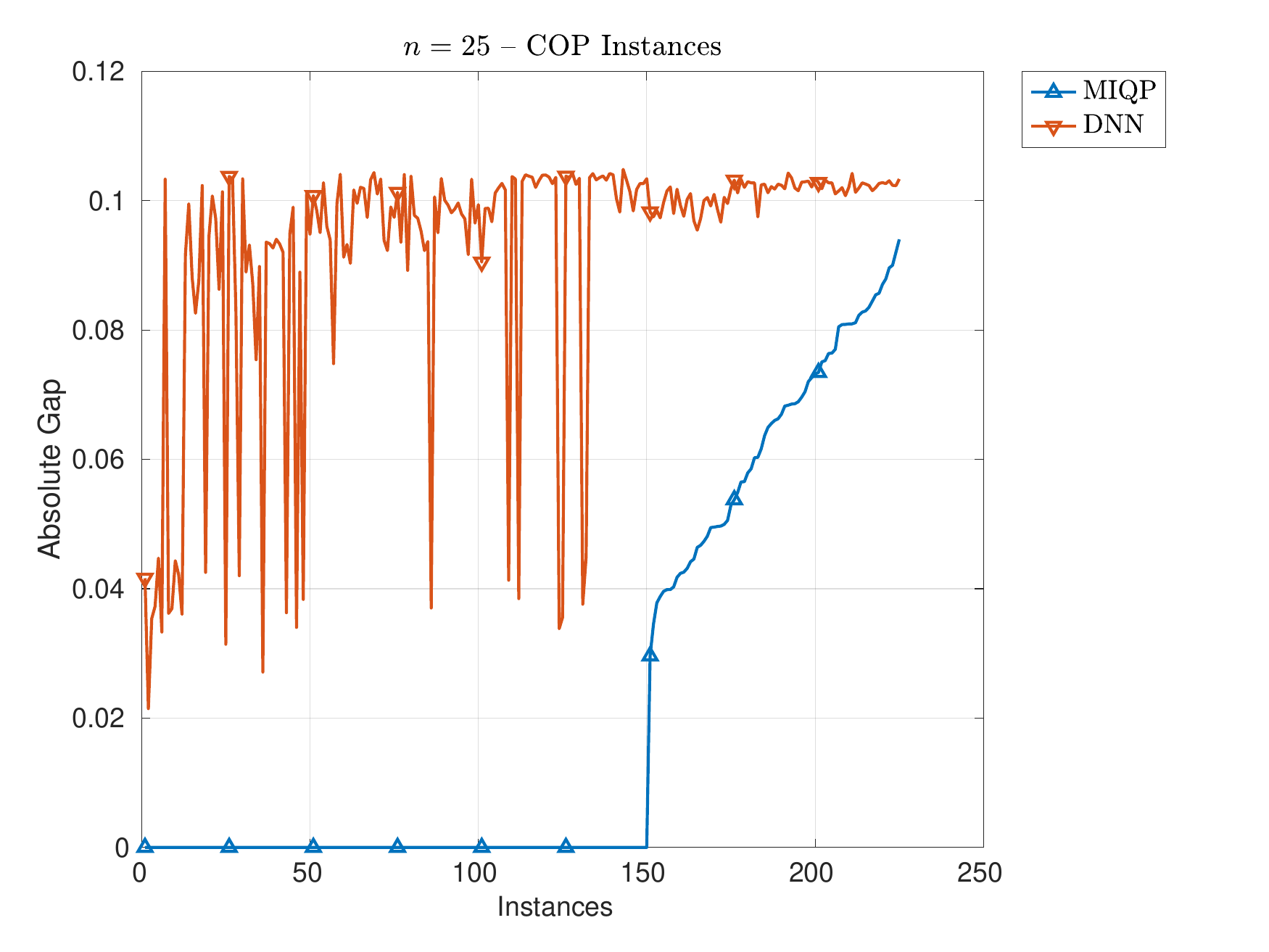}
		\caption{COP Instances ($n = 25$)}
		\label{fig4e}
	\end{subfigure}
	\begin{subfigure}{0.495\linewidth}
		\includegraphics[width=\linewidth]{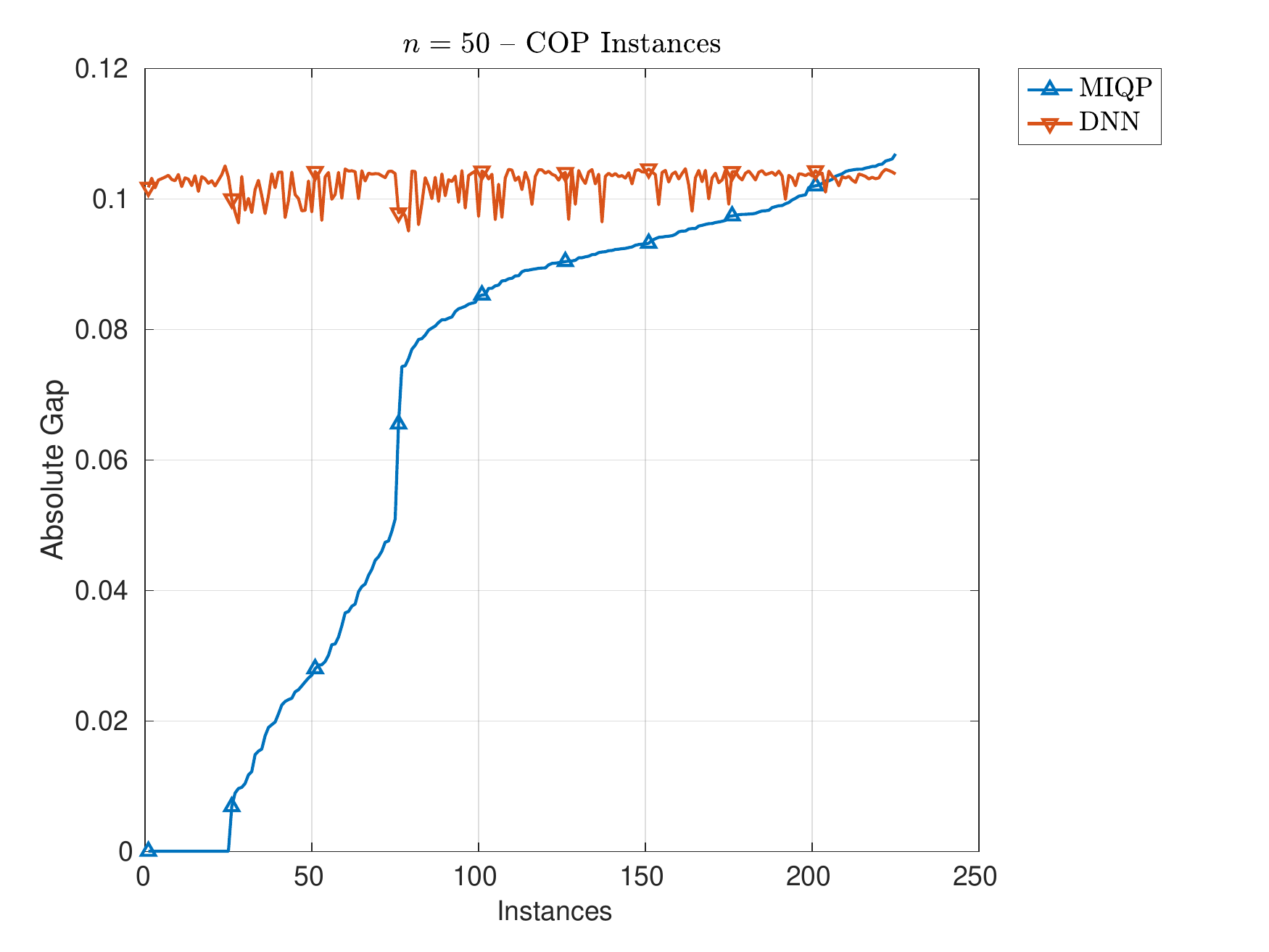}
		\caption{COP Instances ($n = 50$)}
		\label{fig4f}
	\end{subfigure}
    \caption{Absolute gaps of the best of  \eqref{P1} and \eqref{P2} versus the best of \eqref{D1B} and \eqref{D2B}}
    \label{Fig4-Gaps}
\end{figure}

\section{Discussion, conclusions and outlook}\label{concl}

In this paper, we studied the sparse StQP. By utilizing exact copositive reformulations of two mixed integer quadratic optimization models, we proposed two tractable convex relaxations. For both relaxations, we established equivalent tractable reformulations in significantly smaller dimensions. We also presented a theoretical comparison of the two relaxations. Finally, we proposed an instance generation scheme that is guaranteed to construct nontrivial instances. Our computational results clearly illustrate the computational advantages of our reduced relaxations as well as the quality of the relaxation bounds.

In the near future, we aim to focus on the description of the set of instances of \eqref{sStQP} that admit exact relaxations. Motivated by our computational results, we intend to investigate how our relaxations can be strengthened for hard instances such as COP instances.

\section*{Statements and Declarations}
\subsection*{Funding}
Research of Bo Peng supported by the doctoral programme Vienna Graduate School on Computational Optimization, FWF(Austrian Science Funding), Project W1260-N35.
\subsection*{Competing Interests}
The authors declare that they have no conflict of interest.
\subsection*{Author Contributions}
All authors contributed significantly to the study conception and design. The first draft of the manuscript was written with contributions from all authors. The experiments were designed and graphically represented
mostly by E. Alper Y{\i}ld{\i}r{\i}m. All authors commented on all previous versions of the manuscript, and have read and approved the final manuscript.

\subsection*{Data Availability}
The datasets generated and analyzed during the current study are available in the   \href{https://github.com/newfound21/Tractable_relaxations_for_sparse_StQP}{\texttt{github} repository}.

\subsection*{Code Availability}
The codes created during the current study are available in the \href{https://github.com/newfound21/Tractable_relaxations_for_sparse_StQP}{\texttt{github} repository}.

\bibliographystyle{abbrv}
\bibliography{references}

\end{document}